\documentclass[10pt,twoside]{amsart}
\usepackage{amssymb,amsmath,amsthm, amscd, enumerate, mathrsfs}
\usepackage{graphicx, hhline}
\usepackage[all]{xy}
\usepackage[dvipdfmx]{hyperref}

\title{Cone theorem and Mori hyperbolicity}
\author{Osamu Fujino}
\date{2022/12/24, version 0.43}
\subjclass[2010]{Primary 14E30; Secondary 32Q45, 14J45}
\keywords{cone theorem, Mori hyperbolic, 
extremal rational curves, 
quasi-log schemes, adjunction, subadjunction, 
rationally chain connected, 
uniruled}
\address{Department of Mathematics, Graduate School of Science, 
Kyoto University, Kyoto 606-8502, Japan}
\email{fujino@math.kyoto-u.ac.jp}

\DeclareMathOperator{\NE}{\overline{NE}}
\DeclareMathOperator{\Mov}{\overline{Mov}}
\DeclareMathOperator{\codim}{codim}
\DeclareMathOperator{\Spec}{Spec}
\DeclareMathOperator{\Supp}{Supp}
\DeclareMathOperator{\Nqklt}{Nqklt}
\DeclareMathOperator{\Nqlc}{Nqlc}
\DeclareMathOperator{\Nklt}{Nklt}
\DeclareMathOperator{\Nlc}{Nlc}
\DeclareMathOperator{\NLC}{NLC}

\DeclareMathOperator{\Exc}{Exc}
\DeclareMathOperator{\Pic}{Pic}
\DeclareMathOperator{\Div}{Div}
\DeclareMathOperator{\Weil}{Weil}
\DeclareMathOperator{\rank}{rank}
\newtheorem{thm}{Theorem}[section]
\newtheorem{lem}[thm]{Lemma}
\newtheorem{prop}[thm]{Proposition}
\newtheorem{conj}[thm]{Conjecture}
\newtheorem{cor}[thm]{Corollary}
\newtheorem{claim}{Claim}
\newtheorem*{claim-n}{Claim}

\theoremstyle{definition}
\newtheorem{step}{Step}
\newtheorem{case}{Case}
\newtheorem{ex}[thm]{Example}
\newtheorem{defn}[thm]{Definition}
\newtheorem{rem}[thm]{Remark}
\newtheorem*{ack}{Acknowledgments}  

\makeatletter
    
    \@addtoreset{equation}{section}
\makeatother

\begin{document}

\begin{abstract}
We discuss the cone theorem for quasi-log schemes 
and the Mori hyperbolicity. 
In particular, we establish that 
the log canonical divisor of a Mori hyperbolic 
projective normal pair is nef if it is nef 
when restricted to the non-lc locus. 
This answers Svaldi's question completely.  
We also treat the uniruledness of the degenerate 
locus of an extremal contraction morphism for 
quasi-log schemes. Furthermore, we prove 
that every fiber of a relative quasi-log Fano 
scheme is rationally chain connected modulo the non-qlc locus. 
\end{abstract}

\maketitle 

\tableofcontents

\section{Introduction}\label{a-sec1}

This paper gives not only new results around the 
cone theorem and Mori hyperbolicity of quasi-log schemes 
but also a new framework and some techniques to 
treat higher-dimensional complex algebraic varieties 
based on the theory of mixed Hodge structures. 
It also shows that the theory of quasi-log schemes 
is very powerful even for the study of log canonical pairs. 
We note that this paper heavily depends 
on \cite[Chapter 6]{fujino-foundations} and 
\cite{fujino-slc-trivial}. 

In his epoch-making paper \cite{mori}, Shigefumi Mori 
established the following cone theorem for 
smooth projective varieties. 

\begin{thm}[Cone theorem for smooth 
projective varieties]\label{a-thm1.1}
Let $X$ be a smooth projective variety defined over an 
algebraically closed field. 
\begin{itemize}
\item[(i)] There are countably many {\em{(}}possibly singular{\em{)}} 
rational curves 
$C_j\subset X$ such that 
\begin{equation*}
0<-(C_j\cdot K_X)\leq \dim X+1
\end{equation*} and 
\begin{equation*} 
\NE(X)=\NE(X)_{K_X\geq 0}+\sum _j 
\mathbb R_{\geq 0}[C_j].  
\end{equation*}
\item[(ii)] For any $\varepsilon >0$ and 
any ample Cartier divisor $H$ on $X$, 
\begin{equation*}
\NE(X)=\NE(X)_{(K_X+\varepsilon H)
\geq 0}+\sum _{\text{finite}} \mathbb R_{\geq 0}[C_j]. 
\end{equation*} 
\end{itemize}
\end{thm}

In particular, 
we have: 

\begin{thm}\label{a-thm1.2}
Let $X$ be a smooth projective variety defined 
over an algebraically closed field. 
Assume that there are no rational curves on $X$. 
Then $K_X$ is nef. 
\end{thm}

Precisely speaking, Mori proved the existence of 
rational curves on $X$ under the 
assumption that $K_X$ is not nef (see Theorem \ref{a-thm1.2}) 
by his ingenious method of {\em{bend and break}}. 
Then he obtained the above cone theorem 
for smooth projective varieties (see Theorem \ref{a-thm1.1}). 
For the details, see \cite{mori},  
\cite[Sections 1.1, 1.2, and 1.3]{kollar-mori}, 
\cite{debarre}, \cite{kollar-rational}, \cite[Chapter 10]{matsuki}, 
and so on.  

\medskip 

From now on, we will work over $\mathbb C$, the complex number 
field. Our arguments in this paper 
heavily depend on 
Hironaka's resolution of singularities and its generalizations 
and several Kodaira type vanishing 
theorems. Hence they do not work over a field of characteristic 
$p>0$. Let us recall the notion of {\em{Mori hyperbolicity}} following 
\cite{lz} and \cite{svaldi}. 

\begin{defn}[Mori hyperbolicity]\label{a-def1.3} 
Let $(X, \Delta)$ be a normal pair such that $\Delta$ is 
effective. This means that 
$X$ is a normal variety and $\Delta$ 
is an effective $\mathbb R$-divisor on $X$ such that 
$K_X+\Delta$ is $\mathbb R$-Cartier. 
Let $W$ be an lc stratum of $(X, \Delta)$. 
This means that $W$ is an lc center of $(X, \Delta)$ 
or $W$ is $X$ itself. 
We put 
\begin{equation*}
U:=W\setminus \left\{ \left(W\cap \Nlc(X, \Delta)\right)\cup 
\bigcup_{W'}W'\right\}, 
\end{equation*}  
where $W'$ runs over lc centers of $(X, \Delta)$ strictly contained 
in $W$ and $\Nlc(X, \Delta)$ denotes 
the non-lc locus of $(X, \Delta)$, 
and call it the {\em{open lc stratum of 
$(X, \Delta)$ associated to $W$}}. 
We say that $(X, \Delta)$ is {\em{Mori hyperbolic}} 
if there is no non-constant morphism 
\begin{equation*} 
f\colon \mathbb A^1\longrightarrow U
\end{equation*}  
for any open lc stratum $U$ of $(X, \Delta)$. 
\end{defn}

The following theorem is a generalization of 
Theorem \ref{a-thm1.2} for normal pairs and 
is an answer to \cite[Question 6.6]{svaldi}. 

\begin{thm}\label{a-thm1.4}
Let $X$ be a normal projective variety and let $\Delta$ be 
an effective $\mathbb R$-divisor on $X$ such that 
$K_X+\Delta$ is $\mathbb R$-Cartier. 
Assume that $(X, \Delta)$ is Mori hyperbolic and that 
$K_X+\Delta$ is nef when restricted to 
$\Nlc(X, \Delta)$. 
Then $K_X+\Delta$ is nef. 
\end{thm}

Theorem \ref{a-thm1.4} follows from the following cone 
theorem for normal pairs. We can 
see it as a generalization of Theorem \ref{a-thm1.1} 
for normal pairs. 

\begin{thm}[Cone theorem for normal pairs]\label{a-thm1.5}
Let $(X, \Delta)$ be a normal pair such that $\Delta$ is 
effective and let 
$\pi\colon X\to S$ be a projective morphism 
between schemes. 
\begin{itemize}
\item[(i)] Then 
\begin{equation*} 
\NE (X/S)=\NE(X/S)_{(K_X+\Delta)\geq 0} 
+\NE(X/S)_{-\infty}+\sum _j R_j
\end{equation*}  
holds, where $R_j$'s are the $(K_X+\Delta)$-negative extremal rays of 
$\NE(X/S)$ that are 
rational and relatively ample at infinity. 
In particular, each $R_j$ is spanned by an integral 
curve $C_j$ on $X$ such that 
$\pi(C_j)$ is a point. 
Note that an extremal ray $R$ of $\NE(X/S)$ is 
rational and relatively ample at infinitely if and 
only if there exists a $\pi$-nef $\mathbb Q$-line bundle 
$\mathcal L$ on $X$ such that 
$R=\NE(X/S)\cap \mathcal L^{\perp}$ and 
that $\mathcal L|_{\Nlc(X, \Delta)}$ is 
$\pi|_{\Nlc(X, \Delta)}$-ample. 
\item[(ii)] Let $H$ be a $\pi$-ample $\mathbb R$-divisor 
on $X$. 
Then 
\begin{equation*} 
\NE(X/S)=\NE(X/S)_{(K_X+\Delta+H)\geq 0} 
+\NE(X/S)_{-\infty}+\sum _{\text{finite}} R_j
\end{equation*} 
holds. 
\item[(iii)] For each $(K_X+\Delta)$-negative 
extremal ray $R_j$ of $\NE(X/S)$ that 
is rational and relatively ample at infinity, 
there are an open lc stratum $U_j$ of $(X, \Delta)$ and a non-constant  
morphism 
\begin{equation*} 
f_j\colon \mathbb A^1\longrightarrow U_j
\end{equation*} 
such that 
$C_j$, the closure of $f_j(\mathbb A^1)$ in $X$, 
spans $R_j$ in $N_1(X/S)$ with 
\begin{equation*} 
0<-(K_X+\Delta)\cdot C_j\leq 2\dim U_j. 
\end{equation*} 
\end{itemize}
\end{thm}

More generally, we establish the following cone theorem 
for quasi-log schemes. We note that Theorem \ref{a-thm1.5} is a 
very special case of Theorem \ref{a-thm1.6}. 
Since the precise definition of quasi-log schemes may look 
technical and artificial, we omit it here. 
For the details, see Section \ref{d-sec4} below. 
Here, we only explain a typical example of quasi-log schemes. 
Let $(V, \Delta)$ be a log canonical pair which is not 
kawamata log terminal. Then the non-klt locus 
$X:=\Nklt(V, \Delta)$ of $(V, \Delta)$ with $\omega:=
(K_V+\Delta)|_X$ naturally has a quasi-log scheme structure. 
In this case, the non-qlc locus $X_{-\infty}=\Nqlc(X, \omega)$ 
of $[X, \omega]$ is empty and $W$ is a qlc stratum of $[X, \omega]$ 
if and only if $W$ is an lc center of $(V, \Delta)$. 
We can define open qlc stratum of $[X, \omega]$ 
similarly to Definition \ref{a-def1.3}. 
In general, $X$ is reducible and is not equidimensional. 

\begin{thm}[Cone theorem for quasi-log schemes]\label{a-thm1.6}
Let $[X, \omega]$ be a quasi-log scheme and let 
$\pi\colon X\to S$ be a projective morphism 
between schemes. 
\begin{itemize}
\item[(i)] Then 
\begin{equation*} 
\NE(X/S)=\NE(X/S)_{\omega\geq 0} 
+\NE(X/S)_{-\infty}+\sum _j R_j
\end{equation*}  
holds, where $R_j$'s are the $\omega$-negative extremal rays of 
$\NE(X/S)$ that are 
rational and relatively ample at infinity. 
In particular, each $R_j$ is spanned by an integral 
curve $C_j$ on $X$ such that 
$\pi(C_j)$ is a point. 
Note that an extremal ray $R$ of $\NE(X/S)$ is 
rational and relatively ample at infinitely if and 
only if there exists a $\pi$-nef $\mathbb Q$-line bundle 
$\mathcal L$ on $X$ such that 
$R=\NE(X/S)\cap \mathcal L^{\perp}$ and 
that $\mathcal L|_{\Nqlc(X, \omega)}$ is 
$\pi|_{\Nqlc(X, \omega)}$-ample. 
\item[(ii)] Let $H$ be a $\pi$-ample $\mathbb R$-line bundle 
on $X$. 
Then 
\begin{equation*} 
\NE(X/S)=\NE(X/S)_{(\omega+H)\geq 0} 
+\NE(X/S)_{-\infty}+\sum _{\text{finite}} R_j 
\end{equation*} 
holds. 
\item[(iii)] For each $\omega$-negative 
extremal ray $R_j$ of $\NE(X/S)$ that 
is rational and relatively ample at infinity, 
there are an open qlc stratum $U_j$ of $[X, \omega]$ and a non-constant  
morphism 
\begin{equation*} 
f_j\colon \mathbb A^1\longrightarrow U_j
\end{equation*} 
such that 
$C_j$, the closure of $f_j(\mathbb A^1)$ in $X$, 
spans $R_j$ in $N_1(X/S)$ with 
\begin{equation*} 
0<-\omega\cdot C_j\leq 2\dim U_j.
\end{equation*}  
\end{itemize}
\end{thm}

We make a remark on $U_j$ in Theorem \ref{a-thm1.6}. 

\begin{rem}\label{a-rem1.7} 
In Theorem \ref{a-thm1.6} (iii), 
let $\varphi_{R_j}$ be the extremal contraction morphism associated to 
$R_j$. 
Then the proof of 
Theorem \ref{a-thm1.6} shows that $U_j$ is any open 
qlc stratum of $[X, \omega]$ such that 
$\varphi_{R_j}\colon  \overline {U_j}\to \varphi_{R_j}(\overline {U_j})$ is not 
finite and that $\varphi_{R_j}\colon  W^\dag\to \varphi_{R_j}(W^\dag)$ 
is finite 
for every qlc center $W^\dag$ of $[X, \omega]$ with 
$W^\dag\subsetneq \overline{U_j}$, where 
$\overline{U_j}$ is the closure of $U_j$ in $X$. 
\end{rem}

The main ingredients of the proof of Theorem \ref{a-thm1.6} are 
the following three results. 

\begin{thm}\label{a-thm1.8}
Let $X$ be a normal variety and let $\Delta$ 
be an effective $\mathbb R$-divisor 
on $X$ such that $K_X+\Delta$ is $\mathbb R$-Cartier. 
Let $\pi\colon X\to S$ be a projective morphism 
onto a scheme $S$. 
Assume that $(K_X+\Delta)|_{\Nklt(X, \Delta)}$ is nef over $S$, 
where 
$\Nklt(X, \Delta)$ denotes the non-klt locus of 
$(X, \Delta)$, 
and that $K_X+\Delta$ is not nef over $S$. 
Then there exists a non-constant morphism 
\begin{equation*} 
f\colon \mathbb A^1\longrightarrow 
X\setminus \Nklt(X, \Delta)
\end{equation*} 
such that 
$\pi\circ f(\mathbb A^1)$ is a point and that 
the curve $C$, the closure of $f(\mathbb A^1)$ in 
$X$, is a {\em{(}}possibly singular{\em{)}} rational curve 
with 
\begin{equation*} 
0<-(K_X+\Delta)\cdot C\leq 2\dim X. 
\end{equation*}  
\end{thm}

We prove Theorem \ref{a-thm1.8} with the aid of the minimal model 
theory for higher-dimensional algebraic varieties mainly due to \cite{bchm}. 
Theorem \ref{a-thm1.9} is a slight generalization of 
\cite[Theorem 1.1]{fujino-haidong}, where $[X, \omega]$ is a quasi-log canonical 
pair. 
In Theorem \ref{a-thm1.9}, $[X, \omega]$ is not necessarily 
quasi-log canonical. 

\begin{thm}\label{a-thm1.9}
Let $[X, \omega]$ be a quasi-log scheme such that 
$X$ is irreducible. Let $\nu\colon Z\to X$ be the normalization. 
Then there exists a proper surjective 
morphism $f'\colon (Y', B_{Y'})\to Z$ from a 
quasi-projective globally embedded simple normal 
crossing pair $(Y', B_{Y'})$ such that 
every stratum of $Y'$ is dominant onto $Z$ and that 
\begin{equation*} 
\left(Z, \nu^*\omega, f'\colon  (Y', B_{Y'})\to Z\right)
\end{equation*} 
naturally becomes a quasi-log scheme 
with $\Nqklt(Z, \nu^*\omega)=\nu^{-1}\Nqklt(X, \omega)$. 
More precisely, the following equality 
\begin{equation*} 
\nu_*\mathcal I_{\Nqklt(Z, \nu^*\omega)}=
\mathcal I_{\Nqklt(X, \omega)}
\end{equation*} 
holds, 
where $\mathcal I_{\Nqklt(X, \omega)}$ and $\mathcal I_{\Nqklt(Z, 
\nu^*\omega)}$ are the defining ideal sheaves 
of $\Nqklt(X, \omega)$ and 
$\Nqklt(Z, \nu^*\omega)$ respectively. 
\end{thm}

Theorem \ref{a-thm1.10} is similar to \cite[Theorem 1.1]{fujino-subadjunction}. 
The proof of Theorem \ref{a-thm1.10} needs some deep results on 
basic slc-trivial fibrations obtained in \cite{fujino-slc-trivial} and 
\cite{fujino-fujisawa-liu}. 
Therefore, Theorem \ref{a-thm1.10} 
depends on the theory of variations of mixed Hodge structure 
(see \cite{fujino-fujisawa} and \cite{fujino-fujisawa-saito}). 

\begin{thm}\label{a-thm1.10}
Let $[X, \omega]$ be a quasi-log scheme such that 
$X$ is a normal quasi-projective variety. 
Let 
$H$ be an ample $\mathbb R$-divisor on $X$. 
Then there exists an effective $\mathbb R$-divisor 
$\Delta$ on $X$ such that 
\begin{equation*} 
K_X+\Delta\sim 
_{\mathbb R}\omega+H
\end{equation*} 
and 
that 
\begin{equation*}
\Nklt(X, \Delta)=\Nqklt(X, \omega) 
\end{equation*} 
holds set theoretically, 
where $\Nklt(X, \Delta)$ denotes 
the non-klt locus of $(X, \Delta)$. 
Furthermore, if $[X, \omega]$ has a $\mathbb Q$-structure 
and $H$ is an ample $\mathbb Q$-divisor on $X$, 
then we can make $\Delta$ a $\mathbb Q$-divisor on $X$ such that  
\begin{equation*} 
K_X+\Delta\sim _{\mathbb Q} \omega+H
\end{equation*} 
holds. 

When $X$ is a smooth curve, 
we can take an effective $\mathbb R$-divisor $\Delta$ on 
$X$ such that 
\begin{equation*}
K_X+\Delta\sim _{\mathbb R}\omega
\end{equation*} 
and that 
\begin{equation*} 
\Nklt(X, \Delta)=\Nqklt(X, \omega)
\end{equation*} 
holds 
set theoretically. 
Of course, if we further assume that $[X, \omega]$ has a 
$\mathbb Q$-structure, then we can make 
$\Delta$ an effective $\mathbb Q$-divisor 
on $X$ such that 
\begin{equation*}
K_X+\Delta\sim _{\mathbb Q}\omega
\end{equation*}  
holds. 
\end{thm}

Let us briefly explain the idea of the proof of Theorem \ref{a-thm1.6} (iii), 
which is one of the main results of this paper. 
We take an $\omega$-negative extremal ray $R_j$ of $\NE(X/S)$ 
that are rational and relatively ample at infinity. 
Then, by the contraction theorem, 
there exists a contraction morphism 
$\varphi:=\varphi_{R_j}\colon  X\to Y$ over $S$ associated to $R_j$. 
We take a qlc stratum $W$ of $[X, \omega]$ such that 
$\varphi\colon W\to \varphi(W)$ is not finite and that 
$\varphi\colon W^\dag \to \varphi(W^\dag)$ is finite for every 
qlc center $W^\dag$ with $W^\dag\subsetneq W$. 
By adjunction, $W':=W\cup \Nqlc(X, \omega)$ 
with $\omega|_{W'}$ becomes a quasi-log scheme. 
Hence we can replace $[X, \omega]$ with $[W', \omega|_{W'}]$. 
By using Theorem \ref{a-thm1.9}, 
we can reduce the problem to the case where $X$ is 
a normal variety. 
By Theorem \ref{a-thm1.10}, 
we see that it is sufficient to treat normal pairs. 
For normal pairs, by Theorem \ref{a-thm1.8}, 
we can find a non-constant morphism 
\begin{equation*}
f_j\colon  \mathbb A^1\longrightarrow X\setminus 
\Nqklt(X, \omega)
\end{equation*}  
with the desired properties. 

\medskip 

We also treat an ampleness criterion for Mori hyperbolic 
normal pairs. It is a generalization of \cite[Theorem 7.5]{svaldi}. 

\begin{thm}[Ampleness criterion for Mori hyperbolic 
normal pairs]\label{a-thm1.11}
Let $X$ be a normal projective variety and 
let $\Delta$ be an effective $\mathbb R$-divisor 
on $X$ such that 
$K_X+\Delta$ is $\mathbb R$-Cartier. 
Assume that $(X, \Delta)$ is Mori hyperbolic, 
$(K_X+\Delta)|_{\Nlc(X, \Delta)}$ is ample, and 
$K_X+\Delta$ is log big with respect to 
$(X, \Delta)$, 
that is, $(K_X+\Delta)|_W$ is big for every 
lc stratum $W$ of $(X, \Delta)$. 
Then $K_X+\Delta$ is ample. 
\end{thm}

Theorem \ref{a-thm1.11} is a very special case of 
the ampleness criterion for quasi-log schemes 
(see Theorem \ref{k-thm11.1}). 
We omit the precise statement of Theorem \ref{k-thm11.1} 
here since it looks technical. 
We note that $K_X+\Delta$ is nef by Theorem \ref{a-thm1.4} 
since $(X, \Delta)$ is Mori hyperbolic and 
$(K_X+\Delta)|_{\Nlc(X, \Delta)}$ is ample. 
Therefore, $K_X+\Delta$ is nef and log big with respect 
to $(X, \Delta)$ in 
Theorem \ref{a-thm1.11}. 
Hence we can see that $K_X+\Delta$ is semi-ample with 
the aid of the basepoint-free theorem of Reid--Fukuda type 
(see \cite{fujino-reid-fukuda}). Then we prove that $K_X+\Delta$ is ample. 

By using the method established for the proof of 
Theorem \ref{a-thm1.6}, we can prove the 
following theorems. Note that 
Theorems \ref{a-thm1.12}, \ref{a-thm1.13}, 
and \ref{a-thm1.14} are free from the 
theory of minimal models. Theorem \ref{a-thm1.12} is 
a generalization of Kawamata's famous theorem 
(see \cite{kawamata1}). 

\begin{thm}\label{a-thm1.12} 
Let $[X, \omega]$ be a quasi-log scheme and let 
$\varphi\colon X\to W$ be a projective morphism 
between schemes such that $-\omega$ is $\varphi$-ample. 
Let $P$ be an arbitrary closed point of $W$. 
Let $E$ be any positive-dimensional irreducible component of 
$\varphi^{-1}(P)$ such that $E\not\subset X_{-\infty}$. 
Then $E$ is covered by {\em{(}}possibly singular{\em{)}} 
rational curves $\ell$ with 
\begin{equation*} 
0<-\omega\cdot \ell \leq 2\dim E. 
\end{equation*}  
In particular, $E$ is uniruled. 
\end{thm}

For the reader's convenience, 
let us explain the idea of the proof of Theorem \ref{a-thm1.12}. 
We take an effective $\mathbb R$-Cartier divisor 
$B$ on $W$ passing through $P$ such that 
$E$ is a qlc stratum of $[X, \omega+\varphi^*B]$. 
Let $\nu\colon \overline E\to E$ be the normalization. 
By adjunction for quasi-log schemes, 
Theorems \ref{a-thm1.9}, \ref{a-thm1.10}, and so on, 
for any ample $\mathbb R$-divisor $H$ on $\overline E$, 
we obtain an effective $\mathbb R$-divisor 
$\Delta_{\overline E, H}$ on $\overline E$ 
such that 
\begin{equation*} 
\nu^*\omega+H\sim _{\mathbb R} K_{\overline E} +\Delta_{\overline E, H}
\end{equation*} 
holds. 
This 
implies that $C\cdot K_{\overline E}<0$ holds for any general 
curve $C$ on $\overline E$. 
Thus, it is not difficult to see that $\overline E$ is covered 
by rational curves (see \cite{miyaoka-mori}). 
Our approach is different from Kawamata's original one, which 
uses a relative Kawamata--Viehweg 
vanishing theorem for projective bimeromorphic 
morphisms between complex analytic spaces. 
Kawamata's approach does not work for our setting.  

\medskip 

As a direct consequence of Theorem \ref{a-thm1.12}, 
we have: 

\begin{thm}[Lengths of extremal rational curves]\label{a-thm1.13}
Let $[X, \omega]$ be a quasi-log scheme and let 
$\pi\colon X\to S$ be a projective morphism 
between schemes. 
Let $R$ be an $\omega$-negative extremal ray of 
$\NE(X/S)$ that are 
rational and relatively ample at infinity. 
Let $\varphi_R\colon X\to W$ be the contraction morphism 
over $S$ associated to $R$. 
We put 
\begin{equation*} 
d=\min _E \dim E, 
\end{equation*}  
where $E$ runs over positive-dimensional 
irreducible components of $\varphi^{-1}_R(P)$ for all $P\in W$. 
Then $R$ is spanned by a {\em{(}}possibly singular{\em{)}} rational 
curve $\ell$ with 
\begin{equation*} 
0<-\omega\cdot \ell\leq 2d. 
\end{equation*} 
\end{thm}

If $(X, \Delta)$ is a log canonical pair, then 
$[X, K_X+\Delta]$ naturally becomes 
a quasi-log canonical pair. Hence we can apply Theorems \ref{a-thm1.12} 
and \ref{a-thm1.13} to log canonical pairs. 
Note that Theorems \ref{a-thm1.12} and \ref{a-thm1.13} 
are new even for log canonical pairs (see 
also Corollary \ref{l-cor12.3}). 
We can prove the following result 
on rational chain connectedness 
for relative quasi-log Fano schemes. 

\begin{thm}[Rational chain connectedness]\label{a-thm1.14}
Let $[X, \omega]$ be a quasi-log scheme and let $\pi\colon X\to S$ be a 
projective morphism between schemes with $\pi_*\mathcal O_X\simeq 
\mathcal O_S$. 
Assume that $-\omega$ is ample over $S$. Then 
$\pi^{-1}(P)$ is rationally chain connected modulo 
$\pi^{-1}(P)\cap X_{-\infty}$ for every closed point $P\in S$. 
In particular, if further $\pi^{-1}(P)\cap X_{-\infty}=\emptyset$ holds, 
that is, $[X, \omega]$ is quasi-log canonical in a neighborhood of 
$\pi^{-1}(P)$, 
then $\pi^{-1}(P)$ is rationally chain connected. 
\end{thm}

Let us see the idea of the proof of Theorem \ref{a-thm1.14}. 
We assume that $\pi^{-1}(P)\cap X_{-\infty}\ne \emptyset$ for simplicity. 
By using the framework of quasi-log schemes, 
we construct a good finite increasing sequence of 
closed subschemes 
\begin{equation*} 
Z_{-1}:=\Nqlc(X, \omega)\subset Z_0\subsetneq Z_1
\subsetneq \cdots \subsetneq Z_k
\end{equation*} 
of $X$ such that $\pi^{-1}(P)\subset Z_k$ after shrinking 
$X$ around $\pi^{-1}(P)$. 
It is well known that if $(V, \Delta)$ is a projective normal 
pair such that $\Delta$ is effective and that 
$-(K_V+\Delta)$ is ample then 
$V$ is rationally chain connected modulo 
$\Nklt(V, \Delta)$ (see \cite{hacon-mckernan} 
and \cite{bp}). 
By this fact, adjunction for quasi-log schemes, 
Theorems \ref{a-thm1.9}, \ref{a-thm1.10}, and so on, 
we prove that 
$Z_{i+1}\cap \pi^{-1}(P)$ is rationally chain connected 
modulo $Z_i\cap \pi^{-1}(P)$ for every $-1\leq i\leq k-1$. 
Since $Z_k\cap \pi^{-1}(P)=\pi^{-1}(P)$ and 
$Z_{-1}\cap \pi^{-1}(P)=\pi^{-1}(P)\cap X_{-\infty}$, 
we obtain that $\pi^{-1}(P)$ 
is rationally chain connected modulo $\pi^{-1}(P)\cap X_{-\infty}$. 

\medskip 

Theorems \ref{a-thm1.6}, \ref{a-thm1.12}, and 
\ref{a-thm1.14} are closely related one another. 
Let us see these theorems for extremal birational contraction 
morphisms of log canonical pairs. 
Let $(X, \Delta)$ be a projective log canonical pair 
and let $R$ be a $(K_X+\Delta)$-negative extremal ray of 
$\NE(X)$. 
Assume that the contraction morphism 
$\varphi_R\colon X\to W$ associated to $R$ is birational. 
We take a closed point $P$ of $W$ such that 
$\dim \varphi^{-1}_R(P)>0$. 
Then Theorem \ref{a-thm1.14} says that  
$\varphi^{-1}_R(P)$ is rationally chain connected. 
However, Theorem \ref{a-thm1.14} gives no information on 
degrees of rational curves on $\varphi^{-1}_R(P)$ with 
respect to $-(K_X+\Delta)$. 
On the other hand, Theorem \ref{a-thm1.12} shows that 
every irreducible component of $\varphi^{-1}_R(P)$ is covered 
by rational curves $\ell$ with $
0<-(K_X+\Delta)\cdot \ell \leq 2\dim \varphi^{-1}_R(P)$.  
In particular, every irreducible component of 
the exceptional locus of $\varphi_R$ is uniruled. 
Note that the rational chain connectedness of $\varphi^{-1}(P)$ does 
not directly follow from Theorem \ref{a-thm1.12}. 
Theorem \ref{a-thm1.6} (see also 
Theorem \ref{a-thm1.5}) 
shows that there exist a rational curve $C$ on $X$ 
and an open lc stratum $U$ of 
$(X, \Delta)$ such that $\varphi_R(C)$ is a point and that 
the normalization of $C\cap U$ contains $\mathbb A^1$. 

\medskip 

We pose a conjecture related to \cite[Theorem 3.1]{lz}. 

\begin{conj}\label{a-conj1.15} 
Let $[X, \omega]$ be a quasi-log scheme 
and let $\pi\colon  X\to S$ be a projective 
morphism between schemes such that 
$-\omega$ is $\pi$-ample 
and that 
\begin{equation*} 
\pi\colon \Nqklt (X, \omega)\to \pi(\Nqklt(X, \omega))
\end{equation*}  
is finite. Let $P$ be a closed point of $S$ such that 
there exists a curve $C^\dag\subset \pi^{-1}(P)$ with 
$\Nqklt(X, \omega)\cap C^\dag\ne \emptyset$. 
Then there exists a non-constant morphism 
\begin{equation*} 
f\colon \mathbb A^1\longrightarrow \left(X\setminus \Nqklt(X, \omega)\right)
\cap \pi^{-1}(P) 
\end{equation*}  
such that $C$, the closure of $f(\mathbb A^1)$ in $X$, satisfies 
$C\cap \Nqklt(X, \omega)\ne \emptyset$ with 
\begin{equation*} 
0<-\omega\cdot C\leq 1. 
\end{equation*} 
\end{conj}

In this paper, we solve 
Conjecture \ref{a-conj1.15} under the assumption that 
any sequence of klt flips 
terminates. 

\begin{thm}[see Theorem \ref{n-thm14.2}]\label{a-thm1.16} 
Assume that any sequence of klt flips terminates after 
finitely many steps. 
Then Conjecture \ref{a-conj1.15} holds true. 
\end{thm}

For the precise statement of Theorem \ref{a-thm1.16}, see Theorem 
\ref{n-thm14.2}. 
In a joint paper with Kenta Hashizume (see \cite{fujino-hashizume1}), 
we will prove the following theorem, 
which is a very special case of Conjecture \ref{a-conj1.15}, by using 
some deep results in the theory of minimal models for log canonical 
pairs obtained in \cite{hashizume2}.  

\begin{thm}[{see \cite[Theorem1.7]{fujino-hashizume1}}]\label{a-thm1.17} 
Let $X$ be a normal variety and let $\Delta$ be an effective 
$\mathbb R$-divisor on $X$ such that 
$K_X+\Delta$ is $\mathbb R$-Cartier. 
Let $\pi\colon X\to S$ be a projective morphism 
onto a scheme $S$ such that 
$-(K_X+\Delta)$ is $\pi$-ample. 
We assume 
that 
\begin{equation*} 
\pi\colon \Nklt(X, \Delta)\to \pi(\Nklt(X, \Delta))
\end{equation*}  
is finite. Let $P$ be a closed point of $S$ 
such that there exists a curve 
$C^\dag \subset \pi^{-1}(P)$ with 
$\Nklt(X, \Delta)\cap C^\dag\ne \emptyset$. 
Then there exists a non-constant morphism 
\begin{equation*} 
f\colon  \mathbb A^1\longrightarrow 
\left(X\setminus \Nklt(X, \Delta)\right)\cap \pi^{-1}(P)
\end{equation*}  
such 
that 
the curve $C$, the closure of $f(\mathbb A^1)$ in $X$, 
is a {\em{(}}possibly singular{\em{)}} rational curve 
satisfying  
$C\cap \Nklt(X, \Delta)\ne \emptyset$ with  
\begin{equation*} 
0<-(K_X+\Delta)\cdot C\leq 1. 
\end{equation*} 
\end{thm}

Although Theorem \ref{a-thm1.17} looks very similar to 
Theorem \ref{a-thm1.8}, the proof of Theorem \ref{a-thm1.17} is 
much harder. 
By using Theorem \ref{a-thm1.17}, 
we will establish:  

\begin{thm}[{see \cite[Theorem 1.8]{fujino-hashizume1}}]\label{a-thm1.18} 
Conjecture \ref{a-conj1.15} holds true. 
\end{thm}

As an application of Theorem \ref{a-thm1.18}, we will prove the 
following statement in \cite{fujino-hashizume1}, which supplements 
Theorem \ref{a-thm1.6} (iii). 

\begin{thm}[{see \cite[Theorem 1.9]{fujino-hashizume1}}]\label{a-thm1.19} 
Let $[X, \omega]$ be a quasi-log scheme and let 
$\pi\colon X\to S$ be a projective 
morphism between schemes. 
Let $R_j$ be an $\omega$-negative extremal ray of $\NE(X/S)$ 
that are rational and relatively ample at infinity and 
let $\varphi_{R_j}$ be the contraction morphism associated to $R_j$. 
Let $U_j$ be any open qlc 
stratum of $[X, \omega]$ such that 
$\varphi_{R_j}\colon \overline {U_j}\to \varphi_{R_j}(\overline {U_j})$ 
is not finite and that $\varphi_{R_j}\colon W^\dag\to \varphi_{R_j}
(W^\dag)$ is finite for every qlc center $W^\dag$ 
of $[X, \omega]$ with $W^\dag 
\subsetneq \overline {U_j}$, where 
$\overline {U_j}$ is the closure of $U_j$ in $X$. 
Let $P$ be a closed point of $\varphi_{R_j}(U_j)$. 
If there exists a curve $C^\dag$ such that $\varphi_{R_j}(C^\dag)=P$, 
$C^\dag\not \subset U_j$,  
and $C^\dag\subset \overline {U_j}$, 
then there exists a non-constant 
morphism 
\begin{equation*} 
f_j\colon \mathbb A^1\longrightarrow U_j\cap \varphi^{-1}_{R_j}(P)
\end{equation*} 
such that 
$C_j$, the closure of $f_j(\mathbb A^1)$ in $X$, 
spans $R_j$ in $N_1(X/S)$ and satisfies 
$C_j\not\subset U_j$ with  
\begin{equation*} 
0<-\omega\cdot C_j\leq 1.
\end{equation*}  
\end{thm}

We note that Theorem \ref{a-thm1.19} is 
a generalization of \cite[Theorem 3.1]{lz}. 
In this paper, we prove the following simpler statement for 
dlt pairs for the 
reader's convenience since 
Theorems \ref{a-thm1.17}, \ref{a-thm1.18}, and 
\ref{a-thm1.19} are difficult. Theorem \ref{a-thm1.20} is much 
weaker than Theorem \ref{a-thm1.19}. However, 
it contains a generalization of \cite[Theorem 3.1]{lz}.  

\begin{thm}\label{a-thm1.20} 
Let $(X, \Delta)$ be a dlt pair and let $\pi\colon X\to S$ be a projective 
morphism between schemes. Let $R_j$ be a $(K_X+\Delta)$-negative 
extremal ray of $\NE(X/S)$ and let 
$\varphi_{R_j}$ be the contraction morphism associated to $R_j$. 
Let $U_j$ be any open lc 
stratum of $(X, \Delta)$ such that 
$\varphi_{R_j}\colon \overline {U_j}\to \varphi_{R_j}(\overline {U_j})$ 
is not finite and that $\varphi_{R_j}\colon W^\dag\to \varphi_{R_j}
(W^\dag)$ is finite for every lc center $W^\dag$ of $(X, \Delta)$ with $W^\dag 
\subsetneq \overline {U_j}$, 
where $\overline {U_j}$ is the closure of $U_j$ in $X$. 
If there exists a curve $C^\dag$ such that $\varphi_{R_j}(C^\dag)$ is a 
point, 
$C^\dag\not \subset U_j$,  
and $C^\dag\subset \overline {U_j}$, 
then there exists a non-constant 
morphism 
\begin{equation*} 
f_j\colon \mathbb A^1\longrightarrow U_j
\end{equation*} 
such that 
$C_j$, the closure of $f_j(\mathbb A^1)$ in $X$, 
spans $R_j$ in $N_1(X/S)$ and satisfies 
$C_j\not\subset U_j$ with 
\begin{equation*} 
0<-\omega\cdot C_j\leq 1.
\end{equation*}  
\end{thm}

Although we need some deep results on the minimal model 
program for log canonical pairs in \cite{hashizume} 
in the proof of Theorem \ref{a-thm1.20}, 
the proof of Theorem \ref{a-thm1.20} is much simpler 
than that of Theorems \ref{a-thm1.17}, 
\ref{a-thm1.18} and \ref{a-thm1.19} 
in \cite{fujino-hashizume1} and 
will help the reader understand \cite{fujino-hashizume1}. 

Finally, we make a 
conjecture on lengths of extremal rational curves 
(see \cite[Remark-Question 10-3-6]{matsuki}). 

\begin{conj}\label{a-conj1.21} 
If $\varphi_{R_j}\colon  U_j \to \varphi_{R_j} (U_j)$ is 
proper in Theorem \ref{a-thm1.6} (iii), where $\varphi_{R_j}$ 
is the contraction 
morphism associated to $R_j$, then there exists a 
{\em{(}}possibly singular{\em{)}} rational curve 
$C_j\subset U_j$ which 
spans $R_j$ in $N_1(X/S)$ and satisfies 
\begin{equation*} 
0<-\omega\cdot C_j\leq d_j+1 
\end{equation*}  
with 
\begin{equation*} 
d_j=\min_E \dim E, 
\end{equation*}  
where $E$ runs over positive-dimensional 
irreducible components of $(\varphi_{R_j}|_{U_j})^{-1}(P)$ for all 
$P\in \varphi_{R_j}(U_j)$. 
\end{conj}

The following remark on Conjecture \ref{a-conj1.21} is 
obvious. 

\begin{rem}\label{a-rem1.22} 
We use the same notation as in Conjecture \ref{a-conj1.21}. 
If $\varphi_{R_j}\colon  U_j \to \varphi_{R_j}(U_j)$ is proper 
in Theorem \ref{a-thm1.6} (iii), 
we can make $C_j$ satisfy 
\begin{equation*} 
0<-\omega\cdot C_j \leq 2d_j
\end{equation*}  
by Theorem \ref{a-thm1.12}. 
\end{rem}

Of course, we hope that the following sharper estimate 
\begin{equation*} 
0<-\omega\cdot \ell \leq \dim E+1
\end{equation*}  
should hold true in Theorem \ref{a-thm1.12}. 

In \cite{fujino-hashizume2}, we will generalize the framework of 
basic slc-trivial fibrations for $\mathbb R$-divisors and 
establish adjunction and inversion of adjunction for 
log canonical centers of arbitrary codimension in full generality. 
We strongly recommend the interested reader to see 
\cite{fujino-hashizume1} 
and \cite{fujino-hashizume2} after reading this paper. 

\medskip 

We briefly look at the organization of this paper. 
In Section \ref{b-sec2}, 
we recall some basic definitions and results. 
Then we treat the notion of uniruledness, 
rational connectedness, and rational chain connectedness. 
In Section \ref{c-sec3}, we treat some basic definitions and 
results on normal pairs and 
then discuss dlt blow-ups for quasi-projective 
normal pairs. 
In Section \ref{d-sec4}, we briefly review the theory 
of quasi-log schemes and prepare some useful and 
important lemmas. 
In Section \ref{e-sec5}, we give a detailed proof of 
Theorem \ref{a-thm1.9}. Theorem \ref{a-thm1.9} 
plays a crucial role since a quasi-log scheme is 
not necessarily normal even when it is a variety. 
In Section \ref{f-sec6}, we quickly explain 
basic slc-trivial fibrations. 
The results in \cite{fujino-slc-trivial} 
make the theory of quasi-log schemes very powerful. 
In Section \ref{g-sec7}, we prove a very important 
result on normal quasi-log schemes, which is a 
slight generalization of \cite[Theorem 1.7]{fujino-slc-trivial}. 
In Section \ref{p-sec8}, we prove Theorem \ref{a-thm1.10} 
by using the result explained in Section \ref{g-sec7}. 
Hence Theorem \ref{a-thm1.10} heavily depends on 
some deep results on the theory of variations of mixed Hodge 
structure. 
In Section \ref{i-sec9}, we prove Theorem \ref{a-thm1.8}. 
Note that Theorem \ref{a-thm1.8} was essentially obtained 
in \cite{lz} and \cite{svaldi} under some extra assumptions. 
In Section \ref{j-sec10}, we prove Theorems \ref{a-thm1.4}, 
\ref{a-thm1.5}, and 
\ref{a-thm1.6}. We note that Theorem \ref{a-thm1.5} is a 
special case of Theorem \ref{a-thm1.6}. 
In Section \ref{k-sec11}, we discuss an 
ampleness criterion for quasi-log schemes. 
As a very special case, we prove Theorem \ref{a-thm1.11}. 
In Section \ref{l-sec12}, we treat Theorems \ref{a-thm1.12} and 
\ref{a-thm1.13}. They are generalizations of Kawamata's famous 
result for quasi-log schemes. 
In Section \ref{m-sec13}, we prove Theorem \ref{a-thm1.14}, 
which is well known for normal pairs. 
In Section \ref{n-sec14}, we discuss several results 
related to Conjecture \ref{a-conj1.15}. 

\begin{ack}
The author was partially 
supported by JSPS KAKENHI Grant Numbers 
JP16H03925, JP16H06337, JP19H01787, JP20H00111, JP21H00974. 
He thanks Kenta Hashizume very much for many useful 
comments and suggestions. 
He also thanks the referees very much for many useful comments. 
\end{ack}

\section{Preliminaries}\label{b-sec2}

We will work over $\mathbb C$, the complex number field, 
throughout this paper. In this paper, a {\em{scheme}} means 
a separated scheme of finite type over $\mathbb C$.  
A {\em{variety}} means an integral scheme, that is, 
an irreducible and reduced separated scheme of 
finite type over $\mathbb C$. 
Note that $\mathbb Z$, $\mathbb Q$, and 
$\mathbb R$ denote the set of 
{\em{integers}}, 
{\em{rational numbers}}, and {\em{real numbers}}, respectively. 
We also note that $\mathbb Q_{>0}$ and $\mathbb R_{>0}$ 
are the set of {\em{positive rational numbers}} and 
{\em{positive real numbers}}, respectively. 

\subsection{Basic definitions}\label{b-subsec2.1} 

We collect some basic definitions and several useful results. 
Let us start with the definition of {\em{$\mathbb Q$-line bundles}} and 
{\em{$\mathbb R$-line bundles}}. 

\begin{defn}[$\mathbb Q$-line bundles and 
$\mathbb R$-line bundles]\label{b-def2.1} 
Let $X$ be a scheme and let $\Pic(X)$ be 
the group of line bundles on $X$, that is, 
the {\em{Picard group}} of $X$. 
An element of $\Pic(X)\otimes _{\mathbb Z} \mathbb R$ 
(resp.~$\Pic(X)\otimes _{\mathbb Z} \mathbb Q$) is called an 
{\em{$\mathbb R$-line bundle}} (resp.~a 
{\em{$\mathbb Q$-line bundle}}) 
on $X$. 
\end{defn}

In this paper, we write the group law of $\Pic(X)\otimes _{\mathbb Z} 
\mathbb R$ additively for simplicity of notation. 
The notion of {\em{$\mathbb R$-Cartier divisors}} and 
{\em{$\mathbb Q$-Cartier 
divisors}} also plays a crucial role for the study of higher-dimensional 
algebraic varieties. 

\begin{defn}
[$\mathbb Q$-Cartier divisors and $\mathbb R$-Cartier 
divisors]\label{b-def2.2}
Let $X$ be a scheme and let $\Div(X)$ be the group of 
Cartier divisors on $X$. 
An element of $\Div(X)\otimes _{\mathbb Z} \mathbb R$ 
(resp.~$\Div(X)\otimes _{\mathbb Z} \mathbb Q$) is called an 
{\em{$\mathbb R$-Cartier divisor}} (resp.~a 
{\em{$\mathbb Q$-Cartier divisor}}) on $X$. 
Let $\Delta_1$ and $\Delta_2$ be $\mathbb R$-Cartier 
(resp.~$\mathbb Q$-Cartier) divisors on $X$. 
Then $\Delta_1\sim _{\mathbb R} \Delta_2$ (resp.~$\Delta_1\sim 
_{\mathbb Q}\Delta_2$) means that $\Delta_1$ is $\mathbb R$-linearly 
(resp.~$\mathbb Q$-linearly) equivalent to $\Delta_2$. 
Let $f\colon X\to Y$ be a morphism between schemes and let 
$D$ be an $\mathbb R$-Cartier divisor on $X$. 
Then $D\sim _{\mathbb R, f}0$ means that 
there exists an $\mathbb R$-Cartier divisor $G$ on $Y$ such that 
$D\sim _{\mathbb R} f^*G$. 
\end{defn}

The following remark is very important. 

\begin{rem}[{see \cite[Remark 6.2.3]{fujino-foundations}}]\label{b-rem2.3}
Let $X$ be a scheme. 
We have the following group homomorphism 
\begin{equation*} 
\Div(X)\to \Pic (X) 
\end{equation*}  
given by $A\mapsto \mathcal O_X(A)$, where $A$ is a Cartier divisor 
on $X$. 
Hence it induces a homomorphism 
\begin{equation*} 
\delta_X\colon  \Div(X)\otimes _{\mathbb Z} \mathbb R\to 
\Pic(X)\otimes _{\mathbb Z} \mathbb R. 
\end{equation*}  
Note that 
\begin{equation*} 
\Div(X)\to \Pic(X)
\end{equation*}  
is not always surjective. 
We write 
\begin{equation*} 
A+\mathcal L\sim _{\mathbb R} B+\mathcal M
\end{equation*}  
for $A, B\in \Div(X)\otimes _{\mathbb Z} \mathbb R$ and 
$\mathcal L, \mathcal M\in \Pic(X)\otimes _{\mathbb Z} 
\mathbb R$. 
This means that 
\begin{equation*} 
\delta_X(A)+\mathcal L=\delta_X(B)+\mathcal M
\end{equation*}  
holds in $\Pic(X)\otimes _{\mathbb Z} \mathbb R$. 
We usually use this type of abuse of notation, that is, 
the confusion of $\mathbb R$-line bundles with 
$\mathbb R$-Cartier divisors. 
In the theory of minimal models for higher-dimensional 
algebraic varieties, we sometimes 
use $\mathbb R$-Cartier divisors 
for ease of notation even when they should be 
$\mathbb R$-line bundles. 
\end{rem}

On normal varieties or equidimensional reduced schemes, 
we often treat $\mathbb R$-divisors and $\mathbb Q$-divisors. 

\begin{defn}[Operations for $\mathbb Q$-divisors and 
$\mathbb R$-divisors]\label{b-def2.4} 
Let $X$ be an equidimensional reduced scheme. 
Note that $X$ is not necessarily regular in codimension one. 
Let $D$ be an $\mathbb R$-divisor (resp.~a $\mathbb Q$-divisor), 
that is, 
$D$ is a finite formal sum $\sum _i d_iD_i$, where 
$D_i$ is an irreducible reduced closed subscheme of $X$ of 
pure codimension one and $d_i$ is a real number (resp.~a rational 
number) for every $i$ 
such that $D_i\ne D_j$ for $i\ne j$. 
We put 
\begin{equation*}
D^{<c} =\sum _{d_i<c}d_iD_i, \quad 
D^{\leq c}=\sum _{d_i\leq c} d_i D_i, \quad 
D^{= 1}=\sum _{d_i= 1} D_i, \quad \text{and} \quad
\lceil D\rceil =\sum _i \lceil d_i \rceil D_i, 
\end{equation*}
where $c$ is any real number and 
$\lceil d_i\rceil$ is the integer defined by $d_i\leq 
\lceil d_i\rceil <d_i+1$. Similarly, we put 
\begin{equation*} 
D^{>c} =\sum _{d_i>c}d_iD_i \quad \text{and} \quad 
D^{\geq c}=\sum _{d_i\geq c} d_i D_i 
\end{equation*} 
for any real number $c$. 
Moreover, we put $\lfloor D\rfloor =-\lceil -D\rceil$ and 
$\{D\}=D-\lfloor D\rfloor$. 

Let $D$ be an $\mathbb R$-divisor (resp.~a $\mathbb Q$-divisor) 
as above. 
We call $D$ a {\em{subboundary}} $\mathbb R$-divisor 
(resp.~$\mathbb Q$-divisor) 
if $D=D^{\leq 1}$ holds. 
When $D$ is effective and $D=D^{\leq 1}$ holds, 
we call $D$ a {\em{boundary}} 
$\mathbb R$-divisor (resp.~$\mathbb Q$-divisor). 

We further assume that 
$f\colon X\to Y$ is a surjective morphism onto a variety $Y$. 
Then we put 
\begin{equation*} 
D^v=\sum _{f(D_i)\subsetneq Y}d_i D_i \quad 
\text{and} \quad D^h=D-D^v, 
\end{equation*} 
and call $D^v$ the {\em{vertical part}} 
and $D^h$ the {\em{horizontal part}} of $D$ 
with respect to $f\colon X\to Y$, respectively. 
\end{defn}

Since we mainly treat highly singular schemes, we give an 
important remark. 

\begin{rem}\label{b-rem2.5}
In the theory of minimal models, we are mainly interested in normal 
quasi-projective varieties. 
Let $X$ be a normal variety. 
Then, for $\mathbb K=\mathbb Z, \mathbb Q$, and 
$\mathbb R$, the homomorphism 
\begin{equation*}  
\alpha\colon  \Div(X)\otimes _{\mathbb Z}\mathbb K 
\to \Pic(X)\otimes _{\mathbb Z} \mathbb K
\end{equation*}  
is surjective and the homomorphism 
\begin{equation*}  
\beta\colon  \Div(X)\otimes _{\mathbb Z} \mathbb K\to 
\Weil (X)\otimes _{\mathbb Z}\mathbb K
\end{equation*}  
is injective, where $\Weil (X)$ is the abelian group generated by 
Weil divisors on $X$. We usually use the surjection $\alpha$ and the 
injection $\beta$ implicitly. In this paper, however, 
we frequently treat highly singular schemes $X$. 
Hence we have to be careful when we consider $\alpha\colon \Div(X)
\otimes _{\mathbb Z} \mathbb K\to 
\Pic(X)\otimes _{\mathbb Z} \mathbb K$ and $\beta\colon  
\Div(X)\otimes _{\mathbb Z} \mathbb K\to 
\Weil(X)\otimes _{\mathbb Z} \mathbb K$.  
\end{rem}

Let us recall the following standard notation for the 
sake of completeness. 

\begin{defn}[$N^1(X/S)$, $N_1(X/S)$, $\rho(X/S)$, and so on]\label{b-def2.6}
Let $\pi\colon  X\to S$ be a proper morphism between schemes. 
Let $Z_1(X/S)$ be the free abelian group generated by 
integral complete curves which are mapped 
to points on $S$ by $\pi$. 
Then we obtain a bilinear form 
\begin{equation*} 
\cdot\colon  \Pic(X)\times Z_1(X/S)\to \mathbb Z, 
\end{equation*}  
which is induced by the intersection pairing. 
We have the notion of {\em{numerical equivalence}} 
both in $Z_1(X/S)$ and in $\Pic(X)$, which is denoted by 
$\equiv$, and we obtain a perfect pairing 
\begin{equation*} 
N^1(X/S)\times N_1(X/S)\to \mathbb R, 
\end{equation*}  
where 
\begin{equation*} 
N^1(X/S)=\{\Pic(X)/\equiv\}\otimes _{\mathbb Z}\mathbb R 
\quad \text{and}\quad 
N_1(X/S)=\{Z_1(X/S)/\equiv \}\otimes _{\mathbb Z} \mathbb R. 
\end{equation*} 
It is well known that 
\begin{equation*} 
\dim _\mathbb R N^1(X/S)=\dim _\mathbb R N_1(X/S)<\infty. 
\end{equation*}  
We write 
\begin{equation*} 
\rho(X/S)=\dim _\mathbb R N^1(X/S)=\dim _\mathbb R N_1(X/S)
\end{equation*}  
and call it the {\em{relative Picard number}} of $X$ over $S$. 
When $S=\Spec \mathbb C$, we usually drop $/\Spec\mathbb C$ 
from the notation, for example, 
we simply write $N_1(X)$ instead of $N_1(X/\Spec \mathbb C)$. 
\end{defn}

We will freely use the following useful lemma 
without mentioning 
it explicitly in the subsequent sections. 

\begin{lem}[Relative real Nakai--Moishezon 
ampleness criterion]\label{b-lem2.7}
Let $\pi\colon X\to S$ be a proper morphism between 
schemes and 
let $\mathcal L$ be an $\mathbb R$-line bundle 
on $X$. 
Then $\mathcal L$ is $\pi$-ample if and only if 
$\mathcal L^{\dim Z}\cdot Z>0$ for every 
positive-dimensional closed integral subscheme 
$Z\subset X$ such that $\pi(Z)$ is a point. 
\end{lem}

For the details of Lemma \ref{b-lem2.7}, 
see \cite{fujino-miyamoto}. In the theory of quasi-log schemes, 
we mainly treat highly singular reducible schemes. 
Hence Lemma \ref{b-lem2.7} is very useful in order to check 
the ampleness of 
$\mathbb R$-line bundles. 

\subsection{Uniruledness, rational connectedness, and 
rational chain connectedness}\label{b-subsec2.2}
In this subsection, we quickly recall 
the notion of uniruledness, rational connectedness, 
rational chain connectedness, and so on. 
We need it for Theorems \ref{a-thm1.12}, 
\ref{a-thm1.13}, and \ref{a-thm1.14}. 
For the details, see \cite[Chapter IV.]{kollar-rational}. 
We note that a scheme means a separated scheme of 
finite type over $\mathbb C$ 
in this paper. 
Let us start with the definition of uniruled varieties. 

\begin{defn}[{Uniruledness, 
see \cite[Chapter IV.~1.1 Definition]{kollar-rational}}]
\label{b-def2.8}
Let $X$ be a variety. We say that 
$X$ is {\em{uniruled}} if there exist a variety $Y$ of 
dimension $\dim X-1$ and 
a dominant rational map 
\begin{equation*} 
\mathbb P^1\times Y\dashrightarrow X. 
\end{equation*} 
\end{defn}

Although the notion of rational connectedness is 
dispensable for 
Theorem \ref{a-thm1.14}, we explain it for the 
reader's convenience. 

\begin{defn}[{Rational connectedness, see 
\cite[Chapter IV.~3.6 Proposition]{kollar-rational}}]
\label{b-def2.9}
Let $X$ be a projective variety. 
We say that $X$ is {\em{rationally connected}} 
if for general closed points $x_1, x_2\in X$ there 
exists an irreducible rational curve $C$ which 
contains $x_1$ and $x_2$. 
\end{defn}

The following lemma is almost obvious by definition. 

\begin{lem}\label{b-lem2.10}
Let $X\dashrightarrow X'$ be a generically finite dominant 
rational map between varieties. 
If $X$ is uniruled, then $X'$ is also uniruled. 
Furthermore, we assume that 
$X\dashrightarrow X'$ is a birational map between 
projective varieties. 
Then $X$ is rationally connected if and only if 
$X'$ is rationally connected. 
\end{lem}

Let us define rational chain connectedness for 
projective schemes. 

\begin{defn}[{Rational chain connectedness, see 
\cite[Chapter IV.~3.5 Corollary and 3.6 Proposition]{kollar-rational}}]
\label{b-def2.11}
Let $X$ be a projective scheme. 
We say that $X$ is {\em{rationally chain connected}} 
if for arbitrary closed points $x_1, x_2\in X$ there 
is a connected curve $C$ which contains $x_1$ and $x_2$ 
such that every irreducible component of $C$ is rational. 
\end{defn}

Note that $X$ may be reducible in Definition \ref{b-def2.11}. 
For projective varieties, we have:  

\begin{lem}\label{b-lem2.12} 
Let $X$ be a projective variety. 
If $X$ is rationally connected, then $X$ is rationally chain 
connected. 
\end{lem}

\begin{proof}
This follows from \cite[Chapter IV.~3.6 Proposition]{kollar-rational}. 
\end{proof}

We need the following definition for 
Theorem \ref{a-thm1.14}. 

\begin{defn}[{\cite[Definition 1.1]{hacon-mckernan}}]\label{b-def2.13}
Let $X$ be a projective scheme 
and let $V$ be any closed subset. 
We say that X is {\em{rationally chain connected modulo $V$}} if
\begin{itemize}
\item[(1)] either $V=\emptyset$ and 
$X$ is rationally chain connected, or
\item[(2)] 
$V\ne\emptyset$ and, for every $P\in X$, there is a 
connected pointed curve $0, \infty\in C$ with rational irreducible 
components and a morphism $h_P \colon  C\to X$ such that 
$h_P(0)=P$ and $h_P(\infty)\in V$.
\end{itemize}
\end{defn}

We close this subsection with a 
small remark. 

\begin{rem}\label{b-rem2.14}
Let $X$ be a singular normal projective rationally chain connected 
variety. 
Then the resolution of $X$ is not always rationally chain connected. 
Hence the notion of 
rational chain connectedness is more subtle 
than that of uniruledness and rational connectedness 
(see Lemma \ref{b-lem2.10}). 
\end{rem}

\section{On normal pairs}\label{c-sec3}
 
In this section, we collect some basic definitions 
and then discuss dlt blow-ups 
for normal pairs. 
For the details of normal pairs, see \cite{bchm}, 
\cite{fujino-fund}, and \cite{fujino-foundations}. 
Let us start with the definition of normal pairs in this paper. 

\begin{defn}[Normal pairs]\label{c-def3.1}
A {\em{normal pair}} $(X, \Delta)$ consists of 
a normal variety $X$ and an $\mathbb R$-divisor 
$\Delta$ on $X$ such that $K_X+\Delta$ is 
$\mathbb R$-Cartier. 
Here we do not always assume that $\Delta$ is effective. 
\end{defn}
 
We note the following definition of {\em{exceptional loci}} of 
birational morphisms between varieties. 

\begin{defn}[Exceptional loci]\label{c-def3.2}
Let $f\colon X\to Y$ be a birational morphism 
between varieties. 
Then the {\em{exceptional locus}} $\Exc(f)$ of $f\colon X\to Y$ 
is the set 
\begin{equation*} 
\{x\in X\, | \, {\text{$f$ is not biregular at $x$}}\}. 
\end{equation*} 
\end{defn}
 
\subsection{Singularities of pairs}\label{c-subsec3.1}
Let us explain singularities of pairs and some related 
definitions. 

\begin{defn}\label{c-def3.3} 
Let $X$ be a variety and let $E$ be a prime divisor on $Y$ 
for some birational
morphism $f\colon Y\to X$ from a normal variety $Y$. 
Then $E$ is called a divisor {\em{over}} $X$.
\end{defn}

\begin{defn}[Singularities of pairs]\label{c-def3.4}
Let $(X, \Delta)$ be a normal pair and let 
$f\colon Y\to X$ be a projective 
birational morphism from a normal variety $Y$. 
Then we can write 
\begin{equation*} 
K_Y=f^*(K_X+\Delta)+\sum _E a(E, X, \Delta)E
\end{equation*}  
with 
\begin{equation*} 
f_*\left(\underset{E}\sum a(E, X, \Delta)E\right)=-\Delta, 
\end{equation*}  
where $E$ runs over prime divisors on $Y$. 
We call $a(E, X, \Delta)$ the {\em{discrepancy}} of $E$ with 
respect to $(X, \Delta)$. 
Note that we can define the discrepancy $a(E, X, \Delta)$ for 
any prime divisor $E$ over $X$ by taking a suitable 
resolution of singularities of $X$. 
If $a(E, X, \Delta)\geq -1$ (resp.~$>-1$) for 
every prime divisor $E$ over $X$, 
then $(X, \Delta)$ is called {\em{sub log canonical}} (resp.~{\em{sub 
kawamata log terminal}}). 
We further assume that $\Delta$ is effective. 
Then $(X, \Delta)$ is 
called {\em{log canonical}} and {\em{kawamata log terminal}} 
({\em{lc}} and {\em{klt}}, for short) 
if it is sub log canonical and sub kawamata log terminal, respectively. 

Let $(X, \Delta)$ be a log canonical pair. If there 
exists a projective birational morphism 
$f\colon Y\to X$ from a smooth variety $Y$ such that 
both $\Exc(f)$ and  $\Exc(f)\cup \Supp f^{-1}_*\Delta$ are simple 
normal crossing divisors on $Y$ and that 
$a(E, X, \Delta)>-1$ holds for every $f$-exceptional divisor $E$ on $Y$, 
then $(X, \Delta)$ is called {\em{divisorial log terminal}} ({\em{dlt}}, for short). 

Let $(X, \Delta)$ be a normal pair. 
If there exist a projective birational morphism 
$f\colon Y\to X$ from a normal variety $Y$ and a prime divisor $E$ on $Y$ 
such that $(X, \Delta)$ is 
sub log canonical in a neighborhood of the 
generic point of $f(E)$ and that 
$a(E, X, \Delta)=-1$, then $f(E)$ is called a 
{\em{log canonical center}} (an {\em{lc center}}, for short) 
of 
$(X, \Delta)$. A closed subvariety $W$ of $X$ is called 
a {\em{log canonical stratum}} (an {\em{lc stratum}}, for 
short) of $(X, \Delta)$ if $W$ is a log canonical center 
of $(X, \Delta)$ or $W$ is $X$ itself. 
\end{defn}

Although it is well known, 
we recall the notion of {\em{multiplier ideal 
sheaves}} here for the reader's convenience. 

\begin{defn}[Multiplier ideal sheaves and non-lc ideal sheaves]\label{c-def3.5} 
Let $X$ be a normal variety
and let $\Delta$ be an effective $\mathbb R$-divisor on $X$ 
such that $K_X+\Delta$ is $\mathbb R$-Cartier. 
Let $f\colon Y\to X$ be a resolution with
\begin{equation*} 
K_Y+\Delta_Y=f^*(K_X+\Delta)
\end{equation*} 
such that $\Supp \Delta_Y$ is a simple normal crossing divisor on $Y$. 
We put
\begin{equation*} 
\mathcal J(X, \Delta)=f_*\mathcal O_Y(-\lfloor \Delta_Y\rfloor). 
\end{equation*} 
Then $\mathcal J(X, \Delta)$ is an ideal sheaf on $X$ 
and is known as the {\em{multiplier ideal sheaf}} 
associated to the pair $(X, \Delta)$. 
It is independent 
of the resolution $f\colon Y\to X$. The closed subscheme $\Nklt(X, \Delta)$ 
defined by $\mathcal J(X, \Delta)$ is called 
the {\em{non-klt locus}} of $(X, \Delta)$. 
It is obvious that $(X, \Delta)$ is kawamata log terminal 
if and only if $\mathcal J(X, \Delta)=\mathcal O_X$. 
Similarly, we put 
\begin{equation*} 
\mathcal J_{\NLC}(X, \Delta)=f_*\mathcal O_X(-\lfloor 
\Delta_Y\rfloor+\Delta^{=1}_Y)
\end{equation*} 
and call it the {\em{non-lc ideal sheaf}} associated to the pair $(X, \Delta)$. 
We can check that it is independent of the resolution $f\colon Y\to X$. The
closed subscheme $\Nlc(X, \Delta)$ defined by 
$\mathcal J_{\NLC}(X, \Delta)$ 
is called the {\em{non-lc locus}} of 
$(X, \Delta)$. It is obvious that $(X, \Delta)$ is log canonical if and only if 
$\mathcal J_{\NLC}(X, \Delta)=\mathcal O_X$. 

By definition, the natural inclusion 
\begin{equation*} 
\mathcal J(X, \Delta)\subset \mathcal J_{\NLC}(X, \Delta)
\end{equation*}  
always holds. Therefore, 
we have 
\begin{equation*} 
\Nlc(X, \Delta)\subset \Nklt(X, \Delta). 
\end{equation*} 
\end{defn}
 
For the details of $\mathcal J(X, \Delta)$ and 
$\mathcal J_{\NLC}(X, 
\Delta)$, see \cite{fujino-non-lc}, 
\cite[Section 7]{fujino-fund}, and \cite[Chapter 9]{lazarsfeld}. 
In this paper, we need the notion of {\em{open lc strata}}. 

\begin{defn}[Open lc strata]\label{c-def3.6}
Let $(X, \Delta)$ be a normal pair 
such that $\Delta$ is effective. 
Let $W$ be an lc stratum of $(X, \Delta)$. 
We put 
\begin{equation*} 
U:=W\setminus \left\{ \left(W\cap \Nlc(X, \Delta)\right)\cup 
\bigcup_{W'}W'\right\}, 
\end{equation*} 
where $W'$ runs over lc centers of $(X, \Delta)$ strictly contained 
in $W$, 
and call it the {\em{open lc stratum of 
$(X, \Delta)$ associated to $W$}}. 
\end{defn}

\subsection{Dlt blow-ups revisited}\label{c-subsec3.2}

Let us discuss dlt blow-ups. 
We give a slight generalization of 
\cite[Theorem 4.4.21]{fujino-foundations}. 
Here we use the theory of minimal models 
mainly due to \cite{bchm}. 
Let us start with the definition of movable divisors. 

\begin{defn}[{Movable 
divisors and movable cones, see \cite[Definition 2.4.4]{fujino-foundations}}]
\label{c-def3.7}
Let $f\colon X\to Y$ be a projective morphism from a normal variety $X$ onto a 
variety $Y$. A Cartier divisor $D$ on $X$ is called {\em{$f$-movable}} 
or {\em{movable over $Y$}} if 
$f_*\mathcal O_X(D)\ne 0$ and 
if the cokernel of the natural homomorphism 
\begin{equation*} 
f^*f_*\mathcal O_X(D)\to \mathcal O_X(D)
\end{equation*}  
has a support of codimension $\geq 2$. 

We define $\Mov (X/Y)$ as the closure of 
the convex cone in $N^1(X/Y)$ generated by 
the numerical equivalence classes of $f$-movable Cartier divisors. 
We call $\Mov (X/Y)$ the {\em{movable cone}} of $f\colon X\to Y$. 
\end{defn}

We prepare a negativity lemma. 

\begin{lem}[Negativity lemma]\label{c-lem3.8}
Let $f\colon X\to Y$ be a projective birational morphism 
between normal varieties such that $X$ is $\mathbb Q$-factorial. 
Let $E$ be an $\mathbb R$-Cartier $\mathbb R$-divisor 
on $X$ such that $-f_*E$ is effective and $E\in \Mov (X/Y)$. 
Then $-E$ is effective.  
\end{lem}
\begin{proof} 
We can write $E=E_+-E_-$ such that 
$E_+$ and $E_-$ are effective $\mathbb R$-divisors 
and have no common irreducible components. 
We assume that $E_+\ne 0$. Since $-f_*E$ is effective, 
$E_+$ is $f$-exceptional. 
Without loss of generality, 
we may assume that $Y$ is affine by taking 
an affine open covering of $Y$. 
Let $A$ be an ample Cartier divisor on $X$. 
Then we can find an irreducible component $E_0$ of $E_+$ such that 
\begin{equation*}
E_0\cdot (f^*A)^k\cdot H^{n-k-2}\cdot E<0
\end{equation*} 
when $\dim X=n$ and $\codim _Y f(E_+)=k$. 
This is a contradiction. Note that 
\begin{equation*}
E_0\cdot (f^*A)^k\cdot H^{n-k-2}\cdot E\geq  0 
\end{equation*} always hold 
since $E\in \Mov (X/Y)$. Therefore, $-E$ is effective. 
\end{proof}

By Lemma \ref{c-lem3.8}, 
we can prove the existence of dlt blow-ups for quasi-projective 
normal pairs. 
We note that $\Delta$ is assumed to be a boundary 
$\mathbb R$-divisor in \cite[Theorem 4.4.21]{fujino-foundations}. 

\begin{thm}[Dlt blow-ups]\label{c-thm3.9}
Let $X$ be a normal quasi-projective variety and let $\Delta=\sum _i 
d_i \Delta_i$ be an effective $\mathbb R$-divisor on $X$ such 
that $K_X+\Delta$ is $\mathbb R$-Cartier. 
In this case, we can construct a projective birational morphism 
$f\colon Y\to X$ from a normal quasi-projective variety $Y$ with the 
following properties. 
\begin{itemize}
\item[(i)] $Y$ is $\mathbb Q$-factorial. 
\item[(ii)] $a(E, X, \Delta)\leq -1$ for every $f$-exceptional 
divisor $E$ on $Y$. 
\item[(iii)] We put 
\begin{equation*} 
\Delta^\dag=\sum _{0<d_i< 1}d_i f^{-1}_*\Delta_i+\sum _{d_i\geq 1} 
f^{-1}_*\Delta_i+\sum _{\text{$E$:~$f$-exceptional}} E. 
\end{equation*} 
Then $(Y, \Delta^\dag)$ is dlt and the following 
equality 
\begin{equation*} 
K_Y+\Delta^\dag=f^*(K_X+\Delta)+\sum _{a(E, X, \Delta)<-1}
(a(E, X, \Delta)+1)E
\end{equation*} 
holds. 
\end{itemize}
\end{thm}

We only give a sketch of the proof of Theorem \ref{c-thm3.9} 
since the proof of \cite[Theorem 4.4.21]{fujino-foundations} 
works by Lemma \ref{c-lem3.8}. 

\begin{proof}[Sketch of Proof of Theorem \ref{c-thm3.9}]
Let $g\colon Z\to X$ be a resolution such that 
$\Exc(g)\cup \Supp g^{-1}_*\Delta$ is a simple 
normal crossing divisor on $X$ and 
$g$ is projective. 
We write 
\begin{equation*} 
K_Z+\widetilde \Delta=g^*(K_X+\Delta)+F, 
\end{equation*} 
where 
\begin{equation*} 
\widetilde \Delta=\sum _{0<d_i< 1}d_i g^{-1}_*\Delta_i+\sum _{d_i\geq 1} 
g^{-1}_*\Delta_i+\sum _{\text{$E$:~$g$-exceptional}} E. 
\end{equation*} 
We note that $-g_*F$ is effective by construction. 
Then we apply the same argument as in the proof of 
\cite[Theorem 4.4.21]{fujino-foundations}, that is, 
we run a suitable minimal model program with 
respect to $(Z, \widetilde \Delta)$ over $X$. 
After finitely many steps, we see that 
the effective part of $F$ is contracted. 
Note that all we have to do is to use Lemma \ref{c-lem3.8} 
instead of \cite[Lemma 2.4.5]{fujino-foundations}.  
\end{proof}

When $\Delta$ is a boundary $\mathbb R$-divisor, 
Lemma \ref{c-lem3.10} is nothing but \cite[Theorem 3.4]{svaldi}. 

\begin{lem}\label{c-lem3.10}
Let $X$ be a normal quasi-projective variety and 
let $\Delta$ be an effective 
$\mathbb R$-divisor on $X$ such that 
$K_X+\Delta$ is $\mathbb R$-Cartier. 
Then we can construct a projective 
birational morphism 
$g\colon Y\to X$ from a normal $\mathbb Q$-factorial variety 
$Y$ with the following properties. 
\begin{itemize}
\item[(i)] $K_Y+\Delta_Y:=g^*(K_X+\Delta)$, 
\item[(ii)] the pair 
\begin{equation*} 
\left(Y, \Delta'_Y:=\sum _{d_i <1}d_i D_i+
\sum _{d_i\geq 1}D_i\right)
\end{equation*} 
is dlt, where $\Delta_Y=\sum _i d_i D_i$ is 
the irreducible decomposition of 
$\Delta_Y$, 
\item[(iii)] every $g$-exceptional prime divisor 
is a component of $(\Delta'_Y)^{=1}$, and 
\item[(iv)] $g^{-1}\Nklt(X, \Delta)$ coincides with 
$\Nklt(Y, \Delta_Y)$ and $\Nklt(Y, \Delta'_Y)$ set theoretically. 
\end{itemize}
\end{lem}

By Theorem \ref{c-thm3.9}, the proof of \cite[Theorem 3.4]{svaldi} 
works without any changes even when $\Delta$ is not 
a boundary $\mathbb R$-divisor. 
We give a proof for the sake of completeness. 

\begin{proof}[Proof of Lemma \ref{c-lem3.10}]
There exists a dlt blow-up $\alpha\colon Z\to X$ with 
$K_Z+\Delta_Z:=\alpha^*(K_X+\Delta)$ satisfying 
(i), (ii), and (iii) by Theorem \ref{c-thm3.9}. 
Note that $(Z, \Delta^{<1}_Z)$ is a $\mathbb Q$-factorial 
kawamata log terminal pair. 
We take a minimal model $(Z', \Delta^{<1}_{Z'})$ of 
$(Z, \Delta^{<1}_Z)$ over $X$ by \cite{bchm}. 
\begin{equation*} 
\xymatrix{
Z \ar[dr]_-\alpha\ar@{-->}[rr]^-\varphi& & Z' \ar[dl]^-{\alpha'}\\ 
&X &
}
\end{equation*} 
Then $K_{Z'}+\Delta^{<1}_{Z'}
\sim _{\mathbb R} -\Delta^{\geq 1}_{Z'}+\alpha'^*(K_X+\Delta)$ 
is nef over $X$. 
Of course, we put $\Delta_{Z'}=\varphi_*\Delta_Z$. 
We take a dlt blow-up 
$\beta\colon Y\to Z'$ of $(Z', \Delta^{<1}_{Z'}+\Supp \Delta^{\geq 1}_{Z'})$ again 
by Theorem \ref{c-thm3.9} (or 
\cite[Theorem 4.4.21]{fujino-foundations}) and put $g:=
\alpha'\circ \beta\colon Y\to X$. 
It is not difficult to see that this birational morphism $g\colon Y\to X$ with 
$K_Y+\Delta_Y:=g^*(K_X+\Delta)$ satisfies the desired properties. 
It is obvious that $g^{-1}\Nklt(X, \Delta)$ contains the 
support of $\beta^*\Delta^{\geq 1}_{Z'}$. 
Since $-\beta^*\Delta^{\geq 1}_{Z'}$ is nef 
over $X$, we see that $\beta^*\Delta^{\geq 1}_{Z'}$ coincides with 
$g^{-1}\Nklt(X, \Delta)$ set theoretically.  
\end{proof}

For the details of the proof of Lemma \ref{c-lem3.10}, see 
\cite[Theorem 3.4]{svaldi}. 
In \cite{fujino-hashizume1}, Theorem \ref{c-thm3.9} and 
Lemma \ref{c-lem3.10} will be generalized completely 
by using the minimal model program for 
log canonical pairs established in \cite{hashizume2}. 

\section{On quasi-log schemes}\label{d-sec4}

In this section, we explain some basic definitions and 
results on quasi-log schemes. 
For the details of the theory of quasi-log schemes, 
we recommend the reader to see \cite[Chapter 6]{fujino-foundations} 
and \cite{fujino-on-quasi-log}. 

\subsection{Definitions and basic 
properties of quasi-log schemes}\label{d-subsec4.1}
The notion of quasi-log schemes was first introduced by 
Florin Ambro (see \cite{ambro}) 
in order to 
establish the cone and contraction theorem for 
$(X, \Delta)$, where 
$X$ is a normal variety and $\Delta$ is an 
effective $\mathbb R$-divisor 
on $X$ such that $K_X+\Delta$ is $\mathbb R$-Cartier. 
Here we use the formulation in \cite[Chapter 6]
{fujino-foundations}, which is slightly different from 
Ambro's original one. 
We recommend the interested reader to see \cite[Appendix A]{fujino-pull} 
for the difference between our definition of quasi-log schemes 
and Ambro's one (see also \cite[Section 8]{fujino-on-quasi-log}). 

\medskip 

In order to define quasi-log schemes, we use the notion of 
{\em{globally embedded simple normal crossing pairs}}. 

\begin{defn}[{Globally embedded simple normal crossing 
pairs, see \cite[Definition 6.2.1]{fujino-foundations}}]\label{d-def4.1} 
Let $Y$ be a simple normal crossing divisor 
on a smooth 
variety $M$ and let $B$ be an $\mathbb R$-divisor 
on $M$ such that 
$\Supp (B+Y)$ is a simple normal crossing divisor on $M$ and that 
$B$ and $Y$ have no common irreducible components. 
We put $B_Y=B|_Y$ and consider the pair $(Y, B_Y)$. 
We call $(Y, B_Y)$ a {\em{globally embedded simple normal 
crossing pair}} and $M$ the {\em{ambient space}} 
of $(Y, B_Y)$. A {\em{stratum}} of $(Y, B_Y)$ is a log canonical  
center of $(M, Y+B)$ that is contained in $Y$. 
\end{defn}

Let us recall the definition of {\em{quasi-log schemes}}. 

\begin{defn}[{Quasi-log 
schemes, see \cite[Definition 6.2.2]{fujino-foundations}}]\label{d-def4.2}
A {\em{quasi-log scheme}} is a scheme $X$ endowed with an 
$\mathbb R$-Cartier divisor 
(or $\mathbb R$-line bundle) 
$\omega$ on $X$, a closed subscheme 
$X_{-\infty}\subsetneq X$, and a finite collection $\{C\}$ of reduced 
and irreducible subschemes of $X$ such that there is a 
proper morphism $f\colon (Y, B_Y)\to X$ from a globally 
embedded simple 
normal crossing pair satisfying the following properties: 
\begin{itemize}
\item[(1)] $f^*\omega\sim_{\mathbb R}K_Y+B_Y$. 
\item[(2)] The natural map 
$\mathcal O_X
\to f_*\mathcal O_Y(\lceil -(B_Y^{<1})\rceil)$ 
induces an isomorphism 
\begin{equation*}
\mathcal I_{X_{-\infty}}\overset{\simeq}{\longrightarrow} 
f_*\mathcal O_Y(\lceil 
-(B_Y^{<1})\rceil-\lfloor B_Y^{>1}\rfloor),  
\end{equation*} 
where $\mathcal I_{X_{-\infty}}$ is the defining ideal sheaf of 
$X_{-\infty}$. 
\item[(3)] The collection of reduced and irreducible subschemes 
$\{C\}$ coincides with the images 
of the strata of $(Y, B_Y)$ that are not included in $X_{-\infty}$. 
\end{itemize}
We simply write $[X, \omega]$ to denote 
the above data 
\begin{equation*}
\left(X, \omega, f\colon (Y, B_Y)\to X\right)
\end{equation*} 
if there is no risk of confusion. 
Note that a quasi-log scheme $[X, \omega]$ is 
the union of $\{C\}$ and $X_{-\infty}$. 
The reduced and irreducible subschemes $C$ 
are called the {\em{qlc strata}} of $[X, \omega]$, 
$X_{-\infty}$ is called the {\em{non-qlc locus}} 
of $[X, \omega]$, and $f\colon (Y, B_Y)\to X$ is 
called a {\em{quasi-log resolution}} 
of $[X, \omega]$. 
We sometimes use $\Nqlc(X, 
\omega)$ or 
\begin{equation*}
\Nqlc(X, \omega, f\colon (Y, B_Y)\to X)
\end{equation*} 
to denote 
$X_{-\infty}$. 
If a qlc stratum $C$ of $[X, \omega]$ is not an 
irreducible component of $X$, then 
it is called a {\em{qlc center}} of $[X, \omega]$. 

We say that $(X, \omega, f\colon (Y, B_Y)\to X)$ or $[X, \omega]$ 
has a {\em{$\mathbb Q$-structure}} 
if $B_Y$ is a $\mathbb Q$-divisor, $\omega$ is a $\mathbb Q$-Cartier 
divisor 
(or $\mathbb Q$-line bundle), and 
$f^*\omega\sim _{\mathbb Q} K_Y+B_Y$ holds in 
the above definition. 
\end{defn}

In Definition \ref{d-def4.1}, we note that 
$f\colon Y\to X$ is not necessarily surjective and 
that $Y$ may be reducible even when $X$ is irreducible. 
In this paper, the notion of {\em{open qlc strata}} is indispensable. 

\begin{defn}[Open qlc strata]\label{d-def4.3}
Let $W$ be a qlc stratum of a quasi-log scheme $[X, \omega]$. 
We put 
\begin{equation*}
U:=W\setminus \left\{ \left(W\cap \Nqlc(X, \omega)\right)\cup 
\bigcup_{W'}W'\right\}, 
\end{equation*} 
where $W'$ runs over qlc centers of $[X, \omega]$ strictly contained 
in $W$, and call it the {\em{open qlc stratum}} of 
$[X, \omega]$ associated to $W$. 
\end{defn}

In Section \ref{k-sec11}, we need the notion of {\em{log bigness}}. 
For the details of relatively big $\mathbb R$-divisors, 
see \cite[Section 2.1]{fujino-foundations}.  

\begin{defn}[Log bigness]\label{d-def4.4}
Let $[X, \omega]$ be a quasi-log scheme and let 
$\pi\colon X\to S$ be a proper morphism 
between schemes. 
Let $D$ be an $\mathbb R$-Cartier divisor 
(or $\mathbb R$-line bundle) on $X$. 
We say that $D$ is {\em{log big over $S$ with respect 
to $[X, \omega]$}} if $D|_W$ is big over $\pi(W)$ for every 
qlc stratum $W$ of $[X,\omega]$. 
\end{defn}

We collect some basic and important properties of 
quasi-log schemes for the reader's convenience. 

\begin{thm}[{\cite[Theorem 6.3.4]{fujino-foundations}}]
\label{d-thm4.5}
In Definition \ref{d-def4.2}, 
we may assume that the ambient space $M$ of the 
globally embedded simple normal crossing pair $(Y, B_Y)$ is 
quasi-projective. 
In particular, $Y$ is quasi-projective and 
$f\colon Y\to X$ is projective. 
\end{thm}

For the details of Theorem \ref{d-thm4.5}, 
see the proof of \cite[Theorem 6.3.4]{fujino-foundations}. 
In the theory of quasi-log schemes, we sometimes 
need the projectivity of $f$ in order to use the theory 
of variations of mixed Hodge structure (see \cite{fujino-slc-trivial} and 
\cite{fujino-fujisawa-liu}). 
Hence Theorem \ref{d-thm4.5} plays a crucial role. 
The most important result in the theory of 
quasi-log schemes is as follows. 

\begin{thm}[{\cite[Theorem 6.3.5]{fujino-foundations}}]\label{d-thm4.6}
Let $[X, \omega]$ be a quasi-log scheme 
and let $X'$ be the union of $X_{-\infty}$ with 
a {\em{(}}possibly empty{\em{)}} union of some qlc strata of $[X, \omega]$. 
Then we have the following properties. 
\begin{itemize}
\item[(i)] {\em{(Adjunction).}}~Assume that 
$X'\ne X_{-\infty}$. Then $X'$ naturally becomes 
a quasi-log scheme with 
$\omega'=\omega|_{X'}$ and $X'_{-\infty}=X_{-\infty}$. 
Moreover, the qlc strata of $[X', \omega']$ 
are exactly the qlc strata of $[X, \omega]$ that are included in $X'$. 
\item[(ii)] {\em{(Vanishing theorem).}}~Assume that 
$\pi\colon X\to S$ is a proper morphism 
between schemes. Let $L$ be a Cartier divisor on $X$ such that 
$L-\omega$ is nef and log big over $S$ with 
respect to $[X, \omega]$. 
Then $R^i\pi_*(\mathcal I_{X'}\otimes \mathcal O_X(L))=0$ for 
every $i>0$, 
where $\mathcal I_{X'}$ is the defining ideal sheaf of $X'$ on $X$. 
\end{itemize}
\end{thm}

In this paper, we will repeatedly use adjunction for 
quasi-log schemes in Theorem \ref{d-thm4.6} (i). 
We strongly recommend the reader to see the 
proof of \cite[Theorem 6.3.5]{fujino-foundations}. 
Here, we only explain the main idea of the proof 
of Theorem \ref{d-thm4.6} (i) 
for the reader's convenience. 

\begin{proof}[Idea of Proof of Theorem \ref{d-thm4.6} (i)] 
By definition, $X'$ is the union of $X_{-\infty}$ with 
a union of some qlc strata of $[X, \omega]$ set theoretically. 
We assume that $X'\ne X_{-\infty}$ holds. 
By \cite[Proposition 6.3.1]{fujino-foundations}, 
we may assume that 
the union of all strata of $(Y, B_Y)$ mapped to $X'$ by $f$, which 
is denoted by $Y'$, is a union of some irreducible components of $Y$. 
We note that $Y$ is a simple normal crossing divisor on a smooth 
variety $M$ (see Definition \ref{d-def4.1}). 
We put $Y''=Y-Y'$, 
$K_{Y''}+B_{Y''}=(K_Y+B_Y)|_{Y''}$, and 
$K_{Y'}+B_{Y'}=(K_Y+B_Y)|_{Y'}$. We set $f''=f|_{Y''}$ and 
$f'=f|_{Y'}$. Then we claim that 
\begin{equation*}
\left(X', \omega', f'\colon (Y', B_{Y'})\to X'\right)
\end{equation*} 
becomes a quasi-log scheme satisfying the desired properties. 
Let us consider the 
following short exact sequence:   
\begin{equation*}
\begin{split}
0&\to \mathcal O_{Y''}(\lceil -(B^{<1}_{Y''})\rceil -\lfloor 
B^{>1}_{Y''}\rfloor-Y'|_{Y''})\to 
\mathcal O_Y(\lceil -(B^{<1}_Y)\rceil -\lfloor 
B^{>1}_Y\rfloor)
\\ &\to 
\mathcal O_{Y'}(\lceil -(B^{<1}_{Y'})\rceil -\lfloor 
B^{>1}_{Y'}\rfloor)\to 0, 
\end{split} 
\end{equation*} 
which is induced by 
\begin{equation*}
0\to \mathcal O_{Y''}(-Y'|_{Y''})\to \mathcal O_Y\to 
\mathcal O_{Y'}\to 0. 
\end{equation*}
We take the associated long exact sequence:  
\begin{equation*}
\begin{split}
0&\longrightarrow f''_*\mathcal O_{Y''}(\lceil -(B^{<1}_{Y''})\rceil -\lfloor 
B^{>1}_{Y''}\rfloor-Y'|_{Y''})\longrightarrow 
f_*\mathcal O_Y(\lceil -(B^{<1}_Y)\rceil -\lfloor 
B^{>1}_Y\rfloor)
\\ &\longrightarrow 
f'_*\mathcal O_{Y'}(\lceil -(B^{<1}_{Y'})\rceil -\lfloor 
B^{>1}_{Y'}\rfloor)\overset{\delta}{\longrightarrow} 
R^1f''_* \mathcal O_{Y''}(\lceil -(B^{<1}_{Y''})\rceil -\lfloor 
B^{>1}_{Y''}\rfloor-Y'|_{Y''}) \longrightarrow \cdots. 
\end{split} 
\end{equation*} 
Since 
\begin{equation*}
\begin{split}
\lceil -(B^{<1}_{Y''})\rceil -\lfloor B^{>1}_{Y''}\rfloor -Y'|_{Y''}
-(K_{Y''}+\{B_{Y''}\}+B^{=1}_{Y''}-Y'|_{Y''})
&=-(K_{Y''}+B_{Y''})\\ 
& \sim _{\mathbb R}-(f'')^*\omega, 
\end{split}
\end{equation*}
no associated prime of $R^1f''_*\mathcal O_{Y''}
(\lceil -(B^{<1}_{Y''})\rceil -\lfloor B^{>1}_{Y''}\rfloor -Y'|_{Y''})$ 
is contained in $X'$ by \cite[Theorem 5.6.2 (i)]{fujino-foundations}, 
which is 
a generalization of Koll\'ar's torsion-freeness based on the 
theory of mixed Hodge structures on cohomology with compact support 
(see \cite[Chapter 5]{fujino-foundations}). 
Then the connecting homomorphism 
\begin{equation*}
\delta\colon  f'_*\mathcal O_{Y'}(\lceil -(B^{<1}_{Y'})\rceil -\lfloor 
B^{>1}_{Y'}\rfloor)\to 
R^1f''_* \mathcal O_{Y''}(\lceil -(B^{<1}_{Y''})\rceil -\lfloor 
B^{>1}_{Y''}\rfloor-Y'|_{Y''})
\end{equation*} 
is zero since $f(Y')\subset X'$. 
We put 
\begin{equation*} 
\mathcal I_{X'}:=f''_* \mathcal O_{Y''}(\lceil -(B^{<1}_{Y''})\rceil -
\lfloor B^{>1}_{Y''}\rfloor-Y'|_{Y''}), 
\end{equation*} 
which is an ideal sheaf on $X$ since 
$\mathcal I_{X'}\subset \mathcal I_{X_{-\infty}}$, 
and define a scheme structure on $X'$ by $\mathcal I_{X'}$. 
Then we obtain the following big commutative diagram:  
\begin{equation*} 
\xymatrix{
& 0 \ar[d]& 0\ar[d] & &\\ 
0\ar[r]& f''_* \mathcal O_{Y''}(\lceil -(B^{<1}_{Y''})\rceil -\lfloor 
B^{>1}_{Y''}\rfloor-Y'|_{Y''})\ar[r]^-{=}\ar[d] &\mathcal I_{X'}\ar[d]&&
\\ 
0\ar[r]&f_*\mathcal O_Y(\lceil -(B^{<1}_Y)\rceil -\lfloor 
B^{>1}_Y\rfloor)=\mathcal I_{X_{-\infty}}\ar[r]\ar[d]
&\mathcal O_X\ar[d]\ar[r]&\mathcal O_{X_{-\infty}}\ar@{=}[d]
\ar[r]& 0\\ 
0\ar[r]&f'_* \mathcal O_{Y'}(\lceil -(B^{<1}_{Y'})\rceil -\lfloor 
B^{>1}_{Y'}\rfloor)=\mathcal I_{X'_{-\infty}}
\ar[d]\ar[r]&\mathcal O_{X'}\ar[d]\ar[r]&\mathcal O_{X'_{-\infty}}\ar[r]&0\\ 
&0&0&&
}
\end{equation*} 
by the above arguments. More precisely, 
by the above big commutative diagram, 
\begin{equation*} 
\mathcal I_{X'_{-\infty}}=f'_* \mathcal O_{Y'}(\lceil -(B^{<1}_{Y'})\rceil -\lfloor 
B^{>1}_{Y'}\rfloor)
\end{equation*} 
is an ideal sheaf on $X'$ such that $\mathcal O_X/\mathcal I_{X_{-\infty}}
=\mathcal O_{X'}/\mathcal I_{X'_{-\infty}}$. 
Thus we obtain that 
\begin{equation*} 
\left(X', \omega', f'\colon (Y', B_{Y'})\to X'\right)
\end{equation*}  
is a quasi-log scheme satisfying the desired properties. 
\end{proof}

As an obvious corollary, we have: 

\begin{cor}[{\cite[Notation 6.3.10]{fujino-foundations}}]\label{d-cor4.7} 
Let $[X, \omega]$ be a quasi-log scheme. 
The union of $X_{-\infty}$ with all 
qlc centers of $[X, \omega]$ is denoted by 
$\Nqklt(X, \omega)$, or, more precisely, 
\begin{equation*} 
\Nqklt(X, \omega, f\colon (Y, B_Y)\to X). 
\end{equation*}  
If $\Nqklt(X, \omega)\ne X_{-\infty}$, 
then 
\begin{equation*} 
[\Nqklt(X, \omega), \omega|_{\Nqklt(X, \omega)}]
\end{equation*}  
naturally becomes a quasi-log scheme by adjunction. 
\end{cor}

In the framework of quasi-log schemes, 
$\Nqklt(X, \omega)$ plays an important role 
by induction on dimension. 
When $\Nqklt(X, \omega)=\emptyset$, we have 
the following lemma. 

\begin{lem}[{\cite[Lemma 6.3.9]{fujino-foundations}}]\label{d-lem4.8}
Let $[X, \omega]$ be a quasi-log scheme with 
$X_{-\infty}=\emptyset$. 
Assume that every qlc stratum of $[X, \omega]$ 
is an irreducible component of $X$, equivalently, 
$\Nqklt(X, \omega)=\emptyset$. 
Then $X$ is normal. 
\end{lem}

For the proof of Lemma \ref{d-lem4.8}, 
see \cite[Lemma 6.3.9]{fujino-foundations}. 
It is convenient to introduce 
the notion of {\em{quasi-log canonical pairs}}. 

\begin{defn}[{Quasi-log canonical 
pairs, see \cite[Definition 6.2.9]{fujino-foundations}}]
\label{d-def4.9}
Let 
\begin{equation*} 
\left(X, \omega, f\colon (Y, B_Y)\to X\right)
\end{equation*}  
be a quasi-log scheme. 
If $X_{-\infty}=\emptyset$, then 
it is called a {\em{quasi-log canonical pair}} 
({\em{qlc pair}}, for short). 
\end{defn}

By using adjunction, we can prove: 
 
\begin{thm}[{\cite[Theorem 6.3.11 (i)]{fujino-foundations}}]\label{d-thm4.10} 
Let $[X, \omega]$ be a quasi-log canonical pair. 
Then the intersection of two qlc strata is a union of 
qlc strata. 
\end{thm}

The following example is very important. 
Example \ref{d-ex4.11} shows that we can treat log canonical 
pairs as quasi-log canonical pairs. In some sense, 
Ambro introduced the notion of quasi-log schemes 
in order to treat the following example (see \cite{ambro}). 

\begin{ex}[{\cite[6.4.1]{fujino-foundations}}]\label{d-ex4.11}
Let $(X, \Delta)$ be a normal pair such that 
$\Delta$ is effective. 
Let $f\colon Y\to X$ be a resolution of singularities such that 
\begin{equation*} 
K_Y+B_Y=f^*(K_X+\Delta)
\end{equation*}  
and that $\Supp B_Y$ is a simple normal crossing divisor on $Y$. 
We put $\omega:=K_X+\Delta$. 
Then $K_Y+B_Y\sim _{\mathbb R}f^*\omega$ holds. 
Since $\Delta$ is effective, $\lceil -(B^{<1}_Y)\rceil$ is effective 
and $f$-exceptional. 
Therefore, 
the natural map 
\begin{equation*} 
\mathcal O_X\to f_*\mathcal O_Y(\lceil -(B^{<1}_Y)\rceil)
\end{equation*}  
is an isomorphism. 
We put 
\begin{equation*} 
\mathcal I_{X_{-\infty}}:=\mathcal J_{\NLC}(X, \Delta)
=f_*\mathcal O_Y(\lceil -(B^{<1}_Y)\rceil-\lfloor B^{>1}_Y\rfloor),  
\end{equation*} 
where $\mathcal J_{\NLC}(X, \Delta)$ is the non-lc 
ideal sheaf associated to $(X, \Delta)$ in Definition \ref{c-def3.5}. 
We put $M=Y\times \mathbb C$ and $D=B_Y\times 
\mathbb C$. 
Then $(Y, B_Y)\simeq (Y\times \{0\}, B_Y\times 
\{0\})$ is a globally embedded simple normal crossing pair. 
Thus 
\begin{equation*} 
\left(X, \omega, f\colon(Y, B_Y)\to X\right)
\end{equation*} 
becomes a quasi-log scheme. 
By construction, $(X, \Delta)$ is log canonical 
if and only if $[X, \omega]$ is quasi-log canonical. 
We note that $C$ is a log canonical 
center of $(X, B)$ if and only if $C$ is a qlc center 
of $[X, \omega]$. We also note that 
$X$ itself is a qlc stratum of $[X, \omega]$. 
\end{ex}

We make a useful remark. 

\begin{rem}\label{d-rem4.12} 
Let $Y$ be a smooth variety and let $B_Y$ be 
an $\mathbb R$-divisor 
on $Y$ such that 
$\Supp B_Y$ is a simple normal crossing divisor on $Y$. 
We put $M':=Y\times \mathbb P^1$ and 
\begin{equation*}
D':=Y\times \{0\} +Y\times \{\infty\}+p^*B_Y, 
\end{equation*} 
where $p\colon Y\times \mathbb P^1\to Y$ is the 
first projection. 
Then $K_{M'}+D'=p^*(K_Y+B_Y)$ holds. 
We put 
\begin{equation*}
Z:=Y\times \{0\}+p^*B^{=1}_Y
\end{equation*} 
and 
$K_Z+B_Z:=(K_{M'}+D')|_Z$. 
Then $K_Z+B_Z=g^*(K_Y+B_Y)$ holds, 
where $g:=p|_Z\colon Z\to Y$. 
In this case, $(Z, B_Z)$ is a globally embedded 
simple normal crossing pair. 
We can check that 
\begin{equation*}
g_*\mathcal O_Z(\lceil -(B^{<1}_Z)\rceil -
\lfloor B^{>1}_Z\rfloor)\simeq 
\mathcal O_Y(\lceil -(B^{<1}_Y)\rceil -
\lfloor B^{>1}_Y\rfloor)
\end{equation*} 
holds since $g_*\mathcal O_Z\simeq \mathcal O_Y$ and 
\begin{equation*} 
B_Z=(D'-Z)|_Z=(Y\times \{ \infty \})|_{p^*B^{=1}_Y}+g^*B^{\ne 1}_Y, 
\end{equation*}
where $B^{\ne 1}_Y:=B_Y-B^{=1}_Y$. 
Hence, in Example \ref{d-ex4.11}, 
$f\circ g\colon (Z, B_Z)\to X$ gives another 
quasi-log resolution of $[X, \omega]$. 
Although $Z$ may be reducible, this quasi-log resolution 
is useful when we use 
adjunction (see Theorem \ref{d-thm4.6} (i) and 
\cite[Theorem 6.3.5 (i)]{fujino-foundations}). 
\end{rem}

Example \ref{d-ex4.11} shows that $[X, K_X+\Delta]$ has a natural 
quasi-log scheme structure. In general, however, 
$[X, K_X+\Delta]$ has many different quasi-log scheme structures. 

\begin{rem}\label{d-rem4.13}
In Example \ref{d-ex4.11}, we take an effective 
$\mathbb R$-divisor $\Delta'$ on $X$ such that 
$K_X+\Delta\sim _{\mathbb R} K_X+\Delta'$. 
Let $f'\colon Y'\to X$ be a resolution of singularities such that 
\begin{equation*} 
K_{Y'}+B_{Y'}=(f')^*(K_X+\Delta') 
\end{equation*}  
and that $\Supp B_{Y'}$ is a simple 
normal crossing divisor on $Y'$. 
Then 
\begin{equation*} 
\left(X, \omega, f'\colon  (Y', B_{Y'})\to X\right)
\end{equation*}  
is also a quasi-log scheme since $K_{Y'}+B_{Y'}\sim 
_{\mathbb R} (f')^*\omega$. 
In this case, there is no correspondence between 
qlc strata of $(X, \omega, f'\colon  (Y', B_{Y'})\to X)$ and lc strata of 
$(X, \Delta)$. 
\end{rem}

By combining Theorem \ref{d-thm4.10} with 
Example \ref{d-ex4.11}, we have:  

\begin{cor}[{\cite[Theorem 9.1 (2)]{fujino-fund}}]\label{d-cor4.14} 
Let $(X, \Delta)$ be a log canonical pair. 
Then the intersection of two lc centers is a union of 
lc centers. 
\end{cor}
 
For the basic properties of quasi-log schemes, 
see \cite[Chapter 6]{fujino-foundations} and 
\cite{fujino-on-quasi-log}. 
We also recommend the reader to see \cite{fujino-intro}, which is a 
gentle introduction to the theory of quasi-log schemes. 
In \cite{fujino-slc}, we establish that every quasi-projective 
semi-log canonical pair naturally becomes 
a quasi-log canonical pair. Hence 
we can use the theory of quasi-log schemes for the study 
of semi-log canonical pairs. 
For the details, see \cite{fujino-slc}. 

\subsection{Kleiman--Mori cones}\label{d-subsec4.2}
In this subsection, we discuss basic definitions and 
results around Kleiman--Mori cones of 
quasi-log schemes. Let us start with the definition of 
Kleiman--Mori cones. 

\begin{defn}[Kleiman--Mori cones]\label{d-def4.15}
Let $\pi\colon X\to S$ be a proper morphism between schemes. 
Let $NE(X/S)$ be the convex cone in $N_1(X/S)$ generated 
by effective $1$-cycles on $X$ mapped to points by 
$\pi$. 
Let $\NE(X/S)$ be the closure of $NE(X/S)$ in $N_1(X/S)$. 
We call it the {\em{Kleiman--Mori cone}} of $\pi\colon X\to S$. 
As usual, we drop $/\Spec \mathbb C$ from the notation 
when $S=\Spec\mathbb C$. 
\end{defn}

Let us explain some basic definitions. 

\begin{defn}[{\cite[Definition 6.7.1]{fujino-foundations}}]\label{d-def4.16}
Let $[X, \omega]$ be a quasi-log scheme with the 
non-qlc locus $X_{-\infty}$. 
Let $\pi\colon X\to S$ be a projective 
morphism between schemes. 
We put 
\begin{equation*} 
\NE(X/S)_{-\infty}=
\mathrm{Im}\left(\NE(X_{-\infty}/S)\to 
\NE(X/S)\right). 
\end{equation*}  
We sometimes use $\NE(X/S)_{\Nqlc(X/S)}$ 
to denote $\NE(X/S)_{-\infty}$. 
For an $\mathbb R$-Cartier divisor 
(or $\mathbb R$-line bundle) $D$, we 
define 
\begin{equation*} 
D_{\geq 0} =\{z \in N_1(X/S) \, | \, D\cdot z\geq 0\}. 
\end{equation*}  
Similarly, we can define $D_{>0}$, 
$D_{\leq 0}$, and $D_{<0}$. 
We also define 
\begin{equation*} 
D^\perp =\{ z\in N_1(X/S) \, | \, D\cdot z=0\}. 
\end{equation*}  
We use the following notation 
\begin{equation*} 
\NE(X/S)_{D\geq 0}=\NE(X/S) 
\cap D_{\geq 0}, 
\end{equation*} 
and similarly for $>0$, $\leq 0$, and $<0$. 
\end{defn}

In order to treat the cone and contraction theorem, 
we need the following definition. 

\begin{defn}[{\cite[Definition 6.7.2]{fujino-foundations}}]\label{d-def4.17} 
An {\em{extremal face}} 
of $\NE(X/S)$ is a non-zero 
subcone $F\subset \NE(X/S)$ such that 
$z, z'\in \NE(X/S)$ and $z+z'\in F$ imply that $z, z'\in F$. 
Equivalently, $F=\NE(X/S)\cap H^\perp$ for 
some $\pi$-nef $\mathbb R$-divisor 
(or $\pi$-nef $\mathbb R$-line bundle) $H$, which 
is called a {\em{support function}} 
of $F$. An {\em{extremal ray}} is a one-dimensional 
extremal face. 
\begin{itemize}
\item[(1)] An extremal face $F$ is called 
{\em{$\omega$-negative}} if $F\cap 
\NE(X/S)_{\omega\geq 0}=\{0\}$. 
\item[(2)] An extremal face $F$ is called {\em{rational}} 
if we can choose a $\pi$-nef 
$\mathbb Q$-divisor (or $\mathbb Q$-line bundle) 
$H$ as a support function of $F$. 
\item[(3)] An extremal face $F$ is called {\em{relatively 
ample at infinity}} if $F\cap \NE(X/S)_{-\infty}
=\{0\}$. 
Equivalently, $H|_{X_{-\infty}}$ is $\pi|_{X_{-\infty}}$-ample 
for any supporting function 
$H$ of $F$. 
\end{itemize}
\end{defn}

The contraction theorem for quasi-log schemes 
plays an important role 
in this paper. 

\begin{thm}[{Contraction theorem, see 
\cite[Theorem 6.7.3]{fujino-foundations}}]
\label{d-thm4.18}
Let $[X, \omega]$ be a quasi-log scheme and let 
$\pi\colon X\to S$ be a projective morphism between schemes. 
Let $R$ be an $\omega$-negative extremal 
ray of $\NE(X/S)$ that is 
rational and relatively ample at infinity. 
Then there exists a projective morphism 
$\varphi_R\colon  X\to Y$ over $S$ with 
the following properties. 
\begin{itemize}
\item[(i)] Let $C$ be an integral curve on $X$ such that 
$\pi(C)$ is a point. 
Then $\varphi_R(C)$ is a point if and only if 
$[C]\in R$, where $[C]$ denotes the numerical 
equivalence class of $C$ in $N_1(X/S)$. 
\item[(ii)] $\mathcal O_Y\simeq (\varphi_R)_*\mathcal O_X$. 
\item[(iii)] Let $\mathcal L$ be a line bundle 
on $X$ such that $\mathcal L\cdot C=0$ for every 
curve $C$ with $[C]\in R$. 
Then there is a line bundle $\mathcal L_Y$ on $Y$ such 
that $\mathcal L\simeq \varphi^*_R\mathcal L_Y$. 
\end{itemize}
\end{thm}

\begin{proof}
Since $R$ is relatively ample at infinity, $\varphi_R\colon  X_{-\infty}
\to \varphi_R(X_{-\infty})$ is finite. 
Hence $\mathcal L^{\otimes m}|_{X_{-\infty}}$ 
is $\varphi_R|_{X_{-\infty}}$-generated for 
every $m\geq 0$. 
Therefore, this theorem is a 
special case of \cite[Theorem 6.7.3]{fujino-foundations}. 
\end{proof}

Theorem \ref{d-thm4.18} 
is a generalization of the famous Kawamata--Shokurov 
basepoint-free theorem. 

\subsection{Lemmas on quasi-log schemes}\label{d-subsec4.3}

In this subsection, we treat useful lemmas on quasi-log schemes. 
The first two lemmas were already proved in \cite{fujino-reid-fukuda}. 
We will repeatedly use Lemma \ref{d-lem4.20} throughout this paper. 

\begin{lem}[{\cite[Lemma 3.12]{fujino-reid-fukuda}}]\label{d-lem4.19}
Let $\left(X, \omega, f\colon (Y, B_Y)\to X\right)$ be a quasi-log 
scheme. 
Then we can construct a proper morphism 
$f'\colon (Y', B_{Y'})\to X$ from a globally embedded simple 
normal crossing pair $(Y', B_{Y'})$ such that 
\begin{itemize}
\item[(i)] $f'\colon  (Y', B_{Y'})\to X$ gives the same 
quasi-log scheme structure as one given by $f\colon (Y, B_Y)\to X$, and 
\item[(ii)] every irreducible component of $Y'$ is 
mapped by $f'$ to $\overline {X\setminus X_{-\infty}}$, 
the closure of $X\setminus X_{-\infty}$ in $X$. 
\end{itemize}
\end{lem}
We give the proof for the sake of completeness.
\begin{proof} 
Let $Y''$ be the union of all irreducible components 
of $Y$ that are not mapped to $\overline{X\setminus X_{-\infty}}$. 
We put $Y'=Y-Y''$ and $K_{Y''}+B_{Y''}=(K_Y+B_Y)|_{Y''}$. 
Let $M$ be the ambient space of $(Y, B_Y)$. By taking some 
blow-ups of $M$, we may assume that the union of 
all strata of $(Y, B_Y)$ mapped to 
$\overline {X\setminus X_{-\infty}}
\cap X_{-\infty}$ is a union of some irreducible components of $Y$ 
(see \cite[Proposition 6.3.1]{fujino-foundations}). 
We consider the short exact sequence 
\begin{equation*} 
0\to \mathcal O_{Y''}(-Y')\to \mathcal O_Y\to \mathcal O_{Y'}\to 0. 
\end{equation*}  
We put $A=\lceil -(B_Y^{<1})\rceil$ and $N=\lfloor B_Y^{>1}\rfloor$. By 
applying $\otimes \mathcal O_Y(A-N)$, we have 
\begin{equation*} 
0\to \mathcal O_{Y''}(A-N-Y')\to \mathcal O_Y(A-N)\to \mathcal O_{Y'}
(A-N)\to 0. 
\end{equation*} 
By taking $R^if_*$, we obtain 
\begin{align*}
0&\to f_* \mathcal O_{Y''}(A-N-Y')\to f_*\mathcal O_Y(A-N)\to f_* \mathcal O_{Y'}
(A-N)\\
&\to R^1f_*\mathcal O_{Y''}(A-N-Y')\to \cdots.
\end{align*} 
Note that 
\begin{align*}
(A-N-Y')|_{Y''}-(K_{Y''}+\{B_{Y''}\}+B_{Y''}^{=1}
-Y'|_{Y''})
&=-(K_{Y''}+B_{Y''})
\\&\sim _{\mathbb R} -(f^*\omega)|_{Y''}. 
\end{align*}
Hence, by \cite[Theorem 5.6.2]{fujino-foundations}, 
no associated prime of $R^1f_*\mathcal O_{Y''}(A-N-Y')$ is contained 
in $f(Y')\cap X_{-\infty}$. 
Therefore, the connecting homomorphism 
\begin{equation*} 
\delta\colon f_*\mathcal O_{Y'}(A-N)\to 
R^1f_*\mathcal O_{Y''}(A-N-Y')
\end{equation*} 
is zero. 
This implies that 
\begin{equation*} 
0\to f_*\mathcal O_{Y''}(A-N-Y')\to \mathcal I_{X_{-\infty}}\to f_*\mathcal O_{Y'}
(A-N)\to 0 
\end{equation*}  
is exact. 
The ideal sheaf $\mathcal J=f_*\mathcal O_{Y''}
(A-N-Y')$ is zero 
when it is restricted to $X_{-\infty}$ because 
$\mathcal J\subset \mathcal I_{X_{-\infty}}$. 
On the other hand, $\mathcal J$ is zero on 
$X\setminus X_{-\infty}$ because 
$f(Y'')\subset X_{-\infty}$. 
Therefore, we obtain $\mathcal J=0$. 
Thus we have $\mathcal I_{X_{-\infty}}=f_*\mathcal O_{Y'}(A-N)$. 
So $f'=f|_{Y'}\colon (Y', B_{Y'})\to X$, where 
$K_{Y'}+B_{Y'}=(K_Y+B_Y)|_{Y'}$, 
gives the same quasi-log scheme structure as one 
given by $f\colon (Y, B_Y)\to X$ with the property (ii). 
\end{proof}

By using Lemma \ref{d-lem4.19}, we establish the following very 
useful lemma. 

\begin{lem}[{\cite[Lemma 3.14]{fujino-reid-fukuda}}]\label{d-lem4.20}
Let $[X, \omega]$ be a quasi-log 
scheme. 
Let us consider 
$X^\dag=\overline {X\setminus X_{-\infty}}$, the closure 
in $X$, 
with the reduced scheme structure. 
Then $[X^\dag, \omega^\dag]$, where 
$\omega^\dag=\omega|_{X^\dag}$, has a 
natural quasi-log scheme structure induced by $[X, \omega]$. 
This means that 
\begin{itemize}
\item[(i)] $C$ is a qlc stratum of $[X, \omega]$ if and 
only if $C$ is a qlc stratum of $[X^\dag, \omega^\dag]$, and 
\item[(ii)] $\mathcal I_{\Nqlc(X, \omega)}=\mathcal I_{\Nqlc(X^\dag, 
\omega^\dag)}$ holds. 
\end{itemize} 
Moreover, we consider a set of some qlc strata $\{C_i\}_{i\in I}$ of $[X, \omega]$. 
We put 
\begin{equation*} 
\left(X^\dag\right)'=\Nqlc(X^\dag, \omega^\dag)\cup 
\left(\bigcup _{i\in I} C_i\right)
\end{equation*} 
and 
\begin{equation*}
X'=\Nqlc(X, \omega)\cup 
\left(\bigcup _{i\in I} C_i\right). 
\end{equation*}  
Then $\left[(X^\dag)', \omega^\dag|_{(X^\dag)'}\right]$ and 
$\left[X', \omega|_{X'}\right]$ 
naturally become quasi-log schemes by adjunction and 
$\mathcal I_{(X^\dag)'}=\mathcal I_{X'}$ holds, where 
$\mathcal I_{(X^\dag)'}$ and $\mathcal I_{X'}$ are the 
defining ideal sheaves of $(X^\dag)'$ and $X'$ on 
$X^\dag$ and $X$, respectively. 
In particular, $\mathcal I_{\Nqklt(X^\dag, \omega^\dag)}
=\mathcal I_{\Nqklt(X, \omega)}$ holds. 
\end{lem}

We include the proof for the benefit of the reader.

\begin{proof} 
In this proof, we use the same notation as in the proof of 
Lemma \ref{d-lem4.19}. 
Let $\mathcal I_{X^\dag}$ be the defining ideal sheaf of $X^\dag$ on $X$. 
Let $f'\colon (Y', B_{Y'})\to X$ be the 
quasi-log resolution constructed in the proof of 
Lemma \ref{d-lem4.19}. 
By construction, $f'\colon Y'\to X$ factors through $X^\dag$.  
Note that 
\begin{equation*} 
\mathcal I_{X_{-\infty}}=f_*\mathcal O_Y(A-N)=f'_*
\mathcal O_{Y'}(A-N)= f'_*\mathcal O_{Y'}(-N)
\end{equation*} 
and that 
\begin{equation*} 
f'(N)=X_{-\infty}\cap f'(Y')=X_{-\infty}\cap X^\dag
\end{equation*}  
set theoretically, where 
$A=\lceil -(B_{Y}^{<1})\rceil$ and $N=\lfloor B_{Y}^{>1}\rfloor$ 
(see \cite[Remark 6.2.10]{fujino-foundations}). 
Therefore, we obtain 
\begin{equation*} 
\mathcal I_{X^\dag}\cap \mathcal I_{X_{-\infty}}
=\mathcal I_{X^\dag}\cap f_*\mathcal O_Y(A-N)
\subset f_*\mathcal O_{Y''}(A-N-Y')=\{0\}.
\end{equation*}  
Thus we can construct the following big commutative diagram. 
\begin{equation*} 
\xymatrix{&&0\ar[d]&0\ar[d]&
\\&&f'_*\mathcal O_{Y'}(A-N)\ar@{=}[r]\ar[d]& f'_*\mathcal O_{Y'}(A-N)\ar[d]&\\
0\ar[r]&\mathcal I_{X^\dag}\ar[r]\ar@{=}[d]&\mathcal O_X\ar[r]\ar[d]
&\mathcal O_{X^\dag}\ar[r]\ar[d]&0\\
0\ar[r]&\mathcal I_{X^\dag}\ar[r]&\mathcal O_{X_{-\infty}}\ar[r]\ar[d]
&\mathcal O_{X^\dag_
{-\infty}}\ar[d]\ar[r]&0\\
&&0&0&
}
\end{equation*} 
Hence $f'\colon (Y', B_{Y'})\to X^\dag$ gives the desired quasi-log scheme 
structure 
on $[X^\dag, \omega^\dag]$. 

We know 
that $[(X^\dag)', \omega^{\dag}|_{(X^\dag)'}]$ and 
$[X', \omega|_{X'}]$ naturally become quasi-log schemes 
by adjunction (see Theorem \ref{d-thm4.6} (i) and 
\cite[Theorem 6.3.5 (i)]{fujino-foundations}). 
Thus it is sufficient to prove the equality $\mathcal I_{(X^\dag)'}=
\mathcal I_{X'}$. As usual, by \cite[Proposition 6.3.1]{fujino-foundations}, 
we may further assume that the union of all strata of 
$(Y, B_Y)$ that are mapped to $X'$, which is denoted by 
$Z$, is a union of some irreducible components of $Y$. 
We note that $Z\geq Y''$. 
We put $Z'=Y-Z$. 
Then it is obvious that $Z'\leq Y'$ holds. 
By the proof of adjunction (see the idea of the 
proof of Theorem \ref{d-thm4.6} (i) and the proof of 
\cite[Theorem 6.3.5 (i)]{fujino-foundations}), 
we see that 
\begin{equation*} 
\mathcal I_{(X^\dag)'}=f'_*\mathcal O_{Z'}(A-N-(Z-Y'')|_{Z'})
=f'_*\mathcal O_{Z'}(A-N-Z|_{Z'})
=\mathcal I_{X'}
\end{equation*}  
holds. 
\end{proof}

By Lemmas \ref{d-lem4.19} and \ref{d-lem4.20}, 
we can abandon unnecessary components 
from $f\colon (Y, B_Y)\to X$. 
The following examples may help the reader 
understand Lemmas \ref{d-lem4.19} and \ref{d-lem4.20}. 

\begin{ex}\label{d-ex4.21} 
Let $L$ be a line on $\mathbb P^3$. 
We take general hyperplanes $H_i$ with $L\subset H_i$ 
for $1\leq i\leq 4$. 
Let $H_0$ be a general hyperplane of $\mathbb P^3$. 
We put $X:=\mathbb P^3$ and 
\begin{equation*}
\Delta:=H_0+\frac{1}{2}H_1+\frac{1}{2}H_2+\frac{2}{3}H_3+\frac{2}{3}H_4  
\end{equation*}
and consider the normal pair $(X, \Delta)$. 
Then the pair $[X, \omega:=K_X+\Delta]$ naturally becomes 
a quasi-log scheme by Example \ref{d-ex4.11}. 
We can easily check that $\Nqlc(X, \omega)=\Nlc(X, \Delta)=L$. 
Let $p\colon X^\flat\to X$ be the blow-up along $L$. 
Then we have 
\begin{equation*}
K_{X^\flat}+p^{-1}_*\Delta+\frac{4}{3}E=p^*(K_X+\Delta), 
\end{equation*} 
where $E$ is the $p$-exceptional divisor on $X^\flat$. 
By construction, the support of $p^{-1}_*\Delta+E$ is a simple 
normal crossing divisor on $X^\flat$. 
We put $X':=H_0\cup L$ and $\omega':=\omega|_{X'}$. 
By adjunction (see Theorem \ref{d-thm4.6} (i) and 
\cite[Theorem 6.3.5 (i)]{fujino-foundations}), 
$[X', \omega']$ naturally has a quasi-log scheme structure 
induced by $[X, \omega]$. More precisely, 
by using $p\colon X^\flat\to X$, 
we can construct a quasi-log scheme 
\begin{equation*}
\left(X', \omega', f\colon (Y, B_Y)\to X'\right)
\end{equation*} 
such that $Y$ is irreducible (see also Remark \ref{d-rem4.12}). 
In this case, $f\colon Y\to X'$ is not surjective. 
Let $q\colon X^\sharp \to X^\flat$ be the blow-up 
along $E\cap H^\flat_4$, 
where $H^\flat_4$ is the strict transform of $H_4$ on $X^\flat$. 
Then we have 
\begin{equation*}
K_{X^\sharp} +(p\circ q) ^{-1}_* \Delta+\frac{4}{3} E^\sharp 
+F 
=(p\circ q)^*(K_X+\Delta), 
\end{equation*} 
where $E^\sharp$ is the strict transform of $E$ on $X^\sharp$ and 
$F$ is the $q$-exceptional divisor on $X^\sharp$. 
By construction, the support of 
$(p\circ q)^{-1}_*\Delta+E^\sharp+F$ 
is a 
simple normal crossing divisor on $X^\sharp$. 
We can easily check that $p\circ q(F)=L$. 
By using $p\circ q\colon X^\sharp \to X$, 
we can construct another quasi-log resolution 
$\tilde{f} \colon (\tilde Y, B_{\tilde Y})\to X'$ of 
$[X', \omega']$ such that 
$\tilde f\colon \tilde Y\to X'$ is surjective. 
In particular, $\tilde Y$ is reducible and there exists an 
irreducible component of $\tilde Y$ which is dominant onto $L$.  
\end{ex}

\begin{ex}\label{d-ex4.22} 
Let $M$ be a smooth variety and let $D_1$ and $D_2$ be 
prime divisors on $M$ such that 
$D_1+D_2$ is a simple normal crossing divisor with $D_1\ne 
D_2$ and $D_1\cap D_2\ne \emptyset$.  
We consider the normal pair $(M, D_1+2D_2)$. 
Then the pair $[M, K_M+D_1+2D_2]$ 
naturally becomes a quasi-log scheme as explained in 
Example \ref{d-ex4.11}. 
We put $X:=D_1+2D_2$. 
Then, by adjunction (see Theorem \ref{d-thm4.6} (i) 
and 
\cite[Theorem 6.3.5 (i)]{fujino-foundations}), 
$[X, \omega]$ 
is a quasi-log scheme, where $\omega:=(K_M+D_1+2D_2)|_X$. 
More precisely, we put $Y:=D_1$ and consider 
$K_Y+B_Y:=(K_M+D_1+2D_2)|_Y$. Then 
$f\colon (Y, B_Y)\to X$ gives a natural 
quasi-log scheme structure on $[X, \omega]$ by adjunction, 
where $f\colon Y\to X$ is a natural closed embedding. We 
note that $f\colon Y\to X$ is not surjective in this case. 
We put $X^\dag:=D_1$ and consider 
$\omega^\dag:=(K_M+D_1+2D_2)|_{X^\dag}$. 
Then $[X^\dag, \omega^\dag]$ has a 
natural quasi-log scheme structure. 
We can see $f'\colon (Y, B_Y)\to X^\dag$ as a quasi-log 
resolution of $[X^\dag, \omega^\dag]$, where $f'$ is 
the identity morphism of $Y=X^\dag$.  
We note that 
$
\mathcal I_{\Nqlc(X^\dag, \omega^\dag)}
=\mathcal I_{\Nqlc(X, \omega)}
$ obviously 
holds. 
\end{ex}

Lemma \ref{d-lem4.23} is almost obvious by definition. 

\begin{lem}\label{d-lem4.23}
Let 
\begin{equation*} 
\left(X, \omega, f\colon  (Y, B_Y)\to X\right)
\end{equation*} 
be 
a quasi-log scheme and let 
$B$ be an effective 
$\mathbb R$-Cartier divisor on $X$, 
that is, a finite $\mathbb R_{>0}$-linear 
combination of effective Cartier divisors on $X$. 
Let $X'$ be the union of $\Nqlc(X, \omega)$ and all qlc centers of $[X, \omega]$ 
contained in $\Supp B$. 
Assume that 
the union of all strata of $(Y, B_Y)$ mapped to $X'$ by $f$, which is 
denoted by $Y'$, is a union of some irreducible components 
of $Y$. 
We put $Y''=Y-Y'$, $K_{Y''}+B_{Y''}=
(K_Y+B_Y)|_{Y''}$, and $f''=f|_{Y''}$. 
We further assume that 
\begin{equation*} 
\left(Y'', B_{Y''}+(f'')^*B\right)
\end{equation*}  
is a globally embedded simple normal crossing pair. 
Then 
\begin{equation*} 
\left(X, \omega+B, f''\colon  (Y'', B_{Y''}+(f'')^*B)\to X\right)
\end{equation*}  
is a quasi-log scheme. 
\end{lem}
\begin{proof}
Since $K_Y+B_Y\sim _{\mathbb R} f^*\omega$, we have 
$K_{Y''}+B_{Y''}\sim _{\mathbb R} (f'')^*\omega$. 
Therefore, $K_{Y''}+B_{Y''}+(f'')^*B\sim 
_{\mathbb R}(f'')^*(\omega+B)$ holds true. 
By the proof of adjunction (see 
the idea of the proof of 
Theorem \ref{d-thm4.6} (i) and 
the proof of \cite[Theorem 6.3.5 (i)]{fujino-foundations}), 
we have 
\begin{equation*} 
\mathcal I_{X'}=f''_*\mathcal O_{X''}(\lceil -(B^{<1}_{Y''})\rceil 
-\lfloor B^{>1}_{Y''}\rfloor-Y'|_{Y''}), 
\end{equation*}  
where $\mathcal I_{X'}$ is the defining ideal sheaf of $X'$ on $X$. 
Note that the following key inequality 
\begin{equation*} 
\lceil -(B_{Y''}+(f'')^*B)^{<1}\rceil -\lfloor 
(B_{Y''}+(f'')^*B)^{>1}\rfloor 
\leq \lceil -(B^{<1}_{Y''})\rceil -\lfloor 
B^{>1}_{Y''}\rfloor -Y'|_{Y''}
\end{equation*} 
holds. 
Therefore, we put 
\begin{equation*} 
\mathcal I_{\Nqlc(X, \omega+B)}:= 
f''_*\mathcal O_{Y''}(\lceil -(B_{Y''}+(f'')^*B)^{<1}\rceil -\lfloor 
(B_{Y''}+(f'')^*B)^{>1}\rfloor) \subset \mathcal I_{X'}\subset 
\mathcal O_X
\end{equation*}  
and define the closed subscheme $\Nqlc(X, \omega+B)$ of $X$ 
by $\mathcal I_{\Nqlc(X, \omega+B)}$. 
Then 
\begin{equation*} 
\left(X, \omega+B, f''\colon  (Y'', B_{Y''}+(f'')^*B)\to X\right)
\end{equation*}  
is a quasi-log scheme. 
Let $W$ be a reduced and irreducible subscheme of $X$. 
As usual, we say that $W$ is a qlc stratum of $[X, \omega+B]$ 
when $W$ is not contained in $\Nqlc(X, \omega+B)$ and is 
the $f''$-image of a stratum of $(Y'', B_{Y''}+(f'')^*B)$. 
By construction, we have $X'\subset \Nqlc(X, \omega+B)$. 
We note that 
$\left(X, \omega+B, f''\colon  (Y'', B_{Y''}+(f'')^*B)\to X\right)$ coincides 
with $\left(X, \omega, f\colon (Y, B_Y)\to X\right)$ outside $\Supp B$. 
\end{proof}

By using Lemma \ref{d-lem4.23}, 
we can prove the following lemma. 

\begin{lem}\label{d-lem4.24}
Let $[X, \omega]$ be a quasi-log scheme and let $G$ 
be an effective $\mathbb R$-Cartier divisor on $X$, that is, 
a finite $\mathbb R_{>0}$-linear combination 
of effective Cartier divisors on $X$. 
Then, for every $0<\varepsilon \ll 1$, 
$[X, \omega+\varepsilon G]$ naturally becomes a 
quasi-log scheme such that $\Nqklt(X, \omega+\varepsilon G)
=\Nqklt(X, \omega)$ holds. 
More precisely, $
\mathcal I_{\Nqklt(X, \omega+\varepsilon G)}=\mathcal I_{\Nqklt(X, \omega)}$ 
holds. 
\end{lem}

Note that Lemma \ref{d-lem4.24} is almost obvious 
for normal pairs by the definition of multiplier ideal 
sheaves. 

\begin{proof}[Proof of Lemma \ref{d-lem4.24}]
Let $f\colon (Y, B_Y)\to X$ be a proper morphism 
from a globally embedded simple normal crossing  
pair $(Y, B_Y)$ as in Definition \ref{d-def4.2}. 
Let $X'$ be the union of $\Nqlc(X, \omega)$ and 
all qlc centers of $[X, \omega]$ contained 
in $\Supp G$. 
By \cite[Proposition 6.3.1]{fujino-foundations} and 
\cite[Theorem 3.35]{kollar}, 
we may assume that the union of all strata of $(Y, B_Y)$ 
mapped to $X'$ by $f$, which is 
denoted by $Y'$, is a union of some irreducible 
components of $Y$. 
By \cite[Proposition 6.3.1]{fujino-foundations} and 
\cite[Theorem 3.35]{kollar} again, 
we may further assume that 
the union of all strata of $(Y, B_Y)$ mapped to 
$\Nqklt(X, \omega)$ by $f$, which 
is denoted by $Z'$, is a union of 
some irreducible components of $Y$. 
By construction, 
$Y'\leq Z'$ obviously holds. 
As in Lemma \ref{d-lem4.23}, 
we put $Y''=Y-Y'$, $K_{Y''}+B_{Y''}=(K_Y+B_Y)|_{Y''}$, 
and $f''=f|_{Y''}$. 
By \cite[Proposition 6.3.1]{fujino-foundations} and 
\cite[Theorem 3.35]{kollar}, 
we further assume that 
$(Y'', (f'')^*G+\Supp B_{Y''})$ is a globally embedded 
simple normal crossing pair. 
By Lemma \ref{d-lem4.23}, we know that 
\begin{equation*} 
\left(X, \omega+\varepsilon G, f''\colon  
(Y'', B_{Y''}+\varepsilon (f'')^*G)\to X\right)
\end{equation*}  
is a quasi-log scheme for every $\varepsilon >0$. 
We put $Z''=Y-Z'$, $K_{Z''}+B_{Z''}=(K_Y+B_Y)|_{Z''}$, and 
$h=f|_{Z''}$. 
Thus, by the proof of adjunction (see the idea of the proof of 
Theorem \ref{d-thm4.6} (i) and 
the proof of 
\cite[Theorem 6.3.5 (i)]{fujino-foundations}), we have 
\begin{equation*} 
\mathcal I_{\Nqklt(X, \omega)}=h_*\mathcal O_{Z''} 
(\lceil -(B^{<1}_{Z''})\rceil -\lfloor 
B^{>1}_{Z''}\rfloor -Z'|_{Z''}). 
\end{equation*}  
We note that 
\begin{equation*} 
\lceil -(B^{<1}_{Z''})\rceil -\lfloor 
B^{>1}_{Z''}\rfloor -Z'|_{Z''}=\lfloor B_{Z''}\rfloor 
\end{equation*}  
holds by definition. 
On the other hand, 
by the proof of adjunction again 
(see the idea of the proof of 
Theorem \ref{d-thm4.6} (i) and 
the proof of \cite[Theorem 6.3.5 (i)]{fujino-foundations}), 
\begin{equation*} 
\mathcal I_{\Nqklt(X, \omega+\varepsilon G)}
=h_*\mathcal O_{Z''} 
(\lceil -(B_{Z''}+\varepsilon h^*G)^{<1}\rceil -\lfloor 
(B_{Z''}+\varepsilon h^*G)^{>1}\rfloor -(Z'-Y')|_{Z''}) 
\end{equation*} 
for every $0<\varepsilon \ll 1$. 
By direct calculation, for $0<\varepsilon \ll1$, 
\begin{equation*}
\begin{split}
&\lceil -(B_{Z''}+\varepsilon h^*G)^{<1}\rceil -\lfloor 
(B_{Z''}+\varepsilon h^*G)^{>1}\rfloor -(Z'-Y')|_{Z''}
\\&= -\lfloor B_{Z''}\rfloor \\ 
&=\lceil -(B^{<1}_{Z''})\rceil -\lfloor 
B^{>1}_{Z''}\rfloor -Z'|_{Z''}. 
\end{split}
\end{equation*}
Hence we obtain 
\begin{equation*} 
\mathcal I_{\Nqklt(X, \omega+\varepsilon G)}=\mathcal I_{\Nqklt(X, \omega)}. 
\end{equation*}  
This means that 
\begin{equation*} 
\left(X, \omega+\varepsilon G, f''\colon  (Y'', B_{Y''}+\varepsilon 
(f'')^*G)\to X\right)
\end{equation*}  
is a quasi-log scheme with 
\begin{equation*} 
\Nqklt(X, \omega+\varepsilon G)=\Nqklt(X, \omega) 
\end{equation*} 
for 
$0<\varepsilon \ll 1$. 
We finish the proof of Lemma \ref{d-lem4.24}. 
\end{proof}

We need the following lemma in order to 
reduce some problems to the case where 
quasi-log schemes have $\mathbb Q$-structures. 

\begin{lem}\label{d-lem4.25}
Let $\left(X, \omega, f\colon  (Y, B_Y)\to X\right)$ be a quasi-log scheme. 
Then we obtain a $\mathbb Q$-divisor $D_i$ on $Y$, 
a $\mathbb Q$-line bundle $\omega_i$ on $X$, and a positive 
real number $r_i$ for $1\leq i\leq k$ such that 
\begin{itemize}
\item[(i)] $\sum _{i=1}^k r_i=1$, 
\item[(ii)] $\Supp D_i=\Supp B_Y$, 
$D^{=1}_i=B^{=1}_Y$, $\lfloor D^{>1}_i\rfloor=
\lfloor B^{>1}_Y\rfloor$, 
and $\lceil -(D^{<1}_i)\rceil=\lceil -(B^{<1}_Y)\rceil$ 
for every $i$, 
\item[(iii)] $\omega=\sum _{i=1}^k r_i\omega_i$ and 
$B_Y=\sum _{i=1}^k r_i D_i$, and 
\item[(iv)] $\left(X, \omega_i, f\colon (Y, D_i)\to X\right)$ is a quasi-log 
scheme with $K_Y+D_i\sim _{\mathbb Q} f^*\omega_i$ for every $i$. 
\end{itemize}
We note that 
\begin{equation*} 
\Nqlc(X, \omega_i)=\Nqlc(X, \omega)
\end{equation*}  
holds for every $i$. 
We also note that $W$ is a qlc stratum of 
$[X, \omega]$ if and only if 
$W$ is a qlc stratum of $[X, \omega_i]$ for 
every $i$. 
\end{lem}
\begin{proof}
Without loss of generality, we may assume that 
$\omega$ is an $\mathbb R$-line bundle. 
We put $B_Y=\sum _j b_j B_j$, where $B_j$ is a simple normal crossing divisor 
on $Y$ for every $j$, $b_{j_1}\ne b_{j_2}$ for $j_1\ne j_2$, and 
$\Supp B_{j_1}$ and $\Supp B_{j_2}$ have no common irreducible components 
for $j_1\ne j_2$. We may assume that $b_j\in \mathbb R\setminus \mathbb Q$ 
for $1\leq j\leq l$ and 
$b_j\in \mathbb Q$ for $j\geq l+1$. 
We put $\omega=\sum _{p=1}^m a_p \omega_p$, where 
$a_p\in \mathbb R$ and $\omega_p$ is a line bundle 
on $X$ for every $p$. 
We can write 
\begin{equation*} 
K_Y+B_Y=\sum _{p=1}^m a_p f^*\omega_p
\end{equation*}  
in $\Pic(Y)\otimes _{\mathbb Z}\mathbb R$. 
We consider the following linear map 
\begin{equation*} 
\psi\colon  \mathbb R^{l+m}  \longrightarrow 
\Pic(Y)\otimes _{\mathbb Z}\mathbb R
\end{equation*} 
defined by 
\begin{equation*} 
\psi(x_1, \ldots, x_{l+m})=\sum _{\alpha=1}^m x_\alpha f^*\omega_\alpha
-\sum _{\beta=1}^l 
x_{m+\beta}B_\beta. 
\end{equation*} 
We note that $\psi$ is defined over $\mathbb Q$. 
By construction, 
\begin{equation*} 
\mathcal A:=\psi^{-1}\left(K_Y+\sum _{j\geq l+1} b_j B_j\right)
\end{equation*}  
is a nonempty affine subspace of $\mathbb R^{l+m}$ defined over 
$\mathbb Q$. We put 
\begin{equation*} 
P:=(a_1, \ldots, a_m, b_1, \ldots, b_l)\in \mathcal A. 
\end{equation*} 
We can take $P_1, \ldots, P_k \in \mathcal A\cap \mathbb Q^{l+m}$ and 
$r_1, \ldots, r_k\in \mathbb R_{>0}$ such that 
$\sum _{i=1}^k r_i=1$ and $\sum _{i=1}^k r_i P_i=P$ in $\mathcal A$. 
Note that we can make $P_i$ arbitrary close to $P$ for every $i$. 
So we may assume that $P_i$ is sufficiently close to $P$ for every $i$. 
For each $P_i$, we obtain 
\begin{equation}\label{d-eq4.1}
K_Y+D_i\sim _{\mathbb Q} f^*\omega_i
\end{equation}
which satisfies (ii) by using $\psi$. 
By construction, (i) and (iii) hold. 
By \eqref{d-eq4.1} and (ii), 
\begin{equation*}
\left(X, \omega_i, f\colon  (Y, D_i)\to X\right)
\end{equation*}  
is a quasi-log scheme with the desired properties for every $i$. 
Therefore, we get (iv). 
\end{proof}

\section{Proof of Theorem \ref{a-thm1.9}}\label{e-sec5}

In this section, we prove Theorem \ref{a-thm1.9}. In some sense, 
Theorem \ref{a-thm1.9} is a generalization 
of \cite[Theorem 1.1]{fujino-haidong}.  

\begin{proof}[Proof of Theorem \ref{a-thm1.9}]
Let $f\colon (Y, B_Y)\to X$ be a proper surjective morphism from a 
quasi-projective globally 
embedded simple normal crossing pair $(Y, B_Y)$ as in 
Definition \ref{d-def4.2} (see Theorem \ref{d-thm4.5}). 
By \cite[Proposition 6.3.1]{fujino-foundations}, 
we may assume that 
$Y$ is quasi-projective and that 
the union of all strata of $(Y, B_Y)$ mapped to $\Nqklt(X, \omega)$, 
which is denoted by $Y''$, is a union of some irreducible 
components of $Y$. 
We put $Y'=Y-Y''$ and $K_{Y'}+B_{Y'}=(K_Y+B_Y)|_{Y'}$. 
Then we obtain the following commutative diagram: 
\begin{equation*} 
\xymatrix{
Y' \ar[d]_{f'}\ar@{^(->}[r]^\iota&Y\ar[d]^f\\ 
V \ar[r]_p& X
}
\end{equation*} 
where $\iota\colon Y'\to Y$ is a natural closed immersion 
and 
\begin{equation*} 
\xymatrix{Y' \ar[r]^{f'}& V \ar[r]^p& X
}
\end{equation*}  
is the Stein factorization of $f\circ \iota\colon Y'\to X$. 
By construction, 
$\iota\colon Y'\to Y$ is an isomorphism 
over the generic point of $X$.
By construction again, the natural map $\mathcal O_V\to 
f'_*\mathcal O_{Y'}$ is an isomorphism 
and every stratum of $Y'$ is dominant onto $V$. 
Therefore, $p$ is birational. 
\begin{claim}\label{e-claim1}
$V$ is normal. 
\end{claim}
\begin{proof}[Proof of Claim \ref{e-claim1}] 
Let $\pi\colon V^n\to V$ be the normalization. 
Since every stratum of $Y'$ is dominant onto $V$, 
there exists a closed subset $\Sigma$ of $Y'$ such that 
$\codim_{Y'}\Sigma\geq 2$ and that 
$\pi^{-1}\circ f'\colon  Y'\dashrightarrow V^n$ is a morphism 
on $Y'\setminus \Sigma$. 
Let $\widetilde Y$ be the graph of 
$\pi^{-1}\circ f'\colon  Y'\dashrightarrow V^n$. 
Then we have the following commutative 
diagram: 
\begin{equation*} 
\xymatrix{
\widetilde Y \ar[d]_-{\widetilde f}\ar[r]^-q&Y'\ar[d]^-{f'}\\ 
V^n\ar[r]_-\pi& V
}
\end{equation*} 
where $q$ and $\widetilde f$ are natural projections. 
Note that $q\colon \widetilde Y\to Y'$ is an isomorphism 
over $Y'\setminus \Sigma$ by construction. 
Since $Y'$ is a simple normal crossing divisor on a smooth variety, 
$Y'$ satisfies Serre's $S_2$ condition. Hence, by 
$\codim_{Y'}\Sigma\geq 2$, the natural map 
$\mathcal O_{Y'}\to q_*\mathcal O_{\widetilde Y}$ is an isomorphism. 
Therefore, the composition 
\begin{equation*} 
\mathcal O_V\to \pi_*\mathcal O_{V^n}
\to \pi_*\widetilde f_*\mathcal O_{\widetilde Y}
=f'_*q_*\mathcal O_{\widetilde Y}\simeq \mathcal O_V
\end{equation*}  
is an isomorphism. 
Thus we have $\mathcal O_V\simeq \pi_*\mathcal O_{V^n}$. 
This implies that $V$ is normal. 
\end{proof} 
Therefore, $p\colon V\to X$ is nothing but the 
normalization $\nu\colon Z\to X$. 
So we have the following commutative diagram. 
\begin{equation*} 
\xymatrix{
Y' \ar[d]_{f'}\ar@{^(->}[r]^\iota&Y\ar[d]^f\\ 
Z \ar[r]_\nu& X
}
\end{equation*} 
\begin{claim}\label{e-claim2}The natural map 
\begin{equation*} 
\alpha\colon  \mathcal O_Z\to f'_*\mathcal O_{Y'}
(\lceil -(B^{<1}_{Y'})\rceil)
\end{equation*}  
is an isomorphism outside $\nu^{-1}\Nqlc(X, \omega)$. 
\end{claim}
\begin{proof}[Proof of Claim \ref{e-claim2}]
Note that $\nu\colon  Z\to X$ is an isomorphism 
over $X\setminus \Nqklt (X, \omega)$ by Lemma \ref{d-lem4.8}. 
Moreover, $f'\colon Y'\to Z$ is nothing but 
$f\colon Y\to X$ over 
$Z\setminus \nu^{-1}\Nqklt(X, \omega)$ by construction. 
Therefore, $\alpha$ is an isomorphism outside 
$\nu^{-1}\Nqklt (X, \omega)$. 
By replacing $X$ with $X\setminus \Nqlc(X, \omega)$, 
we may assume that $\Nqlc(X, \omega)=\emptyset$. 
Hence the natural map 
$\mathcal O_X\to f_*\mathcal O_Y(\lceil -(B^{<1}_Y)\rceil)$ is an 
isomorphism. 
Therefore, we have $f_*\mathcal O_Y\simeq \mathcal O_X$. 
Since $Z$ is normal and $f'_*\mathcal O_{Y'}(\lceil -(B^{<1}_{Y'})
\rceil)$ is torsion-free, it is sufficient to see that $\alpha$ is an isomorphism 
in codimension one. 
Let $P$ be any prime divisor on $Z$ such that 
$P\subset \nu^{-1}\Nqklt (X, \omega)$. 
We note that every fiber of $f$ is connected by $f_*\mathcal O_Y
\simeq \mathcal O_X$. 
Then, by construction, there exists an irreducible 
component of $B^{=1}_{Y'}$ which maps onto $P$. 
Therefore, the effective divisor $\lceil -(B^{<1}_{Y'})\rceil$ 
does not contain the whole fiber of $f'$ over the generic 
point of $P$. Thus, 
$\alpha$ is an isomorphism at the generic point of $P$. 
This means that $\alpha$ is an isomorphism. 
\end{proof}
We put $\mathcal S:=f'_*\mathcal O_{Y'}(\lceil -(B^{<1}_{Y'})\rceil 
-\lfloor B^{>1}_{Y'}\rfloor-Y''|_{Y'})$. 
Then we have: 
\begin{claim}\label{e-claim3}
$\mathcal S$ is an ideal sheaf on $Z$. 
\end{claim}
\begin{proof}[Proof of Claim \ref{e-claim3}]
By definition, $\mathcal S$ is a torsion-free coherent sheaf on $Z$. 
By the proof of \cite[Theorem 6.3.5 (i)]{fujino-foundations} 
(see also the idea of the proof of Theorem \ref{d-thm4.6} (i)), 
we have 
\begin{equation*} 
\nu_*\mathcal S=f_*\mathcal O_{Y'}(\lceil -(B^{<1}_{Y'})\rceil 
-\lfloor B^{>1}_{Y'}\rfloor-Y''|_{Y'})=\mathcal I_{\Nqklt(X, \omega)}
\subset \mathcal O_X. 
\end{equation*}  
Since $\nu$ is finite, 
\begin{equation*} 
\nu^*\nu_*\mathcal S\to \mathcal S
\end{equation*}  
is surjective. 
This implies that $\mathcal S$ is an ideal sheaf on $Z$. 
\end{proof}
We put $\mathcal T:=f'_*\mathcal O_{Y'}(\lceil -(B^{<1}_{Y'})\rceil 
-\lfloor B^{>1}_{Y'}\rfloor)$. Then we have: 
\begin{claim}\label{e-claim4}
$\mathcal T$ is an ideal sheaf on $Z$. 
\end{claim}
\begin{proof}[Proof of Claim \ref{e-claim4}]
Outside $\nu^{-1}\Nqlc(X, \omega)$, it is 
obvious that 
$\mathcal T=f'_*\mathcal O_{Y'}(\lceil 
-(B^{<1}_{Y'})\rceil)$ holds. 
Therefore, we obtain $\mathcal T=\mathcal O_Z$ outside 
$\nu^{-1}\Nqlc(X, \omega)$ by Claim \ref{e-claim2}. 
Since $\mathcal T$ is torsion-free and 
$Z$ is normal, it is sufficient to show that 
$\mathcal T$ is an ideal sheaf in codimension one. 
Let $Q$ be any prime divisor on $X$ such that 
$Q\subset \Nqlc(X, \omega)$. 
We take a prime divisor $P$ on $Z$ such that 
$\nu(P)=Q$. 

If $\lceil -(B^{<1}_{Y'})\rceil$ does not contain the 
whole fiber of $f'$ over the generic point of $P$, 
then the natural map 
\begin{equation*} 
\alpha\colon \mathcal O_Z\to f'_*\mathcal O_{Y'}(\lceil -(B^{<1}_{Y'})\rceil)
\end{equation*}  
is an isomorphism at the generic point of $P$ since 
the natural map $\mathcal O_Z\to f'_*\mathcal O_{Y'}$ is 
an isomorphism by construction. 
Then $f'_*\mathcal O_{Y'}(\lceil -(B^{<1}_{Y'})\rceil 
-\lfloor B^{>1}_{Y'}\rfloor)$ is an ideal sheaf at the generic point of $P$. 

If $\lceil -(B^{<1}_{Y'})\rceil$ contains the 
whole fiber of $f'$ over the generic point of $P$, then 
$\mathcal S=\mathcal T$ holds over the generic point of 
$P$ because 
$\lceil -(B^{<1}_{Y'})\rceil$ and $Y''|_{Y'}$ have 
no common irreducible components. Therefore, 
$\mathcal T$ is an ideal sheaf at the generic point of $P$ by 
Claim \ref{e-claim3}

Hence $\mathcal T$ is an ideal sheaf on $Z$. This is what we wanted. 
\end{proof}
By construction, 
\begin{equation*} 
K_{Y'}+B_{Y'}\sim _{\mathbb R} f'^*\nu^*\omega
\end{equation*} 
obviously holds. 
We can define $\Nqlc(Z, \nu^*\omega)$ by the ideal 
sheaf 
$f'_*\mathcal O_{Y'}(\lceil -(B^{<1}_{Y'})\rceil 
-\lfloor B^{>1}_{Y'}\rfloor)$ (see Claim \ref{e-claim4}). 
Hence 
\begin{equation*} 
\left( Z, \nu^*\omega, f'\colon (Y', B_{Y'})\to Z\right)
\end{equation*}  
naturally becomes a quasi-log scheme. By Claim \ref{e-claim3} and 
its proof and \cite[Propositions 6.3.1 and 6.3.2]{fujino-foundations},  
\begin{equation*} 
\mathcal I_{\Nqklt(Z, \nu^*\omega)}=f'_*
\mathcal O_{Y'}(\lceil -(B^{<1}_{Y'})\rceil 
-\lfloor B^{>1}_{Y'}\rfloor-Y''|_{Y'})
\end{equation*} satisfies 
\begin{equation*} 
\nu_*\mathcal I_{\Nqklt(Z, \nu^*\omega)}=\mathcal I
_{\Nqklt(X, \omega)}. 
\end{equation*}
Hence 
\begin{equation*}
\left(Z, \nu^*\omega, f'\colon  (Y', B_{Y'})\to Z\right)
\end{equation*}  
is a quasi-log scheme with the desired properties. 
\end{proof}

\section{On basic slc-trivial fibrations}\label{f-sec6}

In this section, we quickly explain {\em{basic slc-trivial fibrations}}. 
For the details, see \cite{fujino-slc-trivial} and \cite{fujino-fujisawa-liu}. 
Let us start with the definition of {\em{potentially nef divisors}}. 

\begin{defn}[{Potentially nef divisors, see 
\cite[Definition 2.5]{fujino-slc-trivial}}]\label{f-def6.1} 
Let $X$ be a normal 
variety and let $D$ be a divisor on $X$. 
If there exist a completion $X^\dag$ of $X$, 
that is, $X^\dag$ is a complete normal 
variety and contains $X$ as a dense Zariski open set, and 
a nef divisor $D^\dag$ on $X^\dag$ such that 
$D=D^\dag|_X$, then $D$ is called 
a {\em{potentially nef}} divisor on $X$. 
A finite $\mathbb Q_{>0}$-linear (resp.~$\mathbb R_{>0}$-linear) 
combination of potentially nef divisors is called 
a {\em{potentially nef}} $\mathbb Q$-divisor (resp.~$\mathbb R$-divisor). 
\end{defn}

It is convenient to use {\em{b-divisors}} to explain 
several results on basic slc-trivial fibrations. 
Here we do not repeat the definition of b-divisors. 
For the details, see \cite[2.3.2 b-divisors]{corti} and 
\cite[Section 2]{fujino-slc-trivial}. 

\begin{defn}[Canonical b-divisors]\label{f-def6.2}
Let $X$ be a normal variety and let 
$\omega$ be a top rational differential 
form of $X$. 
Then $(\omega)$ defines a 
b-divisor $\mathbf K$. We call $\mathbf K$ 
the {\em{canonical b-divisor}} of $X$. 
\end{defn}

\begin{defn}[$\mathbb Q$-Cartier closures]\label{f-def6.3}
The {\em{$\mathbb Q$-Cartier 
closure}} of a $\mathbb Q$-Cartier $\mathbb Q$-divisor 
$D$ on a normal variety $X$ is the $\mathbb Q$-b-divisor 
$\overline D$ with trace 
\begin{equation*} 
\overline D _Y=f^*D, 
\end{equation*} 
where $f \colon Y\to X$ is a proper birational morphism 
from a normal variety $Y$. 
\end{defn}

We use the following definition in order to state  
the main result of \cite{fujino-slc-trivial}. 

\begin{defn}[{\cite[Definition 2.12]{fujino-slc-trivial}}]\label{f-def6.4}
Let $X$ be a normal variety. 
A $\mathbb Q$-b-divisor $\mathbf D$ of $X$ 
is {\em{b-potentially nef}} 
(resp.~{\em{b-semi-ample}}) if there 
exists a proper birational morphism $X'\to X$ from a normal 
variety $X'$ such that $\mathbf D=\overline {\mathbf D_{X'}}$, that 
is, $\mathbf D$ is the $\mathbb Q$-Cartier closure of $\mathbf D_{X'}$, 
and that 
$\mathbf D_{X'}$ is potentially nef 
(resp.~semi-ample). 
A $\mathbb Q$-b-divisor $\mathbf D$ of $X$ is {\em{$\mathbb Q$-b-Cartier}} 
if there is a proper birational morphism $X'\to X$ from a normal 
variety $X'$ such that $\mathbf D=\overline{\mathbf D_{X'}}$. 
\end{defn}

Roughly speaking, a basic slc-trivial fibration is a canonical bundle 
formula for simple normal crossing pairs. 

\begin{defn}[Simple normal crossing pairs]\label{f-def6.5} 
We say that the pair $(X, B)$ is a {\em{simple normal 
crossing pair}} if $(X, B)$ is Zariski locally a globally embedded simple 
normal crossing pair at any point $x\in X$. 
Let $(X, B)$ be a simple normal crossing pair and 
let $\nu\colon X^\nu \to X$ be the normalization. 
We define $\Theta$ by 
\begin{equation*} 
K_{X^\nu}+\Theta=\nu^*(K_X+B), 
\end{equation*}  
that is, $\Theta$ is the sum of the inverse images of 
$B$ and the singular locus of $X$. 
Then a {\em{stratum}} of $(X, B)$ is an irreducible 
component of $X$ or the $\nu$-image of some log canonical 
center of $(X^\nu, \Theta)$. 
\end{defn}

We note that a globally embedded simple normal crossing 
pair is obviously a simple normal crossing pair by 
definition. 
We also note that the definition of strata of 
$(X, B)$ in Definition \ref{f-def6.5} coincides with 
the one in Definition \ref{d-def4.1} when 
$(X, B)$ is a globally embedded simple normal crossing 
pair. 

\begin{rem}\label{f-rem6.6} 
Let $(X, B)$ be a simple normal crossing pair. 
A {\em{stratum}} of $X$ means a stratum of $(X, 0)$. 
Let $X=\bigcup_{i\in I}X_i$ be the irreducible decomposition 
of $X$. 
Then we can easily check that $W$ is a stratum of $X$ if and only 
if $W$ is an irreducible component of $X_{i_1}\cap 
\cdots \cap X_{i_k}$ for some $\{i_1, \ldots, i_k\}\subset I$. 
\end{rem} 

We introduce the notion of basic slc-trivial fibrations. 

\begin{defn}[{Basic slc-trivial fibrations, 
see \cite[Definition 4.1]{fujino-slc-trivial}}]\label{f-def6.7}
A {\em{pre-basic slc-trivial fibration}} $f \colon (X, B)\to Y$ consists of 
a projective surjective morphism 
$f \colon X\to Y$ and a simple normal crossing pair $(X, B)$ satisfying 
the following properties: 
\begin{itemize}
\item[(1)] $Y$ is a normal variety,   
\item[(2)] every stratum of $X$ is dominant onto $Y$ and 
$f_*\mathcal O_X\simeq \mathcal O_Y$, 
\item[(3)] $B$ is a $\mathbb Q$-divisor such that $B=B^{\leq 1}$ holds 
over 
the generic point of $Y$, and 
\item[(4)] there exists 
a $\mathbb Q$-Cartier $\mathbb Q$-divisor $D$ on $Y$ such that 
\begin{equation*} 
K_X+B\sim _{\mathbb Q}f^*D. 
\end{equation*}  
\end{itemize}
If a pre-basic slc-trivial fibration $f \colon (X, B)\to Y$ also satisfies 
\begin{itemize}
\item[(5)] $\rank f_*\mathcal O_X(\lceil -(B^{<1})\rceil)=1$, 
\end{itemize}
then it is called a {\em{basic slc-trivial fibration}}. 
\end{defn}

If $X$ is irreducible 
and $(X, B)$ is sub kawamata log terminal 
(resp.~sub log canonical) over the generic point of $Y$ in 
Definition \ref{f-def6.7}, 
then it is a klt-trivial fibration (resp.~an lc-trivial fibration). 
For the details of lc-trivial fibrations, see \cite{fujino-some}, 
\cite{fujino-gongyo2}, and so on. 

\medskip 

In order to define discriminant $\mathbb Q$-b-divisors and 
moduli $\mathbb Q$-b-divisors for basic slc-trivial fibrations, 
we need the notion of {\em{induced $($pre-$)$basic slc-trivial fibrations}}. 

\begin{defn}[{Induced (pre-)basic slc-trivial 
fibrations, see \cite[4.3]{fujino-slc-trivial}}]\label{f-def6.8}
Let $f \colon (X, B)\to Y$ be a 
(pre-)basic slc-trivial fibration 
and let $\sigma \colon Y'\to Y$ be a generically finite 
surjective morphism from a normal variety $Y'$. 
Then we have an {\em{induced {\em{(}}pre-{\em{)}}basic slc-trivial fibration}} 
$f' \colon (X', B_{X'})\to Y'$, where 
$B_{X'}$ is defined by $\mu^*(K_X+B)=K_{X'}+B_{X'}$, with 
the following commutative diagram: 
\begin{equation*} 
\xymatrix{
   (X', B_{X'}) \ar[r]^{\mu} \ar[d]_{f'} & (X, B)\ar[d]^{f} \\
   Y' \ar[r]_{\sigma} & Y, 
} 
\end{equation*} 
where $X'$ coincides with 
$X\times _{Y}Y'$ over a nonempty Zariski open set of $Y'$. 
More precisely, $(X', B_{X'})$ is a simple normal crossing pair with a morphism 
$X'\to X\times _Y Y'$ that is an isomorphism over 
a nonempty Zariski open set of $Y'$ such that 
$X'$ is projective over $Y'$ and that every stratum of $X'$ is dominant onto 
$Y'$. 
\end{defn}

Now we are ready to define {\em{discriminant 
$\mathbb Q$-b-divisors}} and 
{\em{moduli $\mathbb Q$-b-divisors}} for basic slc-trivial fibrations. 

\begin{defn}[{Discriminant and 
moduli $\mathbb Q$-b-divisors, 
see \cite[4.5]{fujino-slc-trivial}}]\label{f-def6.9} 
Let $f \colon (X, B)\to Y$ be a (pre-)basic 
slc-trivial fibration as in Definition \ref{f-def6.7}. 
Let $P$ be a prime divisor on $Y$. 
By shrinking $Y$ around the generic point of $P$, 
we assume that $P$ is Cartier. We set 
\begin{equation*} 
b_P=\max \left\{t \in \mathbb Q\, \left|\, 
\begin{array}{l}  {\text{$(X^\nu, \Theta+t\nu^*f^*P)$ is sub log canonical}}\\
{\text{over the generic point of $P$}} 
\end{array}\right. \right\},  
\end{equation*}  
where $\nu \colon X^\nu\to X$ is the normalization and 
$K_{X^\nu}+\Theta=\nu^*(K_X+B)$, that is, 
$\Theta$ is the sum of the inverse images of $B$ and the singular 
locus of $X$, and 
set 
\begin{equation*} 
B_Y=\sum _P (1-b_P)P, 
\end{equation*}  
where $P$ runs over prime divisors on $Y$. 
Then it is easy to  see that 
$B_Y$ is a well-defined $\mathbb Q$-divisor on 
$Y$ and is called the {\em{discriminant 
$\mathbb Q$-divisor}} of $f \colon (X, B)\to Y$. We set 
\begin{equation*} 
M_Y=D-K_Y-B_Y
\end{equation*}  
and call $M_Y$ the {\em{moduli $\mathbb Q$-divisor}} of $f \colon 
(X, B)\to Y$. 
By definition, we have 
\begin{equation*} 
K_X+B\sim _{\mathbb Q}f^*(K_Y+B_Y+M_Y). 
\end{equation*} 

Let $\sigma\colon Y'\to Y$ be a proper birational morphism 
from a normal variety $Y'$ and let $f' \colon (X', B_{X'})\to Y'$ be 
an induced (pre-)basic slc-trivial fibration 
by $\sigma \colon Y'\to Y$.  
We can define $B_{Y'}$, $K_{Y'}$ and $M_{Y'}$ such that 
$\sigma^*D=K_{Y'}+B_{Y'}+M_{Y'}$, 
$\sigma_*B_{Y'}=B_Y$, $\sigma _*K_{Y'}=K_Y$ 
and $\sigma_*M_{Y'}=M_Y$. We note that 
$B_{Y'}$ is independent of the choice of $(X', B_{X'})$, 
that is, $B_{Y'}$ is well defined. Hence 
there exist a unique $\mathbb Q$-b-divisor $\mathbf B$ 
such that 
$\mathbf B_{Y'}=B_{Y'}$ for every $\sigma \colon Y'\to Y$ and a unique 
$\mathbb Q$-b-divisor $\mathbf M$ such that $\mathbf M_{Y'}=M_{Y'}$ for 
every $\sigma \colon Y'\to Y$. 
Note that $\mathbf B$ is called the {\em{discriminant $\mathbb Q$-b-divisor}} and 
that $\mathbf M$ is called 
the {\em{moduli $\mathbb Q$-b-divisor}} associated to $f \colon (X, B)\to Y$. 
We sometimes simply say that $\mathbf M$ is 
the {\em{moduli part}} of $f \colon (X, B)\to Y$. 
\end{defn}

Let us see the main result of \cite{fujino-slc-trivial}. 

\begin{thm}[{\cite[Theorem 1.2]{fujino-slc-trivial}}]\label{f-thm6.10}
Let $f\colon (X, B)\to Y$ be a basic slc-trivial fibration and 
let $\mathbf B$ and $\mathbf M$ be the induced 
discriminant and moduli $\mathbb Q$-b-divisors 
associated to $f\colon (X, B)\to Y$ respectively. 
Then we have the following properties: 
\begin{itemize}
\item[(i)] $\mathbf K+\mathbf B$ is $\mathbb Q$-b-Cartier, 
where $\mathbf K$ is the canonical 
b-divisor of $Y$, and 
\item[(ii)] $\mathbf M$ is b-potentially nef, that is,  
there exists a proper birational morphism $\sigma\colon Y'\to Y$ 
from a normal variety $Y'$ such that 
$\mathbf M_{Y'}$ is a potentially nef $\mathbb Q$-divisor on $Y'$ and 
that $\mathbf M=\overline{\mathbf M_{Y'}}$. 
\end{itemize}
\end{thm}

When $\dim Y=1$ in Theorem \ref{f-thm6.10}, we have: 

\begin{thm}[{\cite[Corollary 1.4]{fujino-fujisawa-liu}}]\label{f-thm6.11}
In Theorem \ref{f-thm6.10}, 
we further assume that $\dim Y=1$. 
Then the moduli $\mathbb Q$-divisor 
$M_Y$ of $f\colon (X, B)\to Y$ is semi-ample. 
\end{thm}

The proof of Theorems \ref{f-thm6.10} and \ref{f-thm6.11} 
heavily depends on the theory of variations 
of mixed Hodge structure discussed in \cite{fujino-fujisawa} 
(see also \cite{fujino-fujisawa-saito}). 
For some related topics, see \cite{fujino-higher}, 
\cite{fujino-some}, \cite{fujino-gongyo2}, and so on. 

In \cite{fujino-hashizume2}, 
we will generalize the framework of basic slc-trivial 
fibrations for $\mathbb R$-divisors and establish a generalization 
of Theorem \ref{f-thm6.10} for $\mathbb R$-divisors. 

\section{On normal quasi-log schemes}\label{g-sec7}

In this section, we treat the following deep result on 
the structure of normal quasi-log schemes. 
It is a generalization of \cite[Theorem 1.7]{fujino-slc-trivial}. 
The proof of Theorem \ref{g-thm7.1} uses 
Theorems \ref{f-thm6.10} and 
\ref{f-thm6.11}. 

\begin{thm}\label{g-thm7.1}
Let $[X, \omega]$ be a quasi-log scheme such that 
$X$ is a normal variety. 
Then there exists a projective birational morphism 
$p\colon X'\to X$ from a smooth quasi-projective variety 
$X'$ such that 
\begin{equation*} 
K_{X'}+B_{X'}+M_{X'}=p^*\omega, 
\end{equation*}  
where $B_{X'}$ is an $\mathbb R$-divisor such that 
$\Supp B_{X'}$ is a simple normal crossing divisor and 
that $B^{<0}_{X'}$ is $p$-exceptional, and 
$M_{X'}$ is a potentially nef $\mathbb R$-divisor on $X'$. 
Furthermore, we can make $B_{X'}$ satisfy 
$p(B^{\geq 1}_{X'})=\Nqklt(X, \omega)$ set theoretically. 
When $X$ is a curve, we can make $M_{X'}$ semi-ample 
in the above statement. 

We further assume that 
$[X, \omega]$ has a $\mathbb Q$-structure. 
Then we can make $B_{X'}$ and $M_{X'}$ $\mathbb Q$-divisors in 
the above statement.  
\end{thm}

Let us prove Theorem \ref{g-thm7.1}. 

\begin{proof}[Proof of Theorem \ref{g-thm7.1}]
We divide the proof into several steps. 
\setcounter{step}{0}
\begin{step}\label{g-step7.1}
Although this step is essentially the same as the proof of 
Theorem \ref{a-thm1.9}, 
we explain it again with some remarks on $\Nqlc(X, \omega)$ 
for the reader's convenience. 
Let $f\colon (Y, B_Y)\to X$ be a proper surjective morphism from a 
quasi-projective globally 
embedded simple normal crossing pair $(Y, B_Y)$ as in 
Definition \ref{d-def4.2} (see Theorem \ref{d-thm4.5}). 
By \cite[Proposition 6.3.1]{fujino-foundations}, 
we may assume that 
the union of all strata of $(Y, B_Y)$ mapped to $\Nqklt(X, \omega)$, 
which is denoted by $Y''$, is a union of some irreducible 
components of $Y$. 
We put $Y'=Y-Y''$ and $K_{Y'}+B_{Y'}=(K_Y+B_Y)|_{Y'}$. 
By the proof of Theorem \ref{a-thm1.9}, we obtain 
the following commutative diagram: 
\begin{equation*} 
\xymatrix{
Y' \ar[d]_{f'}\ar@{^(->}[r]^\iota&Y\ar[d]^f\\ 
X\ar@{=}[r]& X
}
\end{equation*} 
where $\iota\colon Y'\to Y$ is a natural closed immersion 
such that the natural map $\mathcal O_X\to 
f'_*\mathcal O_{Y'}$ is an isomorphism 
and that every stratum of $Y'$ is dominant onto $X$. 
By Theorem \ref{a-thm1.9} and its proof,  
\begin{equation*} 
\left(X, \omega, f'\colon (Y', B_{Y'})\to X\right)
\end{equation*}  
is a quasi-log scheme with 
\begin{equation*} 
\mathcal I_{\Nqklt(X, \omega, f'\colon (Y', B_{Y'})\to X)}=
\mathcal I_{\Nqklt(X, \omega, f\colon (Y, B_{Y})\to X)}. 
\end{equation*}  
We note that 
if 
\begin{equation*} 
\left(X, \omega, f\colon (Y, B_Y)\to X\right)
\end{equation*} 
has a 
$\mathbb Q$-structure then it is obvious that 
\begin{equation*} 
\left(X, \omega, f'\colon (Y', B_{Y'})\to X\right)
\end{equation*}  
also has a $\mathbb Q$-structure by construction. 
Therefore, by replacing $f\colon (Y, B_Y)\to X$ with 
$f'\colon (Y', B_{Y'})\to X$, we may assume 
that every stratum of $Y$ is mapped onto $X$ by $f$. 
By construction, we can easily see that 
\begin{equation*} 
\Nqlc(X, \omega, f\colon (Y', B_{Y'})\to X)\subset 
\Nqlc(X, \omega, f\colon (Y, B_Y)\to X)
\end{equation*}  
holds set theoretically. 
However, the relationship 
between 
$
\Nqlc(X, \omega, f\colon (Y', B_{Y'})\to X)$ and  
$\Nqlc(X, \omega, f\colon (Y, B_Y)\to X)
$ is not clear. 
We note that all we need in this proof is the fact that 
\begin{equation*} 
\Nqklt(X, \omega, f\colon (Y', B_{Y'})\to X)
=\Nqklt(X, \omega, f\colon (Y, B_Y)\to X)
\end{equation*}  
holds set theoretically.  
\end{step}
\begin{step}\label{g-step7.2}
By Step \ref{g-step7.1}, 
we may assume that 
$f\colon (Y, B_Y)\to X$ is a projective
surjective 
morphism from a simple normal crossing 
pair $(Y, B_Y)$ such that every stratum of 
$Y$ is dominant onto $X$. 
By taking some more blow-ups, 
we may further assume that $(B^h_Y)^{=1}$ is 
Cartier and that every stratum of $(Y, (B^h_Y)^{=1})$ 
is dominant onto $X$ (see, for example, 
\cite[Theorem 1.4 and Section 8]{bierstone-vera} 
and \cite[Lemma 2.11]{fujino-projectivity}). 
\end{step}
\begin{step}\label{g-step7.3}
In this step, we treat the case where 
$[X, \omega]$ has a $\mathbb Q$-structure. 
We note that 
\begin{equation*}
\mathcal O_X\to f_*\mathcal O_Y(\lceil 
-(B^{<1}_Y)\rceil)
\end{equation*} 
is an isomorphism outside 
$\Nqlc(X, \omega)$. Hence $\rank f_*\mathcal O_Y(\lceil 
-(B^{<1}_Y)\rceil)=1$ holds. Therefore, 
we can check that 
$f\colon (Y, B_Y)\to X$ is a basic slc-trivial 
fibration (see Definition \ref{f-def6.7}). 
Let $\mathbf B$ be the discriminant $\mathbb Q$-b-divisor and let 
$\mathbf M$ be the moduli $\mathbb Q$-b-divisor associated to 
$f\colon (Y, B_Y)\to X$. 
By \cite[Lemma 11.2]{fujino-slc-trivial}, we obtain that 
$\mathbf B_X$ is an effective $\mathbb Q$-divisor on $X$. 
By definition, 
we have $f((B^v_Y)^{\geq 1})=\Nqklt (X, \omega)$. 
We take a projective 
birational morphism $p\colon X'\to X$ from a smooth quasi-projective 
variety $X'$. 
Let $f'\colon  (Y', B_{Y'})\to X'$ be an induced basic slc-trivial 
fibration with the following commutative diagram. 
\begin{equation*} 
\xymatrix{
(Y, B_Y)\ar[d]_-f & (Y', B_{Y'})\ar[d]^-{f'}\ar[l]_-q\\
X & \ar[l]^-p X' 
}
\end{equation*}  
By Theorem \ref{f-thm6.10}, 
we may assume that 
there exists a simple normal crossing divisor $\Sigma_{X'}$ on $X'$ 
such that $\mathbf M=\overline{\mathbf M_{X'}}$, 
$\Supp \mathbf M_{X'}$ and 
$\Supp \mathbf B_{X'}$ are contained in $\Sigma_{X'}$, 
and that every stratum of 
$(Y', \Supp B^h_{Y'})$ is smooth over $X'\setminus \Sigma_{X'}$. 
Of course, 
we may assume that 
$M_{X'}:=\mathbf M_{X'}$ is potentially nef by 
Theorem \ref{f-thm6.10}. 
When $X$ is a curve, we may further assume that 
$M_{X'}$ is semi-ample 
by Theorem \ref{f-thm6.11}. 
We may assume that every irreducible 
component of $q^{-1}_*\left((B^v_Y)^{\geq 1}\right)$ is mapped onto a prime 
divisor in $\Sigma_{X'}$ 
with the aid of the flattening theorem 
(see \cite[Th\'eor\`eme (5.2.2)]{raynaud-g}). 
We put $B_{X'}:=\mathbf B_{X'}$. 
In the above setup, 
$f'(q^{-1}_*(B^v_Y)^{\geq 1})\subset B^{\geq 1}_{X'}$ by the definition of 
$\mathbf B$. 
Thus, we get $\Nqklt (X, \omega)\subset p(B^{\geq 1}_{X'})$. 
On the other hand, we can easily see that 
$p(B^{\geq 1}_{X'})\subset \Nqklt(X, \omega)$ by definition. 
Therefore, $p(B^{\geq 1}_{X'})=\Nqklt (X, \omega)$ holds. 
Since $p_*B_{X'}=\mathbf B_X$ and $\mathbf B_X$ is effective, 
$B^{<0}_{X'}$ is $p$-exceptional. 
Hence, $B_{X'}$ and $M_{X'}$ satisfy the desired properties. 
We note that $B_{X'}$ and $M_{X'}$ are obviously $\mathbb Q$-divisors 
by construction. 
\end{step}
\begin{step}\label{g-step7.4}
In this step, we treat the general case. 
We first use Lemma \ref{d-lem4.25} and get positive real numbers 
$r_i$ and $\left(X, \omega_i, 
f\colon (Y, D_i)\to X\right)$ for $1\leq i\leq k$ with 
the properties in Lemma \ref{d-lem4.25}. 
Then we apply the argument in 
Step \ref{g-step7.3} to 
\begin{equation*} 
\left(X, \omega_i, f\colon (Y, D_i)\to X\right)
\end{equation*} 
for 
every $i$.  
By Theorem \ref{f-thm6.10}, we 
can take a projective birational morphism 
$p\colon X'\to X$ from a smooth quasi-projective variety $X'$ 
which works for 
\begin{equation*}
\left(X, \omega_i, f\colon (Y, D_i)\to X\right)
\end{equation*} 
for every $i$. 
By summing them up with weight $r_i$, we get 
$\mathbb R$-divisors $B_{X'}$ and $M_{X'}$ with the desired properties. 
\end{step} 
We finish the proof of Theorem \ref{g-thm7.1}. 
\end{proof}

\section{Proof of Theorem \ref{a-thm1.10}}\label{p-sec8}

In this section, we prove Theorem \ref{a-thm1.10} as 
an application of Theorem \ref{g-thm7.1}. 
Then, by using Theorem \ref{a-thm1.10}, we prove Corollary \ref{p-cor8.2} and 
Lemma \ref{p-lem8.3}, which will play an important 
role in Section \ref{i-sec9}. 
Let us start with the following elementary 
lemma for the proof of Theorem \ref{a-thm1.10}. 

\begin{lem}\label{p-lem8.1} 
Let $X$ be a quasi-projective variety and let $H$ be an 
ample $\mathbb R$-divisor on $X$. 
Let $p\colon X'\to X$  be a projective 
birational morphism from a smooth quasi-projective 
variety $X'$ and let $F$ be an effective $p$-exceptional 
Cartier divisor on $X'$ such that $-F$ is $p$-ample. 
Let $M'$ be a potentially nef $\mathbb R$-divisor 
on $X'$. 
Then $p^*H-\varepsilon F+M'$ is ample 
for any $0<\varepsilon \ll 1$. 
\end{lem}

\begin{proof} 
We can write $H=\sum _{i=0}^k a_i H_i$ such that 
$a_i$ is a positive real number and $H_i$ is an ample 
Cartier divisor on $X$ for every $i$. 
If $0<\varepsilon \ll 1$, 
then we can take $\varepsilon _i$ 
such that 
$p^*H_i-\varepsilon _i F$ is ample 
for every $i$ with 
$\sum _{i=0}^k a_i \varepsilon _i=\varepsilon$. 
Since $M'$ is a potentially nef $\mathbb R$-divisor 
on $X'$, we can construct 
a smooth 
projective completion $X^\dag$ of 
$X'$ and a nef $\mathbb R$-divisor 
$M^\dag$ on $X^\dag$ such that 
$M^\dag|_{X'}=X'$. 
By taking a suitable birational modification 
of $X^\dag$, 
we may further assume that there is an ample 
$\mathbb R$-divisor $A$ on $X^\dag$ such 
that $A|_{X'}=p^*H_0-\varepsilon _0 F$ holds. 
Hence $a_0p^*H_0-a_0\varepsilon _0F+M'$ is ample. 
Therefore, $p^*H-\varepsilon F+M'$ is ample 
for any $0<\varepsilon \ll 1$. 
This is what we wanted. 
\end{proof}

Let us start the proof of Theorem \ref{a-thm1.10}. 

\begin{proof}[Proof of Theorem \ref{a-thm1.10}] 
By Theorem \ref{g-thm7.1}, 
there is a projective 
birational morphism $p\colon X'\to X$ from a smooth quasi-projective 
variety $X'$ such that 
\begin{equation*}
K_{X'}+B_{X'}+M_{X'}=p^*\omega, 
\end{equation*}  
where $B_{X'}$ is an $\mathbb R$-divisor on $X'$ whose 
support is a simple normal crossing divisor, $B^{<0}_{X'}$ is $p$-exceptional, 
$M_{X'}$ is a potentially nef $\mathbb R$-divisor on $X'$, 
and $p(B^{\geq 1}_{X'})=\Nqklt (X, \omega)$. 
By taking some more blow-ups, 
we may further assume that there is an effective 
$p$-exceptional Cartier 
divisor $F$ on $X'$ such that $-F$ is $p$-ample and 
that $\Supp 
F\cup \Supp B_{X'}$ is contained in 
a simple normal crossing divisor on $X'$. Then 
$p^*H-\varepsilon F+M_{X'}$ is ample 
for any $0<\varepsilon \ll 1$ by Lemma 
\ref{p-lem8.1}. We 
take a general effective $\mathbb R$-divisor $G$ on $X'$ such that $G\sim 
_{\mathbb R} p^*H-\varepsilon F+M_{X'}$ with $0<\varepsilon \ll 1$, 
$\Supp G \cup \Supp B_{X'}\cup \Supp F$ 
is contained in a simple normal crossing divisor on $X'$, 
and $(B_{X'}+\varepsilon F+G)^{\geq 1}=B^{\geq 1}_{X'}$ holds set 
theoretically. 
Then we have 
\begin{equation*}
\begin{split}
K_{X'}+B_{X'}+M_{X'}+p^*H&=K_{X'}+B_{X'}+\varepsilon 
F +p^*H-\varepsilon F +M_{X'}\\&\sim _{\mathbb R} 
K_{X'}+B_{X'} +\varepsilon F +G. 
\end{split}
\end{equation*}
We put $\Delta:=p_*(B_{X'}+\varepsilon F+G)$. 
By construction, $K_X+\Delta\sim _{\mathbb R}
\omega+H$. 
By construction again, we have 
\begin{equation*} 
\Nklt(X, \Delta)=
p\left((B_{X'}+\varepsilon F+G)^{\geq 1}\right)=p\left(B^{\geq 1}_{X'}
\right)=\Nqklt(X, \omega)
\end{equation*}  
set theoretically. 

When $[X, \omega]$ has a $\mathbb Q$-structure, 
we can make $B_{X'}$ and $M_{X'}$ $\mathbb Q$-divisors 
by Theorem \ref{g-thm7.1}. 
Then it is easy to see that we can make $\Delta$ a $\mathbb Q$-divisor 
on $X$ such that 
$K_X+\Delta\sim _{\mathbb Q} \omega+H$ 
when $H$ is an ample $\mathbb Q$-divisor and $[X, \omega]$ has a $\mathbb 
Q$-structure by the above construction of $\Delta$. 

Finally, if $X$ is a curve in the above argument, then $p\colon X'\to X$ 
is an isomorphism and $M_{X'}$ is semi-ample 
(see Theorem \ref{g-thm7.1}). 
Hence we can take $\Delta$ such that 
\begin{equation*} 
K_X+\Delta\sim _{\mathbb R} \omega
\end{equation*}  
with the desired properties. 
\end{proof}

For some related results, see \cite{fujino-gongyo}, 
\cite{fujino-subadjunction}, and so on. 
By applying Theorem \ref{a-thm1.10} to normal pairs, 
we have the following useful result. 

\begin{cor}\label{p-cor8.2}
Let $X$ be a normal variety and let $\Delta$ be an effective 
$\mathbb R$-divisor 
on $X$ such that $K_X+\Delta$ is $\mathbb R$-Cartier. 
Let $C$ be a log canonical center 
of $(X, \Delta)$ such that $C$ is a smooth 
curve. 
Then 
\begin{equation*} 
(K_X+\Delta)|_C\sim _{\mathbb R}K_C+\Delta_C
\end{equation*}  
holds for some effective $\mathbb R$-divisor 
$\Delta_C$ such that 
\begin{equation*} 
\Supp \Delta^{\geq 1}_C=
C\cap \left( \Nlc(X, \Delta) \cup \bigcup _{C\not\subset W}W\right), 
\end{equation*} 
where $W$ runs over lc centers of $(X, \Delta)$ 
which do not contain $C$, 
holds set theoretically. 
When $K_X+\Delta$ is $\mathbb Q$-Cartier, 
we can make $\Delta_C$ a 
$\mathbb Q$-divisor such that 
\begin{equation*} 
(K_X+\Delta)|_C\sim _{\mathbb Q}K_C+\Delta_C
\end{equation*} 
in the above statement. 
\end{cor}
\begin{proof}
As we saw in Example \ref{d-ex4.11}, 
$[X, K_X+\Delta]$ naturally becomes a quasi-log scheme. 
By construction, $\Nqlc(X, K_X+\Delta)=\Nlc(X, \Delta)$, 
$W$ is a qlc center of $[X, K_X+\Delta]$ if and only if 
$W$ is a log canonical center of $(X, \Delta)$. 
Hence we can see that $C$ is a qlc center of $[X, K_X+\Delta]$. 
Therefore, by adjunction (see 
Theorem \ref{d-thm4.6} (i) and 
\cite[Theorem 6.3.5 (i)]{fujino-foundations}), 
$[C', (K_X+\Delta)|_{C'}]$ is a quasi-log scheme, where 
$C'=C\cup \Nlc(X, \Delta)$. 
By Lemma \ref{d-lem4.20}, 
we see that 
$[C, (K_X+\Delta)|_C]$ is also a quasi-log scheme such that 
\begin{equation*} 
\Nqklt(C, (K_X+\Delta)|_C)=\Nqklt(C', (K_X+\Delta)|_{C'})\cap C
\end{equation*}  
holds set theoretically. 
By construction, we can easily see 
that 
\begin{equation*} 
\Nqklt(C', (K_X+\Delta)|_{C'})\cap C=
C\cap \left( \Nlc(X, \Delta) \cup \bigcup _{C\not\subset W}W\right), 
\end{equation*} 
where $W$ runs over lc centers of $(X, \Delta)$ 
which do not contain $C$, 
holds set theoretically (see Theorem \ref{d-thm4.10} 
and Corollary \ref{d-cor4.14}). 
By applying Theorem \ref{a-thm1.10} to 
$[C, (K_X+\Delta)|_C]$, we can find an effective 
$\mathbb R$-divisor $\Delta_C$ on $C$ such that 
\begin{equation*} 
(K_X+\Delta)|_C\sim _{\mathbb R} K_C+\Delta_C
\end{equation*}  
with 
\begin{equation*} 
\Supp \Delta^{\geq 1}_C=\Nqklt(C, (K_X+\Delta)|_C)=
C\cap \left( \Nlc(X, \Delta) \cup \bigcup _{C\not\subset W}W\right). 
\end{equation*}  
Of course, if $K_X+\Delta$ is 
$\mathbb Q$-Cartier, then we can make $\Delta_C$ 
a $\mathbb Q$-divisor 
such that 
\begin{equation*} 
(K_X+\Delta)|_C\sim _{\mathbb Q}K_C+\Delta_C
\end{equation*}  
in the above statement. 
\end{proof} 

We will use the following lemma in Section \ref{i-sec9}. 

\begin{lem}\label{p-lem8.3}
Let $\varphi\colon X\to Y$ be a proper surjective morphism 
between normal varieties such that 
$R^1\varphi_*\mathcal O_X=0$ and 
that $\dim \varphi^{-1}(y)\leq 1$ holds 
for every closed point $y\in Y$. 
Let $C$ be a projective curve 
on $X$ such that 
$\varphi(C)$ is a point. 
Then 
\begin{equation*} 
C\simeq \mathbb P^1. 
\end{equation*}  

Let $\Delta$ be an effective $\mathbb R$-divisor 
on $X$ such that $K_X+\Delta$ is $\mathbb R$-Cartier. 
If $C\not\subset \Nlc(X, \Delta)$ and 
\begin{equation*} 
C\cap \left( \Nlc(X, \Delta) \cup \bigcup _{C\not\subset W}W\right)\ne 
\emptyset, 
\end{equation*} 
where $W$ runs over lc centers of $(X, \Delta)$ 
which do not contain $C$, 
then the following inequality 
\begin{equation*} 
-(K_X+\Delta)\cdot C\leq 1
\end{equation*}  
holds. 
\end{lem}
\begin{proof} 
In Step \ref{p-step8.1}, we will prove that 
$C\simeq \mathbb P^1$ holds. 
In Step \ref{p-step8.2}, we will prove that 
$-(K_X+\Delta)\cdot C\leq 1$ by Corollary \ref{p-cor8.2}. 
\setcounter{step}{0}
\begin{step}\label{p-step8.1}
Although the argument in this step is well known, 
we will explain it in detail for the reader's convenience.  
Let us consider the following short exact sequence 
\begin{equation*} 
0\to \mathcal I_C\to \mathcal O_X\to \mathcal O_C\to 0, 
\end{equation*}  
where $\mathcal I_C$ is the defining ideal sheaf of 
$C$ on $X$. 
Since $\dim \varphi^{-1}(y)\leq 1$ for every $y\in Y$ by 
assumption, $R^2\varphi_*\mathcal I_C=0$ holds. 
Therefore, we get the following surjection 
\begin{equation*} 
R^1\varphi_*\mathcal O_X\to R^1\varphi_*\mathcal O_C\to 0. 
\end{equation*}  
By assumption, $R^1\varphi_*\mathcal O_X=0$. 
Hence $R^1\varphi_*\mathcal O_C=0$ holds. 
Since $\varphi(C)$ is a point by assumption, 
$H^1(C, \mathcal O_C)=0$ holds. 
This means that 
$C\simeq \mathbb P^1$. 
\end{step}
\begin{step}\label{p-step8.2}
By shrinking $Y$ around $\varphi(C)$, we may assume that 
$Y$ is quasi-projective. 
Let $B_1, \ldots, B_{n+1}$ be general very ample Cartier 
divisors on $Y$ passing through 
$\varphi(C)$ with $n=\dim X$. 
Then it is well known that 
\begin{equation*} 
\left(X, \Delta+\sum _{i=1}^{n+1} \varphi^*B_i\right)
\end{equation*}  
is not log canonical at any point of $C$ 
(see, for example, \cite[Lemma 13.2]{fujino-fund}) 
such that 
\begin{equation*}
\Nklt\left(X, \Delta+(1-\varepsilon)\sum _{i=1}^{n+1}\varphi^*B_i\right)
=\Nklt(X, \Delta)
\end{equation*}  
holds outside $\varphi^{-1}(\varphi(C))$ for 
every $0<\varepsilon \leq 1$. 
Hence we can take $0\leq c<1$ such that 
$C$ is a log canonical center of $(X, \Delta+\varphi^*B)$, where 
$B=c\sum_{i=1}^{n+1}B_i$. 
Since $B$ is effective, we see that 
\begin{equation*} 
C\cap \left ( \Nlc (X, \Delta+\varphi^*B)\cup \bigcup _{C\not 
\subset W} W\right)\ne \emptyset, 
\end{equation*}  
where $W$ runs over lc centers of $(X, \Delta+\varphi^*B)$ which do not 
contain $C$. 
By Corollary \ref{p-cor8.2}, we can take an effective 
$\mathbb R$-divisor $\Delta_C$ on $C$ such that 
\begin{equation*} 
(K_X+\Delta)|_C\sim _{\mathbb R} (K_X+\Delta+\varphi^*B)|_C
\sim _{\mathbb R} K_C+\Delta_C
\end{equation*}  
and that 
\begin{equation*} 
\Supp \Delta^{\geq 1}_C=
C\cap \left ( \Nlc (X, \Delta+\varphi^*B)\cup \bigcup _{C\not 
\subset W} W\right)\ne \emptyset
\end{equation*}  
holds. 
This implies that 
\begin{equation*} 
-(K_X+\Delta)\cdot C=-\deg (K_C+\Delta_C)=2-\deg \Delta_C\leq 1. 
\end{equation*} 
\end{step}
We finish the proof of Lemma \ref{p-lem8.3}. 
\end{proof}

\section{Proof of Theorem \ref{a-thm1.8}}\label{i-sec9}

In this section, we prove Theorem \ref{a-thm1.8}. 
Let us start with the following proposition, 
which is a consequence of 
the cone and contraction theorem for normal pairs (see \cite[Theorem 1.1]
{fujino-fund}) 
with the aid of Lemma \ref{p-lem8.3}. 
This is essentially due to \cite[Proposition 5.2]{svaldi}. 

\begin{prop}[{\cite[Proposition 5.2]{svaldi} and 
\cite[Proposition 7.1]{fujino-subadjunction}}]\label{i-prop9.1}
Let $\pi\colon X\to S$ be a projective morphism from a normal 
$\mathbb Q$-factorial 
variety $X$ onto a scheme $S$. Let $\Delta=\sum _i d_i 
\Delta_i$ be an effective 
$\mathbb R$-divisor on $X$, where the $\Delta_i$'s are the distinct 
prime components of $\Delta$ for all $i$, such that 
\begin{equation*} 
\left(X, \Delta':=\sum _{d_i<1}d_i\Delta_i
+\sum _{d_i\geq 1} \Delta_i\right)
\end{equation*} 
is dlt. 
Assume that 
$(K_X+\Delta)|_{\Nklt(X, \Delta)}$ is nef over $S$. 
Then $K_X+\Delta$ is nef over $S$ or 
there exists a non-constant morphism 
\begin{equation*} 
f\colon \mathbb A^1\longrightarrow X\setminus \Nklt(X, \Delta)
\end{equation*} 
such that 
$\pi\circ f(\mathbb A^1)$ is a point. 
More precisely, the curve $C$, the closure of $f(\mathbb A^1)$ in 
$X$, is a {\em{(}}possibly singular{\em{)}} rational curve 
with 
\begin{equation*} 
0<-(K_X+\Delta)\cdot C\leq 2\dim X. 
\end{equation*}  
Moreover, if $C\cap \Nklt(X, \Delta)\ne \emptyset$, then we can make 
$C$ satisfy a sharper estimate 
\begin{equation*} 
0<-(K_X+\Delta)\cdot C\leq 1. 
\end{equation*} 
\end{prop}
\begin{proof}
We note that $\Nklt(X, \Delta)$ coincides with $(\Delta')^{=1}=\lfloor 
\Delta'\rfloor$, 
$\Delta^{\geq 1}$, and $\lfloor \Delta\rfloor$ set theoretically 
because $(X, \Delta')$ is dlt by assumption. 
It is sufficient to construct a non-constant morphism 
\begin{equation*} 
f\colon \mathbb A^1\longrightarrow X\setminus \Nklt(X, \Delta)
\end{equation*} 
such that 
$\pi\circ f(\mathbb A^1)$ is a point with the desired 
properties when 
$K_X+\Delta$ is not nef over $S$. 
When $(X, \Delta)$ is kawamata log terminal, 
that is, $\lfloor \Delta\rfloor=0$, 
the statement is well known 
(see, for example, \cite[Theorem 1.1]{fujino-fund}, 
Theorem \ref{a-thm1.12},  
or Corollary \ref{l-cor12.3} below). 
Therefore, we may assume that 
$(X, \Delta)$ is not kawamata log terminal. 
By shrinking $S$ suitably, we may assume that $S$ and $X$ are 
both quasi-projective. 
By the cone and contraction theorem for 
normal pairs (see \cite[Theorem 
1.1]{fujino-fund}), we can take a $(K_X+\Delta)$-negative 
extremal ray $R$ of $\NE(X/S)$ and the 
associated extremal contraction morphism 
$\varphi:=\varphi_R\colon X\to Y$ over $S$ since 
$(K_X+\Delta)|_{\Nklt(X, \Delta)}$ is nef over $S$. 
Note that $(K_X+\Delta^{<1})\cdot R<0$ and 
$(K_X+\Delta')\cdot R<0$ hold because 
$(K_X+\Delta)|_{\Nklt(X, \Delta)}$ is nef over $S$. 
Since $(X, \Delta^{<1})$ is kawamata log terminal and $-(K_X+\Delta^{<1})$ is 
$\varphi$-ample, 
we get $R^i\varphi_*\mathcal O_X=0$ for every $i>0$ 
by the relative Kawamata--Viehweg vanishing theorem 
(see \cite[Corollary 5.7.7]{fujino-foundations}). 
By construction, $\varphi\colon  \Nklt(X, \Delta)\to \varphi(\Nklt(X, \Delta))$ is 
finite. 
We have the following 
short exact sequence  
\begin{equation*} 
0\to \mathcal O_X(-\lfloor \Delta'\rfloor)\to 
\mathcal O_X\to \mathcal O_{\lfloor \Delta'\rfloor}\to 0. 
\end{equation*}  
Since $-\lfloor \Delta'\rfloor -(K_X+\{\Delta'\})=-(K_X+\Delta')$ 
is $\varphi$-ample and $(X, \{\Delta'\})$ is kawamata log terminal, 
$R^i\varphi_*\mathcal O_X(-\lfloor \Delta'\rfloor)=0$ holds 
for every $i>0$ by the 
relative Kawamata--Viehweg vanishing theorem again (see 
\cite[Corollary 5.7.7]{fujino-foundations}).   
Therefore, 
\begin{equation*} 
0\to \varphi_*\mathcal O_X(-\lfloor \Delta'\rfloor)\to 
\mathcal O_Y\to \varphi_*\mathcal O_{\lfloor \Delta'\rfloor}\to 0
\end{equation*}  
is exact. 
This implies that 
$\Supp \lfloor \Delta'\rfloor=\Supp \Delta^{\geq 1}$ is connected 
in a neighborhood of any fiber of $\varphi$. 
\begin{case}\label{i-case9.1.1}
Assume that $\varphi$ is a Fano contraction, that is, 
$\dim Y<\dim X$. 
Then we see that $\Delta^{\geq 1}$ is $\varphi$-ample 
and that $\dim Y=\dim X-1$. Note that 
$\Supp \Delta^{\geq 1}$ is finite over $Y$ 
since no curves in $\Supp \Delta^{\geq 1}$ are 
contracted by $\varphi$. 

Assume that there exists a closed subvariety $\Sigma$ on $X$ with 
$\dim \Sigma\geq 2$ such that 
$\varphi(\Sigma)$ is a point. 
Then 
\begin{equation*} 
\dim \left(\Sigma\cap \Supp \Delta^{\geq 1}\right)\geq 1
\end{equation*}  
holds since 
$\Delta^{\geq 1}$ is $\varphi$-ample. 
This is a contradiction because 
$\Supp \Delta^{\geq 1}$ is finite over $Y$. 
Hence we obtain that $\dim \varphi^{-1}(y)=1$ for 
every closed point $y\in Y$. 

Let $C$ be any projective curve on $X$ such that 
$\varphi(C)$ is a point. 
Then $(X, \Delta)$ is log canonical at the generic point of 
$C$, equivalently, $C\not\subset \Nlc(X, \Delta)$, 
since 
$\Supp \Delta^{\geq 1}$ is finite over $Y$. 
More precisely, since $\Supp \Delta^{\geq 1}=\Nklt(X, \Delta)$ is finite 
over $Y$, $C\not\subset \Nklt(X, \Delta)$ and 
\begin{equation*} 
\Supp \Delta^{\geq 1}=\Nklt(X, \Delta)=\left( \Nlc(X, \Delta)\cup 
\bigcup _{C\not \subset W} W\right)
\end{equation*} 
holds, 
where $W$ runs over lc centers of $(X, \Delta)$ 
which do not contain $C$. On the other hand, 
\begin{equation*} 
C\cap \Supp \Delta^{\geq 1}
\ne \emptyset
\end{equation*}  
because $\Delta^{\geq 1}$ is $\varphi$-ample. 
Hence, by Lemma \ref{p-lem8.3}, 
we obtain that $C\simeq \mathbb P^1$ and 
that $-(K_X+\Delta)\cdot C\leq 1$. 

By the connectedness of $\Supp\Delta^{\geq 1}$ discussed 
above, $C\cap \Supp \Delta^{\geq 1}$ is a point. 
Therefore, we can find 
a non-constant morphism 
\begin{equation*} 
f\colon \mathbb A^1\longrightarrow X\setminus \Nklt(X, \Delta)
\end{equation*} 
such that 
$\pi\circ f(\mathbb A^1)$ is a point and that 
$0<-(K_X+\Delta)\cdot C\leq 1$ holds, 
where $C$ is the closure of $f(\mathbb A^1)$ in $X$. 
\end{case}
\begin{case}\label{i-case9.1.2}
Assume that $\varphi$ is a birational contraction and that 
the exceptional locus $\Exc(\varphi)$ of $\varphi$ 
is disjoint from $\Nklt(X, \Delta)$. 
In this situation, we can find a rational curve $C$ in a fiber of $\varphi$ 
with $0<-(K_X+\Delta)\cdot C\leq 2\dim X$ by the 
cone theorem for kawamata log terminal pairs 
(see \cite[Theorem 1.1]{fujino-fund}, Theorem \ref{a-thm1.12},  
or Corollary \ref{l-cor12.3} below). 
It is obviously disjoint from $\Nklt(X, \Delta)$. 
Therefore, we can take a non-constant morphism 
\begin{equation*} 
f\colon \mathbb A^1\longrightarrow X\setminus \Nklt(X, \Delta)
\end{equation*}  
such that the closure of $f(\mathbb A^1)$ is $C$. 
\end{case}
\begin{case}\label{i-case9.1.3}
Assume that $\varphi$ is a birational 
contraction and that 
$\Exc(\varphi)\cap \Nklt(X, \Delta)\ne \emptyset$. 
In this situation, as in Case \ref{i-case9.1.1}, 
we see that $\Delta^{\geq 1}$ is $\varphi$-ample and 
that $\dim \varphi^{-1}(y)\leq 1$ for every $y\in Y$. 
Let $C$ be any projective curve $C$ on $X$ such that 
$\varphi(C)$ is a point. Then, $C\cap \Supp \Delta^{\geq 1}\ne 
\emptyset$ holds since 
$\Delta^{\geq 1}$ is $\varphi$-ample, and 
$C\cap \Supp \Delta^{\geq 1}$ is a point 
by the connectedness of $\Supp \Delta^{\geq 1}$ discussed above. 
In particular, we obtain $C\not\subset \Nklt(X, \Delta)$ and 
\begin{equation*} 
C\cap \Supp \Delta^{\geq 1}
\ne \emptyset,  
\end{equation*}   
and 
\begin{equation*} 
\Supp \Delta^{\geq 1}=\Nklt(X, \Delta)=\left(\Nlc(X, \Delta)\cup 
\bigcup _{C\not\subset W}W\right), 
\end{equation*}  
where $W$ runs over lc centers of $(X, \Delta)$ which do not 
contain $C$. 
Hence, by Lemma \ref{p-lem8.3}, 
$C\simeq \mathbb P^1$ with $-(K_X+\Delta)\cdot 
C\leq 1$. 
Since $C\cap \Supp \Delta^{\geq 1}$ is a point, 
we get a non-constant morphism 
\begin{equation*} 
f\colon \mathbb A^1\longrightarrow X\setminus \Nklt(X, \Delta)
\end{equation*} 
such that 
$f(\mathbb A^1)=C\cap (X\setminus \Nklt(X, \Delta))$. 
\end{case} 
Therefore, we get the desired statement. 
\end{proof}

Let us prove Theorem \ref{a-thm1.8} as an application of 
Proposition \ref{i-prop9.1}. 

\begin{proof}[Proof of Theorem \ref{a-thm1.8}]
By shrinking $S$ suitably, 
we may assume that 
$X$ and $S$ are both quasi-projective. 
By Lemma \ref{c-lem3.10}, we can construct a 
projective birational morphism 
$g\colon Y\to X$ from a normal $\mathbb Q$-factorial 
variety $Y$ satisfying 
(i), (ii), and (iv) in Lemma \ref{c-lem3.10}. 
Let us consider $\pi\circ g\colon  Y\to S$. 
Note that $K_Y+\Delta_Y$ is not nef over $S$ 
since $K_Y+\Delta_Y=g^*(K_X+\Delta)$ holds. 
It is obvious that $(K_Y+\Delta_Y)|_{\Nklt(Y, \Delta_Y)}$ is nef 
over $S$ by (iv) because so is $(K_X+\Delta)|_{\Nklt(X, \Delta)}$. 
Therefore, by Proposition \ref{i-prop9.1}, 
we have a non-constant morphism 
\begin{equation*} 
h\colon \mathbb A^1\longrightarrow Y\setminus \Nklt (Y, \Delta_Y)
\end{equation*} 
such that 
$(\pi\circ g)\circ h(\mathbb A^1)$ is a point 
and that 
\begin{equation*} 
0<-(K_Y+\Delta_Y)\cdot C_Y\leq 2\dim Y=2\dim X
\end{equation*} 
holds, 
where $C_Y$ is the closure of $h(\mathbb A^1)$ in $Y$. 
Since $K_Y+\Delta_Y=h^*(K_X+\Delta)$ holds, $g$ does not 
contract $C_Y$ to a pont. 
This implies that 
\begin{equation*} 
f:=g\circ h\colon  \mathbb A^1\longrightarrow X\setminus \Nklt(X, \Delta)
\end{equation*}  
is a desired non-constant morphism 
such that $\pi\circ f(\mathbb A^1)$ is a point by 
(iv). 
\end{proof}

For the proof of Theorem \ref{a-thm1.6}, we prepare the following somewhat 
artificial statement as an application of 
Theorem \ref{a-thm1.8}. 

\begin{thm}\label{i-thm9.2}
Let $\pi\colon X\to S$ be a proper surjective morphism 
from a normal quasi-projective variety $X$ onto a scheme 
$S$. 
Let $\mathcal P$ be an $\mathbb R$-Cartier 
divisor on $X$ and let $H$ be an ample Cartier divisor 
on $X$. 
Let $\Sigma$ be a closed subset of $X$. 
Assume that $\pi$ is not finite, 
$-\mathcal P$ is $\pi$-ample, and 
$\pi\colon  \Sigma\to \pi(\Sigma)$ is finite. 
We further assume 
\begin{itemize}
\item $\{\varepsilon _i\}_{i=1}^\infty$ is 
a set of positive real numbers with $\varepsilon _i\searrow 
0$ for $i\nearrow \infty$, and 
\item for every $i$, there exists an effective $\mathbb R$-divisor 
$\Delta_i$ on $X$ such that 
\begin{equation*} 
\mathcal P+\varepsilon _i H\sim _{\mathbb R} K_X+\Delta_i
\end{equation*}  
and that 
\begin{equation*} 
\Sigma =\Nklt(X, \Delta_i)
\end{equation*} 
holds set theoretically. 
\end{itemize} 
Then there exists a non-constant morphism 
\begin{equation*} 
f\colon \mathbb A^1\longrightarrow X\setminus \Sigma 
\end{equation*}  
such that 
$\pi\circ f(\mathbb A^1)$ is a point and that 
the curve $C$, the closure of $f(\mathbb A^1)$ in $X$, 
is a rational curve with 
\begin{equation*} 
0< -\mathcal P\cdot C\leq 2\dim X. 
\end{equation*}  
\end{thm}
\begin{proof}
We take an ample $\mathbb Q$-divisor $A$ on $X$ such that 
$-(\mathcal P+A)$ is $\pi$-ample. 
Without loss of generality, we may assume that 
$-(\mathcal P+A+\varepsilon_i H)$ is $\pi$-ample for every $i$ because 
$\varepsilon _i\searrow 0$ for $i\nearrow \infty$. 
By assumption, 
\begin{equation*} 
\mathcal P+\varepsilon _iH\sim _{\mathbb R} K_X+\Delta_i
\end{equation*}  
with 
\begin{equation*} 
\Nklt(X, \Delta_i)=\Sigma 
\end{equation*}  
for every $i$. Hence, by Theorem \ref{a-thm1.8}, there is 
a non-constant morphism 
\begin{equation*} 
f_i\colon  \mathbb A^1\longrightarrow X\setminus \Sigma
\end{equation*} 
such 
that $\pi\circ f_i(\mathbb A^1)$ is a point and that 
\begin{equation*} 
0<-(K_X+\Delta_i)\cdot C_i =
-(\mathcal P+\varepsilon _i H)\cdot C_i \leq 2\dim X, 
\end{equation*}  
where $C_i$ is the closure of $f_i(\mathbb A^1)$ in $X$. 
We note that 
\begin{equation*} 
0<A\cdot C_i =\left( (\mathcal P +\varepsilon _i H+A)
-(\mathcal P+\varepsilon _i H)\right)\cdot C_i<2\dim X. 
\end{equation*}  
It follows that the curves $C_i$ belong to a bounded family. 
Thus, possibly passing to 
a subsequence, 
we may assume that $f_i$ and $C_i$ are constant, 
that is, there is a non-constant morphism 
\begin{equation*} 
f\colon \mathbb A^1\longrightarrow X\setminus \Sigma
\end{equation*}  
such that $C_i=C$ for every $i$, 
where $C$ is the closure of 
$f(\mathbb A^1)$ in $X$. 
Therefore, we get 
\begin{equation*} 
0<-\mathcal P\cdot C=
\lim _{i\to \infty}-(\mathcal P+\varepsilon_i H)\cdot C 
=\lim _{i\to \infty} -(\mathcal P+\varepsilon _iH)\cdot 
C_i\leq 2\dim X.  
\end{equation*}  
We finish the proof of Theorem \ref{i-thm9.2}. 
\end{proof}

\section{Proof of Theorems \ref{a-thm1.4}, 
\ref{a-thm1.5}, and \ref{a-thm1.6}}\label{j-sec10}

In this section, we prove Theorems \ref{a-thm1.4}, 
\ref{a-thm1.5} and 
\ref{a-thm1.6}. Since Theorem \ref{a-thm1.4} 
is an easy consequence of Theorem \ref{a-thm1.5} and 
Theorem \ref{a-thm1.5} can be seen as a very 
special case of Theorem \ref{a-thm1.6} by Example 
\ref{d-ex4.11}, it is sufficient to 
prove Theorem \ref{a-thm1.6}. 
Let us start with the proof of Theorem \ref{a-thm1.6}. 

\begin{proof}[Proof of Theorem \ref{a-thm1.6}]
We note that (i) and (ii) 
were already established in \cite[Theorem 
6.7.4]{fujino-foundations}. 
Therefore, it is sufficient to prove (iii). 
From Step \ref{j-step1.6.1} to Step \ref{j-step1.6.4}, 
we will reduce the problem to the case where 
$X$ is a normal variety. 
Then, in Step \ref{j-step1.6.5}, 
we will obtain a desired non-constant morphism 
from $\mathbb A^1$ by Theorem \ref{i-thm9.2}. 
\setcounter{step}{0}
\begin{step}\label{j-step1.6.1} 
Let $\varphi_{R_j}\colon  X\to Y$ be the extremal contraction 
associated to $R_j$ (see Theorem \ref{d-thm4.18} and 
\cite[Theorems 6.7.3 and 6.7.4]{fujino-foundations}). 
We note that 
\begin{equation*} 
\varphi_{R_j}\colon \Nqlc(X, \omega)\to 
\varphi_{R_j}\left(\Nqlc(X, \omega)\right)
\end{equation*} 
is finite. 
By replacing $\pi\colon X\to S$ with $\varphi_{R_j}\colon  X\to Y$, we may 
assume that 
$-\omega$ is $\pi$-ample and that 
$\NE(X/S)_{-\infty}=\{0\}$. 
\end{step}
\begin{step}\label{j-step1.6.2}
We take a qlc stratum $W$ of $[X, \omega]$ such that 
$\pi\colon W\to \pi(W)$ is not finite 
and that $\pi\colon W^\dag\to \pi(W^\dag)$ is finite for 
every qlc center $W^\dag$ of $[X, \omega]$ with 
$W^\dag\subsetneq W$. 
We put $W'=W\cup \Nqlc(X, \omega)$. 
Then, by adjunction (see Theorem 
\ref{d-thm4.6} (i) and 
\cite[Theorem 6.3.5 (i)]{fujino-foundations}), 
$[W', \omega|_{W'}]$ naturally becomes a quasi-log scheme. 
By replacing $[X, \omega]$ with $[W', \omega|_{W'}]$, 
we may further assume that $X\setminus X_{-\infty}$ is irreducible and that 
\begin{equation*} 
\pi\colon \Nqklt(X, \omega)\to \pi\left(\Nqklt(X, \omega)\right)
\end{equation*}  
is finite. 
\end{step}
\begin{step}\label{j-step1.6.3} 
By Lemma \ref{d-lem4.20}, we may replace 
$X$ with $\overline {X\setminus X_{-\infty}}$ and 
assume that 
$X$ is a variety. 
We note that the finiteness of 
\begin{equation*} 
\pi\colon \Nqklt(X, \omega)\to \pi\left(\Nqklt(X, \omega)\right)
\end{equation*}  
still holds. 
\end{step}
\begin{step}\label{j-step1.6.4}
Let $\nu\colon Z\to X$ be the normalization. 
Then $[Z, \nu^*\omega]$ naturally becomes a quasi-log scheme 
by Theorem \ref{a-thm1.9}. 
Since $\Nqklt(Z, \nu^*\omega)=\nu^{-1}\Nqklt(X, \omega)$ 
by Theorem \ref{a-thm1.9}, we may assume 
that $X$ is normal by replacing $[X, \omega]$ with 
$[Z, \nu^*\omega]$. 
\end{step}
\begin{step}\label{j-step1.6.5} 
By shrinking $S$ suitably, 
we may further assume that 
$X$ and $S$ are both quasi-projective. 
Hence we have the following properties: 
\begin{itemize}
\item[(a)] $\pi\colon X\to S$ is a projective morphism 
from a normal quasi-projective variety $X$ to a scheme $S$, 
\item[(b)] $-\omega$ is $\pi$-ample, and 
\item[(c)] $\pi\colon \Sigma\to \pi(\Sigma)$ is finite, 
where $\Sigma: =\Nqklt(X, \omega)$. 
\end{itemize}
Let $H$ be an ample Cartier divisor on $X$ 
and let $\{\varepsilon _i \}_{i=1}^\infty$ be 
a set of positive real numbers such 
that $\varepsilon _i\searrow 0$ for $i\nearrow \infty$. 
Then, by Theorem \ref{a-thm1.10}, 
we have: 
\begin{itemize}
\item[(d)] there exists an effective $\mathbb R$-divisor 
$\Delta_i$ on $X$ such that 
\begin{equation*} 
K_X+\Delta_i\sim _{\mathbb R} \omega+\varepsilon _i H
\end{equation*}  
with 
\begin{equation*} 
\Nklt(X, \Delta_i)=\Sigma
\end{equation*}  
for every $i$. 
\end{itemize}
Thus, by Theorem \ref{i-thm9.2}, 
we have a desired non-constant morphism 
\begin{equation*} 
f\colon  \mathbb A^1\longrightarrow X\setminus \Nqklt(X, \omega). 
\end{equation*}  
\end{step}
We finish the proof of Theorem \ref{a-thm1.6}. 
\end{proof}

As we already mentioned above, 
Theorem \ref{a-thm1.5} is a very special case of 
Theorem \ref{a-thm1.6}. 

\begin{proof}[Proof of Theorem \ref{a-thm1.5}]
By Example \ref{d-ex4.11}, $[X, K_X+\Delta]$ naturally becomes 
a quasi-log scheme. Then, by Theorem \ref{a-thm1.6}, 
the desired cone theorem holds for $(X, \Delta)$. 
\end{proof}

Theorem \ref{a-thm1.4} easily follows from Theorem \ref{a-thm1.5}. 

\begin{proof}[Proof of Theorem \ref{a-thm1.4}]
Since $(X, \Delta)$ is Mori hyperbolic by assumption, 
there is no $(K_X+\Delta)$-negative 
extremal ray of $\NE(X)$ that is 
rational and relatively ample at infinity (see Theorem \ref{a-thm1.5}). 
By assumption, $(K_X+\Delta)|_{\Nlc(X, \Delta)}$ is nef. 
Hence the subcone 
$\NE(X)_{-\infty}$ is included in 
$\NE(X)_{(K_X+\Delta)\geq 0}$. 
This implies that 
\begin{equation*} 
\NE(X)=\NE(X)_{(K_X+\Delta)\geq 0}
\end{equation*} 
holds by Theorem \ref{a-thm1.5}. 
Thus $K_X+\Delta$ is nef. 
\end{proof}

The author thinks that 
the proof of Theorems \ref{a-thm1.4}, 
\ref{a-thm1.5} and \ref{a-thm1.6} 
shows that 
the framework of quasi-log schemes established in 
\cite[Chapter 6]{fujino-foundations} and 
\cite{fujino-slc-trivial} is very powerful and 
useful even for the study of normal pairs. 

\section{Ampleness criterion for quasi-log schemes}\label{k-sec11}

The main purpose of this section is to establish the 
following ampleness criterion for quasi-log schemes. 
Then we will see that Theorem \ref{a-thm1.11} is a very special 
case of Theorem \ref{k-thm11.1}. 

\begin{thm}[Ampleness criterion for 
quasi-log schemes]\label{k-thm11.1}
Let $[X, \omega]$ be a quasi-log scheme 
and let $\pi\colon X\to S$ be a projective morphism 
between schemes. 
Assume that $\omega|_{\Nqlc(X, \omega)}$ is ample over 
$S$ and that $\omega$ is log big over $S$ with respect 
to $[X, \omega]$. 
We further assume that 
there is no non-constant morphism 
\begin{equation*} 
f\colon \mathbb A^1\longrightarrow U 
\end{equation*}  
such that $\pi\circ f(\mathbb A^1)$ is a point, 
where $U$ is any open qlc stratum of $[X, \omega]$. 
Then $\omega$ is ample over $S$. 
\end{thm}

Let us treat a special case of Theorem \ref{k-thm11.1}. 

\begin{thm}\label{k-thm11.2}
Let $[X, \omega]$ be a quasi-log scheme such that 
$X$ is a normal variety. 
Let $\pi\colon X\to S$ be a projective morphism 
onto a scheme $S$. 
Assume that $\omega|_{\Nqklt(X, \omega)}$ is ample 
over $S$ and that there is no non-constant morphism 
\begin{equation*} 
f\colon \mathbb A^1\longrightarrow X\setminus \Nqklt(X, \omega)
\end{equation*} 
such that 
$\pi\circ f(\mathbb A^1)$ is a point. 
We further assume that 
$\omega$ is big over $S$. 
Then $\omega$ is ample over $S$. 
\end{thm}

\begin{proof} 
We divide the proof into several small steps. 
\setcounter{step}{0}
\begin{step}\label{k-step11.2.1}
By Lemma \ref{d-lem4.25}, 
we can obtain quasi-log schemes 
\begin{equation*} 
\left(X, \omega_i, f\colon (Y, D_i)\to X\right)
\end{equation*}  
for $1\leq i\leq k$ with the following properties: 
\begin{itemize}
\item[(a)] $[X, \omega_i]$ has a $\mathbb Q$-structure 
for every $i$, 
\item[(b)] $\Nqlc (X, \omega_i)=\Nqlc(X, \omega)$ holds for 
every $i$, 
\item[(c)] $W$ is a qlc stratum of $[X, \omega]$ if and 
only if $W$ is a qlc stratum of $[X, \omega_i]$ 
for every $i$, and 
\item[(d)] there exist positive real numbers $r_i$ for 
$1\leq i\leq k$ such that 
$\omega=\sum _{i=1}^k r_i \omega_i$ with $\sum _{i=1}^k r_i=1$. 
\end{itemize} 
By construction, we can make $\omega_i$ sufficiently 
close to 
$\omega$ (see the proof of Lemma \ref{d-lem4.25}). 
Therefore, we may assume that 
$\omega_i|_{\Nqklt(X, \omega_i)}$ is ample over $S$ for every $i$ by 
(b) and (c). 
Thus $[X, \omega_i]$ satisfies all the assumptions for $[X, \omega]$ 
in Theorem \ref{k-thm11.2}. 
Hence, by replacing $[X, \omega]$ with $[X, \omega_i]$, 
it is sufficient to prove the ampleness of $\omega$ under 
the extra assumption that $[X, \omega]$ has a 
$\mathbb Q$-structure by (a) and (d). 
\end{step}
\begin{step}\label{k-step11.2.2} 
By assumption and Theorem \ref{a-thm1.6} (iii), 
$\omega$ is nef over $S$. 
Since $\omega|_{\Nqklt(X, \omega)}$ is ample over $S$ by 
assumption, 
$\omega$ is nef and log big over $S$ with respect to $[X, \omega]$. 
Therefore, by 
\cite[Theorem 1.1]{fujino-reid-fukuda}, 
we obtain that $\omega$ is semi-ample over $S$. 
Hence $m\omega$ gives a birational contraction morphism 
$\Phi\colon X\to Y$ between normal varieties over $S$, 
where $m$ is a sufficiently large and divisible positive integer. 
\end{step}
\begin{step}\label{k-step11.2.3} 
In this step, we will get a contradiction under 
the assumption that $\Phi$ is not an isomorphism. 

By shrinking $S$, we may assume that 
$S$, $X$, and $Y$ are quasi-projective. 
By construction, 
\begin{equation*}
\Phi\colon \Nqklt(X, \omega)\to \Phi(\Nqklt(X, \omega))
\end{equation*}  
is finite. Since $\Phi$ is birational and $Y$ is 
quasi-projective, 
we can take an effective Cartier divisor $G$ on $X$ such that 
$-G$ is $\Phi$-ample. 
By Lemma \ref{d-lem4.24}, for $0<\varepsilon \ll 1$, $[X, \omega+
\varepsilon G]$ is a quasi-log scheme such that 
\begin{equation*} 
\Nqklt(X, \omega+\varepsilon G)=\Nqklt(X, \omega)
\end{equation*}  
holds.  
By the cone theorem (see Theorem \ref{a-thm1.6} (iii)), 
we can find a non-constant morphism 
\begin{equation*} 
f\colon \mathbb A^1\longrightarrow X\setminus 
\Nqklt(X, \omega+\varepsilon G)
=X\setminus \Nqklt(X, \omega)
\end{equation*}  
such that $\pi\circ f(\mathbb A^1)$ is a point and 
that 
$0<-(\omega+\varepsilon G)\cdot C\leq 2\dim X$ holds, 
where $C$ is the closure of $f(\mathbb A^1)$ in $X$. 
This is a contradiction. 
\end{step}
Hence $\Phi$ is an isomorphism. 
Therefore, we obtain that $\omega$ is ample over 
$S$. This is what we wanted. 
\end{proof}

Once we know Theorem \ref{k-thm11.2}, 
it is not difficult to prove Theorem \ref{k-thm11.1}. 

\begin{proof}[Proof of Theorem \ref{k-thm11.1}]
By Theorem \ref{a-thm1.6} (iii), $\omega$ is nef and 
log big over $S$ with 
respect to $[X, \omega]$. 
We put 
\begin{equation*} 
[X_0, \omega_0]:=[X, \omega]
\end{equation*}  
and 
\begin{equation*} 
[X_{i+1}, \omega_{i+1}]:=[\Nqklt(X_i, \omega_i), 
\omega_i|_{\Nqklt(X_i, \omega_i)}]
\end{equation*}  
for $i\geq 0$. 
Then there exists $k\geq 0$ such that 
\begin{equation*} 
\Nqklt(X_k, \omega_k)=\Nqlc(X_k, \omega_k)=
\Nqlc(X, \omega). 
\end{equation*}  
We note that $\Nqlc(X, \omega)$ may be empty. 
By assumption, $\omega_k|_{\Nqklt(X_k, \omega_k)}$ is ample 
over $S$. 
We want to show by inverse induction on $i$ that $\omega_i$ 
is ample over $S$. 
Therefore, it is sufficient to prove 
the following claim. 
\begin{claim-n}
Let $[X, \omega]$ be a quasi-log scheme and let 
$\pi\colon X\to S$ be a projective morphism 
between schemes such that 
$\omega|_{\Nqklt(X, \omega)}$ is ample over $S$ 
and that $\omega$ is nef and log big over $S$ 
with respect to $[X, \omega]$. 
Assume that there is no non-constant morphism 
\begin{equation*} 
f\colon \mathbb A^1\longrightarrow X\setminus \Nqklt(X, \omega)
\end{equation*}  
such that $\pi\circ f(\mathbb A^1)$ is a point. 
Then $\omega$ is ample over $S$. 
\end{claim-n}
\begin{proof}[Proof of Claim]
By adjunction (see Theorem \ref{d-thm4.6} (i) and 
\cite[Theorem 6.3.5 (i)]{fujino-foundations}), 
we may assume that $\overline {X\setminus X_{-\infty}}$ is irreducible. 
By Lemma \ref{d-lem4.20}, 
we may further assume that 
$X$ is irreducible. 
Then, by Theorem \ref{a-thm1.9}, we can reduce the problem 
to the case where $X$ is a normal variety. 
Hence $\omega$ is ample over $S$ by Theorem \ref{k-thm11.2}. 
\end{proof}
As we have already mentioned above, 
by applying Claim inductively, 
we obtain the desired relative ampleness 
of $\omega=\omega_0$. 
\end{proof}

We close this section with the proof of Theorem \ref{a-thm1.11}. 

\begin{proof}[Proof of Theorem \ref{a-thm1.11}]
By Example \ref{d-ex4.11}, 
$[X, K_X+\Delta]$ naturally becomes a quasi-log scheme. 
We apply Theorem \ref{k-thm11.1} to $[X, K_X+\Delta]$. 
Then we obtain that $K_X+\Delta$ is ample. 
This is what we wanted. 
\end{proof}

The author knows no proof of 
Theorem \ref{a-thm1.11} that does not use the framework of quasi-log schemes. 
Note that a similar result for 
dlt pairs was already established in \cite[Theorem 5.1]{fujino-maximal}, 
whose proof is much 
easier than that of Theorem \ref{a-thm1.11} and 
depends on the basepoint-free 
theorem of Reid--Fukuda type 
for dlt pairs (see \cite[Theorem 0.1]{fujino-rf}). 
We recommend the interested reader to 
see \cite[Theorem 5.1]{fujino-maximal} and 
\cite[Theorem 0.1]{fujino-rf}. 

\section{Proof of Theorems \ref{a-thm1.12} and \ref{a-thm1.13}}\label{l-sec12}

In this section, we prove Theorems \ref{a-thm1.12} and \ref{a-thm1.13}, 
and explain an application for normal pairs. 
For the basic properties of uniruled varieties, 
see \cite[Chapter IV.~1]{kollar-rational}. 
Let us start with the following lemma, which is a generalization of 
\cite[Lemma]{kawamata1}. 

\begin{lem}\label{l-lem12.1}
Let $[X, \omega]$ be a quasi-log scheme and let 
$\varphi\colon  X\to W$ be a projective morphism between 
schemes. 
Let $P$ be an arbitrary closed point 
of $W$. 
Let $E$ be a positive-dimensional 
irreducible component of $\varphi^{-1}(P)$ such that  
$E\not\subset X_{-\infty}$ and 
let $\nu\colon \overline E\to E$ be the normalization. 
Then, for every ample $\mathbb R$-divisor 
$H$ on $\overline E$, 
there exists an effective $\mathbb R$-divisor $\Delta_{\overline E, H}$ on 
$\overline E$ such that 
\begin{equation*} 
\nu^*\omega+H\sim _{\mathbb R} K_{\overline E}+\Delta_{\overline E, H} 
\end{equation*} 
holds. 
Therefore, 
\begin{equation*} 
\mathcal A^{\dim E-1}\cdot \omega\cdot E\geq \left(\nu^*\mathcal A\right)
^{\dim E-1} \cdot K_{\overline E}
\end{equation*}  
holds for every $\varphi$-ample line bundle $\mathcal A$ on $X$. 

In the above statement, if $[X, \omega]$ 
has a $\mathbb Q$-structure and $H$ is an ample 
$\mathbb Q$-divisor on $\overline E$, then 
we can make $\Delta_{\overline E, H}$ an effective 
$\mathbb Q$-divisor on $\overline {E}$ with 
\begin{equation*} 
\nu^*\omega+H\sim _{\mathbb Q} K_{\overline E}+\Delta_{\overline E, H}. 
\end{equation*} 
\end{lem}

\begin{proof}
Our approach is different from Kawamata's in \cite{kawamata1}. 
A key ingredient of this proof is Theorem \ref{a-thm1.10}. 
\setcounter{step}{0}
\begin{step}\label{l-step12.1.1}
If $E$ is a qlc stratum of $[X, \omega]$, then we put $B=0$ 
and go to Step \ref{l-step12.1.3}. 
\end{step}
\begin{step}\label{l-step12.1.2}
By Step \ref{l-step12.1.1}, 
we may assume that $E$ is not a qlc stratum of $[X, \omega]$. 
Without loss of generality, we may assume that 
$W$ is quasi-projective by shrinking $W$ around $P$. 
Let $B_1, \ldots, B_{n+1}$ be general very ample Cartier divisors on 
$W$ passing through $P$ with 
$n=\dim X$. 
Let $f\colon (Y, B_Y)\to X$ be a proper morphism 
from a globally embedded simple normal crossing 
pair $(Y, B_Y)$ as in Definition \ref{d-def4.2}. 
Let $X'$ be the union of $X_{-\infty}=
\Nqlc(X, \omega)$ and all qlc strata of $[X, \omega]$ 
mapped to $P$ by $\varphi$. 
By \cite[Proposition 6.3.1]{fujino-foundations} 
and 
\cite[Theorem 3.35]{kollar}, 
we may assume that the union of all strata of $(Y, B_Y)$ mapped to 
$X'$ by $f$, which is denoted by $Y'$, is 
a union of some irreducible components of $Y$. As usual, 
we put $Y''=Y-Y'$, $K_{Y''}+B_{Y''}=(K_Y+B_Y)|_{Y''}$, 
and $f''=f|_{Y''}$. By 
\cite[Proposition 6.3.1]{fujino-foundations} 
and 
\cite[Theorem 3.35]{kollar} again, 
we may further assume that 
\begin{equation*} 
\left(Y'', (f'')^*\varphi^*\sum _{i=1}^{n+1}B_i+\Supp B_{Y''}\right)
\end{equation*}  
is a globally embedded simple normal crossing pair. 
By \cite[Lemma 6.3.13]{fujino-foundations}, 
we can take $0< c<1$ such that 
\begin{equation*} 
f''\left(\left(B_{Y''}+c(f'')^*\varphi^*\sum _{i=1}^{n+1}B_i\right)^{>1}\right)
\not\supset E
\end{equation*}  
and that there exists an irreducible component $G$ of 
\begin{equation*} 
\left(B_{Y''}+c(f'')^*\varphi^*\sum _{i=1}^{n+1}B_i\right)^{=1}
\end{equation*} 
with $f''(G)=E$. 
By Lemma \ref{d-lem4.23}, we obtain that  
\begin{equation*} 
\left(X, \omega+B, f''\colon  (Y'', B_{Y''}+(f'')^*B)\to X\right), 
\end{equation*}  
where $B=\varphi^*\left(c\sum _{i=1}^{n+1}B\right)$, 
is a quasi-log scheme such that 
$E$ is a qlc stratum of this quasi-log scheme. 
\end{step}
\begin{step}\label{l-step12.1.3}
We put $E'=E\cup \Nqlc(X, \omega+B)$. Then, by adjunction (see 
Theorem \ref{d-thm4.6} (i) and 
\cite[Theorem 6.3.5 (i)]{fujino-foundations}), 
$[E', (\omega+B)|_{E'}]$ is a quasi-log scheme. 
By Lemma \ref{d-lem4.20}, $[E, (\omega+B)|_E]$ is also a quasi-log 
scheme. 
We note that $(\omega+B)|_E\sim _{\mathbb R}\omega|_E$ since 
$\varphi(E)=P$. 
Hence $[E, \omega|_E]$ is a quasi-log scheme. 
By Theorem \ref{a-thm1.9}, 
$[\overline E, \nu^*\omega]$ naturally becomes a quasi-log scheme. 
By Theorem \ref{a-thm1.10}, 
there exists an effective $\mathbb R$-divisor 
$\Delta_{\overline E, H}$ on $\overline E$ such that 
\begin{equation*} 
\nu^*\omega+H\sim _{\mathbb R} K_{\overline E}+
\Delta_{\overline E, H}. 
\end{equation*}  
This implies that 
\begin{equation*} 
(\nu^*\mathcal A)^{\dim E-1}\cdot 
(\nu^*\omega+H)\cdot \overline E 
=(\nu^*\mathcal A)^{\dim E-1}(K_{\overline E} +\Delta_{\overline E, H}) 
\geq (\nu^*\mathcal A)^{\dim E-1} \cdot K_{\overline E}. 
\end{equation*}  
Since the above inequality holds for every ample $\mathbb R$-divisor 
$H$ on $\overline E$, 
we obtain 
\begin{equation*} 
\mathcal A^{\dim E-1}\cdot \omega\cdot E=(\nu^*\mathcal A)
^{\dim E-1} \cdot \nu^*\omega\cdot \overline E 
\geq (\nu^*\mathcal A)^{\dim E-1} \cdot K_{\overline E}. 
\end{equation*}  
This is what we wanted. 
By the above proof, it is easy to see that 
we can make $\Delta_{\overline E, H}$ an effective 
$\mathbb Q$-divisor on $\overline E$ if $[X, \omega]$ 
has a $\mathbb Q$-structure and $H$ is an ample 
$\mathbb Q$-divisor on $\overline E$. 
\end{step}
We finish the proof of Lemma \ref{l-lem12.1}. 
\end{proof}

\begin{rem}\label{l-rem12.2} 
In the proof of \cite[Lemma]{kawamata1}, 
Kawamata uses a relative Kawamata--Viehweg 
vanishing theorem for projective bimeromorphic 
morphisms between complex analytic spaces. 
His argument does not work for quasi-log schemes. 
\end{rem}

Let us prove Theorem \ref{a-thm1.12}. 

\begin{proof}[Proof of Theorem \ref{a-thm1.12}]
In this proof, we will freely use the notation of Lemma \ref{l-lem12.1}. 
\setcounter{case}{0}
\begin{case}
We will treat the case where $\dim E=1$. 

We take an ample $\mathbb Q$-divisor $H$ on $\overline E$ such that 
$-(\nu^*\omega+H)$ is still ample. 
Then, by Lemma \ref{l-lem12.1}, 
$-K_{\overline E}$ is ample since 
$\Delta_{\overline E, H}$ is effective. 
This means that $\overline E\simeq \mathbb P^1$. 
By Lemma \ref{l-lem12.1} again, we have 
\begin{equation*} 
0<-\omega\cdot E\leq -\deg K_{\overline E}=2. 
\end{equation*} 
\end{case}
\begin{case}
We will treat the case where $\dim E\geq 2$. 

We take a $\varphi$-ample line bundle 
$\mathcal A$ such that $\nu^*\mathcal A$ is very ample. 
We put $C=D_1\cap \cdots \cap D_{\dim E-1}$, 
where $D_i$ is a general member of $|\nu^*\mathcal A|$ for 
every $i$. 
Then $C$ is a smooth irreducible curve on $\overline E$ such 
that $C$ lies in the smooth locus of $\overline E$. 
By Lemma \ref{l-lem12.1}, 
we obtain 
\begin{equation*} 
C\cdot K_{\overline E}\leq \mathcal A^{\dim E-1}\cdot 
\omega\cdot E<0
\end{equation*}  
because $-\omega$ is $\varphi$-ample. 
We note that 
\begin{equation*}
\begin{split}
0>\nu^*\omega\cdot C &=\nu^*\omega\cdot 
(\nu^*\mathcal A)^{\dim E-1} \cdot \overline E \\ 
&= \omega\cdot \mathcal A^{\dim E-1} \cdot E \\ 
&\geq (\nu^*\mathcal A)^{\dim E-1} \cdot K_{\overline E} 
\\ &=C\cdot K_{\overline E}. 
\end{split}
\end{equation*}
Therefore, for any given point $x\in C$, there exists 
a rational curve $\Gamma$ on $\overline E$ passing through $x$ 
with 
\begin{equation*}
\begin{split}
0<-\nu^*\omega\cdot \Gamma &\leq 2\dim \overline E\cdot 
\frac{-\nu^*\omega\cdot C}{-K_{\overline E}\cdot C}
\\ &\leq 2\dim \overline E.  
\end{split}
\end{equation*}
This is essentially due to Miyaoka--Mori (see \cite{miyaoka-mori}). 
We note that $\overline E$ is not always smooth but 
it is smooth in a neighborhood of 
$C$. 
Hence we can use the argument of \cite{miyaoka-mori}. 
For the details, see \cite[Chapter II.~5.8 Theorem]{kollar-rational}. 
Thus, $E$ is covered by rational curves $\ell:=\nu_*\Gamma$ with 
\begin{equation*} 
0<-\omega\cdot \ell \leq 2\dim E. 
\end{equation*} 
\end{case}
Hence, by \cite[Chapter IV.~1.4 Proposition--Definition]{kollar-rational}, 
$E$ is uniruled. 
We finish the proof of Theorem \ref{a-thm1.12}. 
\end{proof}

We prove Theorem \ref{a-thm1.13}. 

\begin{proof}[Proof of Theorem \ref{a-thm1.13}]
Since $\varphi_R$ is the contraction morphism 
associated to $R$, 
\begin{equation*} 
\varphi_R\colon  \Nqlc(X, \omega)
\to \varphi_R(\Nqlc(X, \omega))
\end{equation*} 
is finite. 
We apply Theorem \ref{a-thm1.12} to $\varphi_R\colon  X\to W$, 
we can take a rational curve $\ell$ on $X$ 
such that $\varphi_R(\ell)$ is a point with 
\begin{equation*} 
0<-\omega\cdot \ell \leq 2d. 
\end{equation*}  
We finish the proof of Theorem \ref{a-thm1.13}. 
\end{proof}

We explain an application of Theorems \ref{a-thm1.12} and 
\ref{a-thm1.13} for normal pairs, which is a generalization of 
\cite[Theorem 1]{kawamata1}. 

\begin{cor}\label{l-cor12.3}
Let $X$ be a normal variety and let $\Delta$ be an effective 
$\mathbb R$-divisor on $X$ such that 
$K_X+\Delta$ is $\mathbb R$-Cartier. Let 
$\pi\colon X\to S$ be a projective morphism 
between schemes. 
Let $R$ be a $(K_X+\Delta)$-negative extremal ray of 
$\NE(X/S)$ that 
are rational and relatively ample at infinity. 
Let $\varphi_R\colon X\to W$ be the contraction morphism 
over $S$ associated to $R$. 
We put 
\begin{equation*} 
d=\min _E \dim E, 
\end{equation*}  
where $E$ runs over positive-dimensional 
irreducible components of $\varphi^{-1}_R(P)$ for all $P\in W$. 
Then $R$ is spanned by a {\em{(}}possibly singular{\em{)}} rational 
curve $\ell$ with 
\begin{equation*} 
0<-(K_X+\Delta)\cdot \ell\leq 2d. 
\end{equation*} 
Furthermore, if $\varphi_R$ is birational and 
$(X, \Delta)$ is kawamata log terminal, 
then $R$ is spanned by a {\em{(}}possibly singular{\em{)}} rational 
curve $\ell$ with 
\begin{equation*} 
0<-(K_X+\Delta)\cdot \ell< 2d. 
\end{equation*} 

Let $V$ be an irreducible component 
of the degenerate locus 
\begin{equation*} 
\left\{x\in X\, | \, \text{$\varphi_R$ is not an isomorphism 
at $x$}\right\} 
\end{equation*}  
of $\varphi_R$. 
Then $V$ is uniruled. 
\end{cor}
\begin{proof}
We divide the proof into three small steps. 
\setcounter{step}{0}
\begin{step}\label{l-step12.3.1}
By Example \ref{d-ex4.11}, $[X, K_X+\Delta]$ 
naturally becomes a quasi-log scheme. 
By applying Theorem \ref{a-thm1.13} to $[X, K_X+\Delta]$, 
we see that $R$ is spanned by a rational curve $\ell$ 
with 
\begin{equation*} 
0<-(K_X+\Delta)\cdot \ell \leq 2d. 
\end{equation*}  
\end{step}
\begin{step}\label{l-step12.3.2}
When $(X, \Delta)$ is kawamata log terminal and 
$\varphi_R$ is a birational contraction, we take 
a $d$-dimensional irreducible component $E$ 
of $\varphi^{-1}_R(P)$ for some 
$P\in W$. 
By shrinking $W$ around $P$, we may assume that 
$W$ is affine. 
Since $\varphi_R$ is birational, there exists an effective 
$\mathbb Q$-divisor $G$ on $X$ such that 
$(X, \Delta+G)$ is kawamata log terminal and 
that $-G$ is $\varphi_R$-ample. 
By applying Theorem \ref{a-thm1.12} to $[X, K_X+\Delta+G]$, 
we see that $E$ is covered by rational curves $\ell$ with 
\begin{equation*} 
0<-(K_X+\Delta+G)\cdot \ell \leq 2d. 
\end{equation*}  
Since $-G$ is $\varphi_R$-ample, we have 
\begin{equation*} 
0<-(K_X+\Delta)\cdot \ell< 2d. 
\end{equation*} 
This implies that $R$ is spanned by a rational curve 
$\ell$ with 
\begin{equation*} 
0<-(K_X+\Delta)\cdot \ell <2d 
\end{equation*}  
when $(X, \Delta)$ is kawamata log terminal and $\varphi_R$ is 
birational. 
\end{step}
\begin{step}\label{l-step12.3.3}
From now on, we will check that $V$ is uniruled. 
We shrink $W$ around the generic point of $\varphi_R(V)$ 
and assume that $W$ is quasi-projective. 
Then we take a sufficiently ample Cartier divisor $H$ on 
$W$ such that $-(K_X+\Delta)+\varphi^*_RH$ is ample. 
By Theorem \ref{a-thm1.12}, 
$V\cap \varphi^{-1}_R(P)$ is covered by rational curves $\ell$ 
of $-\left((K_X+\Delta)+\varphi^*_RH\right)$-degree at 
most $2\dim V$ for every $P\in \varphi_R(V)
\subset W$. 
We take a suitable projective completion $\overline X$ of $X$ and 
apply \cite[Chapter IV.~1.4 Proposition--Definition]{kollar-rational}. 
Then we obtain that $V$ is uniruled.  
\end{step}
We finish the proof of Corollary \ref{l-cor12.3}. 
\end{proof}

\section{Proof of Theorem \ref{a-thm1.14}}\label{m-sec13}

In this section, we prove Theorem \ref{a-thm1.14}. 
Let us start with the following definition. 

\begin{defn}[{\cite[Definition 6.8.1]{fujino-foundations}}]\label{m-def13.1} 
Let $[X, \omega]$ be a quasi-log scheme and 
let $\pi\colon X\to S$ be a projective morphism between schemes. 
If $-\omega$ is ample over $S$, then $[X, \omega]$ 
is called a {\em{relative quasi-log Fano scheme over $S$}}. 
When $S$ is a point, we simply say that 
$[X, \omega]$ is a {\em{quasi-log Fano scheme}}. 
\end{defn}

We recall an easy consequence of the vanishing theorem 
(see Theorem \ref{d-thm4.6} (ii)), which is missing in 
\cite[Section 6.8]{fujino-foundations}. 

\begin{lem}\label{m-lem13.2}
Let $[X, \omega]$ be a quasi-log scheme and let 
$\pi\colon X\to S$ be a proper morphism between schemes with 
$\pi_*\mathcal O_X\simeq \mathcal O_S$. 
Assume that $-\omega$ is nef and log big over $S$ with 
respect to $[X, \omega]$. 
Then $X_{-\infty}\cap \pi^{-1}(P)$ is connected for 
every closed point $P\in S$. 
\end{lem}
\begin{proof}
By Theorem \ref{d-thm4.6} (ii), 
$R^1\pi_*\mathcal I_{X_{-\infty}}=0$. 
Therefore, the restriction map 
\begin{equation*} 
\mathcal O_S\simeq \pi_*\mathcal O_X\to \pi_*\mathcal O_{X_{-\infty}} 
\end{equation*}  
is surjective. This implies that 
$X_{-\infty}\cap \pi^{-1}(P)$ is connected 
for every closed point $P\in S$. 
\end{proof}

Lemma \ref{m-lem13.2} should have been stated in 
\cite[Lemma 6.8.3]{fujino-foundations}, which 
plays an important role in this section. 
The main ingredient of the proof of Theorem \ref{a-thm1.14} 
is the following theorem. 

\begin{thm}[{\cite[Theorem 1]{zhang}, 
\cite{hacon-mckernan}, and \cite[Corollary 1.4]{bp}}]
\label{m-thm13.3}
Let $X$ be a normal projective variety and let $\Delta$ be 
an effective $\mathbb R$-divisor on $X$ such that 
$K_X+\Delta$ is $\mathbb R$-Cartier. 
Assume that $-(K_X+\Delta)$ is ample. 
Then $X$ is rationally chain connected modulo 
$\Nklt(X, \Delta)$. 
\end{thm}
\begin{proof}
We take an effective $\mathbb Q$-divisor 
$\Delta'$ on $X$ such that 
$K_X+\Delta'$ is $\mathbb Q$-Cartier, $-(K_X+\Delta')$ 
is ample, and $\Nklt(X, \Delta')=\Nklt(X, \Delta)$ holds. 
If $\Nklt(X, \Delta')=\emptyset$, that is, $(X, \Delta')$ is 
kawamata log terminal, then $X$ is 
rationally connected by \cite[Theorem 1]{zhang}. 
In particular, $X$ is rationally chain connected 
by Lemma \ref{b-lem2.12}. 
When $\Nklt(X, \Delta')\ne\emptyset$, 
by applying \cite[Corollary 1.4]{bp} to $(X, \Delta')$, 
we obtain that for any general point $x$ of $X$ there exists 
a rational curve $R_x$ passing through $x$ and 
intersecting $\Nklt(X, \Delta')$. 
By \cite[Chapter II.~2.4 Corollary]{kollar-rational}, 
for every $x\in X$, we can find a chain of rational curves 
$R_x$ such that 
$x\in R_x$ and $R_x\cap \Nklt(X, \Delta')\ne\emptyset$. 
Hence $X$ is rationally chain connected modulo 
$\Nklt(X, \Delta)$. 
We note that if $-(K_X+\Delta)$ is an ample $\mathbb Q$-divisor 
then the proof of \cite[Theorem 1.2 and Corollary 1.4]{bp} 
becomes much simpler than the general case. 
Hence we obtain that $X$ is rationally chain connected 
modulo $\Nklt(X, \Delta)$. 
\end{proof}

We prepare one useful lemma. 

\begin{lem}\label{m-lem13.4}
Let $[X, \omega]$ be a projective 
quasi-log canonical pair such that 
$\Nqklt(X, \omega)=\emptyset$, $-\omega$ is 
ample, and $X$ is connected. 
Then $X$ is rationally connected. 
Hence $X$ is rationally chain connected. 
\end{lem}

\begin{proof}
By Lemma \ref{d-lem4.8} and 
\cite[Theorem 6.3.11 (i)]{fujino-foundations}, 
$X$ is a normal variety. By Theorem \ref{a-thm1.10}, 
we can find an effective $\mathbb R$-divisor 
$\Delta$ on $X$ such that 
$-(K_X+\Delta)$ is ample 
with $\Nklt(X, \Delta)=\emptyset$. 
Hence $X$ is rationally connected by \cite[Theorem 1]{zhang} 
(see the proof of Theorem \ref{m-thm13.3}). 
\end{proof}

By the framework of quasi-log schemes, 
we can prove the following lemma 
as a generalization of Theorem \ref{m-thm13.3} 
without any difficulties. We note that 
if $\Nqlc(X, \omega)=\emptyset$ in Lemma \ref{m-lem13.5} then 
it is nothing but \cite[Theorem 1.7]{fujino-subadjunction}. 
For semi-log canonical Fano pairs, we recommend the reader 
to see 
\cite{fujino-wenfei}. 

\begin{lem}\label{m-lem13.5} 
Let $[X, \omega]$ be a projective quasi-log scheme such that 
$X$ is connected. 
Assume that $-\omega$ is ample. 
Then $X$ is rationally chain connected modulo $X_{-\infty}$. 
\end{lem}

\begin{proof}
As in the proof of Theorem \ref{k-thm11.1}, 
we put 
\begin{equation*} 
[X_0, \omega_0]:=[X, \omega]
\end{equation*}  
and 
\begin{equation*} 
[X_{i+1}, \omega_{i+1}]:=[\Nqklt(X_i, \omega_i), 
\omega_i|_{\Nqklt(X_i, \omega_i)}]
\end{equation*}  
for $i\geq 0$. 
Then there exists $k\geq 0$ such that 
\begin{equation*} 
\Nqklt(X_k, \omega_k)=\Nqlc(X_k, \omega_k)=
\Nqlc(X, \omega)=X_{-\infty}. 
\end{equation*}  
We note that if $X_{-\infty}=\emptyset$, that is, 
$[X, \omega]$ is quasi-log canonical, then 
$X_k$ is the unique minimal qlc stratum of $[X, \omega]$ 
by \cite[Theorem 6.8.3 (ii)]{fujino-foundations}. 
By applying Lemma \ref{m-lem13.4} to $[X_k, \omega_k]$, we 
obtain that $X_k$ is rationally connected when $X_{-\infty}=\emptyset$. 
We want to show by inverse induction on $i$ that $X_{i+1}$ is 
rationally chain connected modulo $X_{-\infty}
=\Nqlc(X, \omega)$. 
Note that we want to show that $X$ is rationally chain 
connected modulo $X_k$ when 
$X_{-\infty}=\emptyset$. 
We also note that $X_i$ is connected by 
Lemma \ref{m-lem13.2} and  
\cite[Theorem 6.8.3]{fujino-foundations}. 
Hence it is sufficient to prove the following claim. 
\begin{claim-n}
Let $[X, \omega]$ be a quasi-log scheme such that 
$X$ is connected, $\Nqklt(X, \omega)\ne \emptyset$, and 
$-\omega$ is 
ample. Then $X$ is rationally chain connected modulo $\Nqklt(X, \omega)$. 
\end{claim-n}
\begin{proof}[Proof of Claim]
By adjunction (see Theorem \ref{d-thm4.6} 
(i) and \cite[Theorem 6.3.5 (i)]{fujino-foundations}) and 
\cite[Theorem 6.8.3]{fujino-foundations}, 
we may assume that $\overline {X\setminus X_{-\infty}}$ is irreducible. 
We note that every qlc stratum of $[X, \omega]$ 
intersects with $\Nqklt(X, \omega)$ 
(see \cite[Theorem 6.8.3]{fujino-foundations}). 
By Lemma \ref{d-lem4.20}, we may further assume that 
$X$ itself is irreducible. 
Then, by Theorem \ref{a-thm1.9}, 
we can further reduce the problem to the case where $X$ 
is a normal variety. 
Then, by Theorem \ref{a-thm1.10}, 
we can take an ample $\mathbb R$-divisor 
$H$ on $X$ such that 
$-(\omega+H)$ is still ample and that 
\begin{equation*} 
K_X+\Delta\sim _{\mathbb R} \omega+H
\end{equation*}  
holds for some effective $\mathbb R$-divisor $\Delta$ 
on $X$ with 
\begin{equation*} 
\Nklt(X, \Delta)=\Nqklt(X, \omega). 
\end{equation*}  
By applying Theorem \ref{m-thm13.3} to 
$(X, \Delta)$, 
we obtain that $X$ is rationally chain connected modulo 
$\Nqklt(X, \omega)$. We finish the proof of Claim. 
\end{proof}
By using Claim inductively, we can check that 
$X$ is rationally chain connected modulo $X_{-\infty}=
\Nqlc(X, \omega)$. 
\end{proof}

Let us prove Theorem \ref{a-thm1.14}. 

\begin{proof}[Proof of Theorem \ref{a-thm1.14}]
When $\pi^{-1}(P)\cap X_{-\infty}=\emptyset$, we may assume that 
$X_{-\infty}=\emptyset$ by shrinking 
$X$ around $\pi^{-1}(P)$. 
We divide the proof into several steps. 
\setcounter{step}{0} 
\begin{step}\label{m-step1.13.1} 
Let $X_0$ be the union of $X_{-\infty}$ and all qlc strata of 
$[X, \omega]$ contained in $\pi^{-1}(P)$. 
By Lemma \ref{m-lem13.2} and 
\cite[Theorem 6.8.3]{fujino-foundations}, 
$X_0\cap \pi^{-1}(P)$ is connected. 

\setcounter{case}{0}
\begin{case}\label{m-case1.13.1}
If $X_0\ne X_{-\infty}$, then $[X_0, \omega_0]$, where 
$\omega_0=\omega|_{X_0}$, is a quasi-log scheme 
by adjunction (see Theorem \ref{d-thm4.6} (i) and 
\cite[Theorem 6.3.5 (i)]{fujino-foundations}). 
Let us consider $X^\dag_0=\overline 
{X_0\setminus \Nqlc(X_0, \omega_0)}$. 
Then $[X^\dag_0, \omega^\dag_0]$, 
where $\omega^\dag_0=\omega|_{X^\dag_0}$, is a 
quasi-log scheme by Lemma \ref{d-lem4.20}. 
By construction, $-\omega^\dag_0$ is ample 
since $\pi(X^\dag_0)=P$. 
Therefore, by Lemma \ref{m-lem13.5}, $X^\dag_0$ is rationally 
chain connected modulo $\Nqlc(X^\dag_0, \omega^\dag_0)$. 
This means that 
$X_0\cap \pi^{-1}(P)$ is rationally chain connected modulo $\pi^{-1}(P)\cap 
X_{-\infty}$. 
\end{case}
\begin{case}\label{m-case1.13.2}
If $X_0=X_{-\infty}$, that is, there is no qlc stratum of $[X, \omega]$ 
contained in $\pi^{-1}(P)$, 
then $X_0\cap \pi^{-1}(P)$ is obviously rationally chain 
connected modulo $\pi^{-1}(P)\cap X_{-\infty}$ because 
$X_0\cap \pi^{-1}(P)=\pi^{-1}(P)\cap X_{-\infty}$. 
Note that $X_0=X_{-\infty}=\emptyset$ may happen. 
\end{case}
\end{step}
Hence $\pi^{-1}(P)$ is rationally chain connected modulo 
$\pi^{-1}(P)\cap X_{-\infty}$ when $\pi^{-1}(P)\subset X_0$. 
Thus, from now on, 
we may assume that $\pi^{-1}(P)\not \subset X_0$. 
\begin{step}\label{m-step1.13.2}
Without loss of generality, we may assume that 
$S$ is quasi-projective by shrinking $S$ around $P$. 
We take general very ample Cartier divisors $B_1, \ldots, B_{n+1}$ 
passing through $P$ with $n=\dim X$. 
Let $f\colon (Y, B_Y)\to X$ be a proper morphism 
from a globally embedded simple normal crossing 
pair $(Y, B_Y)$ as in Definition \ref{d-def4.2}. 
By \cite[Proposition 6.3.1]{fujino-foundations} 
and 
\cite[Theorem 3.35]{kollar}, 
we may assume that the union of all strata of $(Y, B_Y)$ mapped to 
$X_0$ by $f$, which is denoted by $Y'$, is 
a union of some irreducible components of $Y$. As usual, 
we put $Y''=Y-Y'$, $K_{Y''}+B_{Y''}=(K_Y+B_Y)|_{Y''}$, 
and $f''=f|_{Y''}$. By 
\cite[Proposition 6.3.1]{fujino-foundations} 
and 
\cite[Theorem 3.35]{kollar} again, 
we may further assume that 
\begin{equation*} 
\left(Y'', (f'')^*\pi^*\sum _{i=1}^{n+1}B_i+\Supp B_{Y''}\right)
\end{equation*}  
is a globally embedded simple normal crossing pair. 
By \cite[Lemma 6.3.13]{fujino-foundations}, 
we can take $0<c_1<1$ such that 
\begin{equation*} 
f''\left(\left(B_{Y''}+c_1(f'')^*\pi^*
\sum _{i=1}^{n+1}B_i\right)^{>1}\right)
=X_0 
\end{equation*} 
holds set theoretically 
and that there exists an irreducible component $G$ of 
\begin{equation*} 
\left(B_{Y''}+c_1(f'')^*\pi^*\sum _{i=1}^{n+1}B_i\right)^{=1}
\end{equation*} 
with $f''(G)\not\subset X_0$. 
By Lemma \ref{d-lem4.23}, 
\begin{equation*} 
\left(X, \omega+c_1B, f''\colon  (Y'', B_{Y''}+c_1(f'')^*B)\to X\right), 
\end{equation*}  
where $B=\pi^*\left(\sum _{i=1}^{n+1}B\right)$, 
is a quasi-log scheme. 

Let $X_1$ be the union of $\Nqlc(X, \omega+c_1B)$ and all 
qlc strata of $[X, \omega+c_1B]$ contained in 
$\pi^{-1}(P)$. 
By construction, $\Nqlc(X, \omega+c_1B)=X_0$ holds set theoretically. 
Therefore, by Case \ref{m-case1.13.1} in 
Step \ref{m-step1.13.1}, 
$X_1\cap \pi^{-1}(P)$ is rationally chain connected modulo 
$X_0\cap \pi^{-1}(P)$. 
We note that by Step \ref{m-step1.13.1} 
$X_0\cap \pi^{-1}(P)$ is rationally chain connected 
modulo $\pi^{-1}(P)\cap X_{-\infty}$. 
This means that $X_1\cap \pi^{-1}(P)$ is rationally chain connected 
modulo $\pi^{-1}(P)\cap X_{-\infty}$. 
\end{step}
\begin{step}\label{m-step1.13.3}
By repeating the argument in Step \ref{m-step1.13.2}, 
we can construct a finite increasing 
sequence of positive real numbers 
\begin{equation*} 
0<c_1<\cdots <c_k<1 
\end{equation*}  
and closed subschemes 
\begin{equation*} 
X_1\subsetneq \cdots \subsetneq X_k
\end{equation*}  
of $X$ with the following properties: 
\begin{itemize}
\item[(a)] $[X_i, \omega_i]$ is a quasi-log scheme, 
where $\omega_i=(\omega+c_iB)|_{X_i}$, for every $i$, 
\item[(b)] $\Nqlc(X_{i+1}, \omega_{i+1})=X_i$ holds 
set theoretically for every $i$, 
\item[(c)] $\pi^{-1}(P)\subset X_k$ holds, and 
\item[(d)] $X_{i+1}\cap \pi^{-1}(P)$ is rationally chain connected 
modulo $X_i\cap \pi^{-1}(P)$ for every $i$. 
\end{itemize}
Hence we obtain that 
$\pi^{-1}(P)=\pi^{-1}(P)\cap X_k$ is rationally chain connected 
modulo $\pi^{-1}(P)\cap X_{-\infty}$. 
\end{step} 
We finish the proof of Theorem \ref{a-thm1.14}. 
\end{proof}

\section{Towards Conjecture \ref{a-conj1.15}}\label{n-sec14}

In this final section, 
we treat several results related to Conjecture \ref{a-conj1.15}. 
This section needs some deep results on 
the theory of minimal models 
for higher-dimensional algebraic varieties. 
Let us start with the following special case of 
the flip conjecture II. 

\begin{conj}[Termination of flips]\label{n-conj14.1} 
Let $(X, \Delta)$ be a $\mathbb Q$-factorial 
klt pair and let $\pi\colon X\to S$ be a projective surjective morphism 
between normal quasi-projective varieties such that 
$K_X+\Delta$ is not pseudo-effective over $S$. 
Let 
\begin{equation*} 
(X, \Delta)=:(X_0, \Delta_0)\dashrightarrow 
(X_1, \Delta_1)\dashrightarrow \cdots \dashrightarrow 
(X_i, \Delta_i)\dashrightarrow \cdots
\end{equation*}  
be a sequence of flips over $S$ starting from 
$(X, \Delta)$. 
Then it terminates after finitely many steps. 
\end{conj}

In this section, we establish the following theorem, 
which is a precise version of Theorem \ref{a-thm1.16}. 

\begin{thm}[see Theorem \ref{a-thm1.16}]\label{n-thm14.2} 
Assume that Conjecture \ref{n-conj14.1} holds true in 
dimension at most $\dim \pi^{-1}(P)$. 
Then Conjecture \ref{a-conj1.15} holds true. 
\end{thm}

For the proof of Theorem \ref{n-thm14.2}, 
we prepare a variant of Theorem \ref{a-thm1.8}. 
We need the termination of flips in this theorem. 

\begin{thm}\label{n-thm14.3}  
Let $X$ be a normal variety and let $\Delta$ be an effective 
$\mathbb R$-divisor on $X$ such that 
$K_X+\Delta$ is $\mathbb R$-Cartier. 
Assume that Conjecture \ref{n-conj14.1} holds true in $\dim X$. 
Let $\pi\colon X\to S$ be a projective morphism 
onto a scheme $S$ such that 
$-(K_X+\Delta)$ is $\pi$-ample with $\dim S<\dim X$.  
We assume that $\Nklt(X, \Delta)$ is not empty such that 
\begin{equation*} 
\pi\colon \Nklt(X, \Delta)\to \pi(\Nklt(X, \Delta))
\end{equation*}  
is finite. 
Then there exists a non-constant morphism 
\begin{equation*} 
f\colon  \mathbb A^1\longrightarrow X\setminus \Nklt(X, \Delta)
\end{equation*}  
such that $\pi\circ f(\mathbb A^1)$ is a point 
and that 
the curve $C$, the closure of $f(\mathbb A^1)$ in $X$, 
is a {\em{(}}possibly singular{\em{)}} rational curve 
satisfying  
$C\cap \Nklt(X, \Delta)\ne \emptyset$ with 
\begin{equation*} 
0<-(K_X+\Delta)\cdot C\leq 1. 
\end{equation*} 
\end{thm}

\begin{proof} 
By shrinking $S$ suitably, we may assume that 
$X$ and $S$ are both quasi-projective. 
By Lemma \ref{c-lem3.10}, we can construct a projective 
birational morphism $g\colon  Y\to X$ from a normal 
$\mathbb Q$-factorial variety satisfying (i), (ii), and (iv) in Lemma 
\ref{c-lem3.10}. Since $K_Y+\Delta_Y=g^*(K_X+\Delta)$, 
$(K_Y+\Delta_Y)|_{\Nklt(Y, \Delta_Y)}$ is nef 
over $S$ by Lemma \ref{c-lem3.10} (iv). 
Let us consider $\pi\circ g\colon  Y\to S$. 
By construction, $(Y, \Delta^{<1}_Y)$ is a $\mathbb Q$-factorial 
klt pair. 
Since $-(K_X+\Delta)$ is $\pi$-ample by assumption, 
$K_Y+\Delta_Y$ is not pseudo-effective over $S$. 
Hence $K_Y+\Delta^{< 1}_Y$ is not pseudo-effective over $S$. 
Since $(K_Y+\Delta_Y)|_{\Nklt(Y, \Delta_Y)}$ is nef over 
$S$, 
the cone theorem 
\begin{equation*} 
\NE(Y/S)=\NE(Y/S)_{(K_Y+\Delta_Y)\geq 0}
+\sum _j R_j 
\end{equation*}  
holds by \cite[Theorem 1.1]{fujino-fund} (see also 
Theorem \ref{a-thm1.5} (i)). 
Since $K_Y+\Delta_Y$ is not pseudo-effective over 
$S$, $K_Y+\Delta_Y$ is not nef over $S$. 
Hence we have a $(K_Y+\Delta_Y)$-negative extremal ray 
$R$ of $\NE(Y/S)$. 
Then we consider the contraction morphism 
$\varphi_R\colon Y\to W$ over $S$ associated to 
$R$ (see \cite[Theorem 1.1]{fujino-fund} and 
Theorem \ref{d-thm4.18}). 
We note that $-(K_Y+\Delta^{<1}_Y)\cdot R>0$ 
since $(K_Y+\Delta_Y)|_{\Nklt(Y, \Delta_Y)}$ is nef over 
$S$. 
If $\varphi_R$ is an isomorphism in a neighborhood of 
$\Nklt(Y, \Delta_Y)$, then we can run a minimal model program with 
respect to $K_Y+\Delta_Y$ over $S$ by \cite{bchm}. 
Thus we run a minimal model program with respect to $K_Y+\Delta_Y$ over 
$S$. Then 
we have a sequence of flips and divisorial contractions 
\begin{equation*} 
\xymatrix{
Y=:Y_0\ar@{-->}[r]^-{\phi_0}
& Y_1\ar@{-->}[r]^-{\phi_1} & \cdots \ar@{-->}[r]^-{\phi_{i-1}}
& Y_i \ar@{-->}[r]^-{\phi_i}&\cdots 
} 
\end{equation*} 
over $S$. 
As usual, we put 
$(Y_0, \Delta_{Y_0}):=(Y, \Delta_Y)$ and 
$\Delta_{Y_{i+1}}={\phi_i}_*\Delta_{Y_i}$ for 
every $i$. 
\setcounter{case}{0}
\begin{case}\label{n-case14.3.1}
We assume that $\phi_i$ is an isomorphism in a neighborhood of 
$\Nklt(Y_i, \Delta_{Y_i})$ for every $i$. Then this minimal model 
program is a minimal model program with respect to 
$K_Y+\Delta^{<1}_Y$. Hence, by Conjecture \ref{n-conj14.1}, 
we finally get the following diagram 
\begin{equation*} 
\xymatrix{
Y=:Y_0\ar@{-->}[r]^-{\phi_0}\ar[ddr]_-{\pi\circ g}
& Y_1\ar@{-->}[r]^-{\phi_1} & \cdots \ar@{-->}[r]^-{\phi_{k-1}}
& Y_k\ar[d]^-p\\
& & &  Z\ar[lld] \\ 
& S &&
}
\end{equation*} 
where $\phi_i$ is a flip or a divisorial contraction 
for every $i$ and 
$p\colon Y_k\to Z$ is a Fano contraction over $S$. 
We note that 
$(K_{Y_k}+\Delta_{Y_k})|_{\Nklt(Y_k, \Delta_{Y_k})}$ is nef over 
$S$. 
By Case \ref{i-case9.1.1} in the proof of Proposition \ref{i-prop9.1}, 
we can find a curve $C_{Y_k}\simeq \mathbb P^1$ 
on $Y_k$ such that 
$p(C_{Y_k})$ is a point, 
$C_{Y_k}\cap \Nklt(Y_k, \Delta_{Y_k})$ is a point, 
and $0<-(K_{Y_k}+\Delta_{Y_k})\cdot C_{Y_k}\leq 1$ holds. 
By using the negativity lemma, 
we can check that 
\begin{equation*} 
-(K_Y+\Delta_Y)\cdot C_Y\leq 
-(K_{Y_k}+\Delta_{Y_k})\cdot C_{Y_k}\leq 1
\end{equation*}  
holds, where $C_Y$ is the strict transform of $C_{Y_k}$ on $Y$. 
Note that $C_Y\cap \Nklt(Y, \Delta_Y)$ 
is a point since $\phi_i$ is an isomorphism in a neighborhood of 
$\Nklt(Y_i, \Delta_{Y_i})$ for every $i$. 
Therefore, $C=g(C_Y)$ is a curve on $X$ such that 
$C\cap \Nklt(X, \Delta)$ is a point by Lemma 
\ref{c-lem3.10} (iv) with 
$0<-(K_X+\Delta)\cdot C\leq 1$. 
Hence we can construct a morphism 
\begin{equation*} 
f\colon \mathbb A^1\longrightarrow X\setminus \Nklt(X, \Delta) 
\end{equation*}  
such that $f(\mathbb A^1)= C\cap (X\setminus \Nklt(X, \Delta))$. 
This is a desired morphism. 
\end{case}
\begin{case}\label{n-case14.3.2}
We assume that there exists $i_0$ such that 
$\phi_i$ is an isomorphism in a neighborhood of 
$\Nklt(Y_i, \Delta_{Y_i})$ for 
$0\leq i<i_0$ and $\phi_{i_0}$ is not an isomorphism 
in a neighborhood of $\Nklt(Y_{i_0}, \Delta_{Y_{i_0}})$. 
Then, by using Case \ref{i-case9.1.3} in the proof of 
Proposition \ref{i-prop9.1}, we can find a curve 
$C_{Y_{i_0}}\simeq \mathbb P^1$ on $Y_{i_0}$ such that 
$C_{Y_{i_0}}\cap \Nklt(Y_{i_0}, \Delta_{Y_{i_0}})$ is a 
point, $C_{Y_{i_0}}$ is mapped to a point on $S$, and 
$0<-(K_{Y_{i_0}}+\Delta_{Y_{i_0}})\cdot 
C_{Y_{i_0}}\leq 1$ holds.  
By the same argument as in Case \ref{n-case14.3.1} above, we get 
a desired morphism 
\begin{equation*} 
f\colon  \mathbb A^1\longrightarrow X\setminus \Nklt(X, \Delta). 
\end{equation*} 
\end{case}
We finish the proof of Theorem \ref{n-thm14.3}. 
\end{proof}

By Theorem \ref{n-thm14.3}, we have:  

\begin{thm}\label{n-thm14.4} 
In Theorem \ref{i-thm9.2}, 
we further assume that 
$\dim S<\dim X$ and that $\Sigma\ne \emptyset$. 
If Conjecture \ref{n-conj14.1} holds true in 
$\dim X$, then there exists a non-constant morphism 
\begin{equation*} 
f\colon \mathbb A^1\longrightarrow X\setminus \Sigma 
\end{equation*}  
such that 
$\pi\circ f(\mathbb A^1)$ is a point and that 
the curve $C$, the closure of $f(\mathbb A^1)$ in $X$, 
is a rational curve satisfying $C\cap \Sigma\ne\emptyset$ with 
\begin{equation*} 
0< -\mathcal P\cdot C\leq 1. 
\end{equation*}  
\end{thm}
\begin{proof} 
We use Theorem \ref{n-thm14.3} instead of Theorem \ref{a-thm1.8}. 
Then the proof of Theorem \ref{i-thm9.2} implies 
the existence of 
\begin{equation*} 
f\colon \mathbb A^1\longrightarrow X\setminus \Sigma
\end{equation*} 
with the desired properties. 
\end{proof}

For the proof of Theorem \ref{n-thm14.2}, 
we establish the following somewhat technical lemma. 

\begin{lem}\label{n-lem14.5} 
Let $\pi\colon X\to S$ be a projective surjective morphism 
between normal quasi-projective 
varieties with $\pi_*\mathcal O_X\simeq 
\mathcal O_S$ and $\dim S>0$ 
and let $[X, \omega]$ be a quasi-log scheme 
such that 
\begin{equation*} 
\pi\colon \Nqklt(X, \omega)\to \pi(\Nqklt(X, \omega))
\end{equation*}  
is finite. Let $P$ be a closed point of $S$ such that 
$\dim \pi^{-1}(P)>0$. 
Then there exists an effective $\mathbb R$-Cartier divisor 
$B$ on $S$ such that 
$[X, \omega+\pi^*B]$ is a quasi-log 
scheme with the following properties: 
\begin{itemize}
\item[(i)] $\Nqklt(X, \omega)\subset \Nqklt(X, \omega+
\pi^*B)$, 
\item[(ii)] $\Nqklt(X, \omega)=\Nqklt(X, \omega+\pi^*B)$ holds 
outside $\pi^{-1}(P)$, 
\item[(iii)] 
$\pi\colon \Nqlc(X, \omega+\pi^*B)\to 
\pi(\Nqlc(X, \omega+\pi^*B))$ is finite, and 
\item[(iv)] there exists a positive-dimensional 
qlc center of $[X, \omega+\pi^*B]$ in $\pi^{-1}(P)$. 
\end{itemize} 
We further assume that 
$-\omega$ is $\pi$-ample. 
Let $W$ be a positive-dimensional qlc center of $[X, \omega+
\pi^*B]$ with $\pi(W)=P$. 
Let $\nu\colon W^\nu\to W$ be the normalization. 
Then 
$[W^\nu, \nu^*\omega]$ naturally becomes a quasi-log 
Fano scheme such that 
\begin{equation*} 
\nu^{-1}\left(\Nqklt(X, \omega)\cap \pi^{-1}(P)\right)\subset 
\Nqklt(W^\nu, \nu^*\omega)
\end{equation*}  
holds set theoretically. 
\end{lem}

\begin{proof} 
Let $B_1, \ldots, B_{n+1}$ be general very ample Cartier divisors on 
$S$ passing through $P$ with 
$n=\dim X$. 
Let $f\colon (Y, B_Y)\to X$ be a proper morphism 
from a globally embedded simple normal crossing 
pair $(Y, B_Y)$ as in Definition \ref{d-def4.2}. 
Let $X'$ be the union of $\Nqlc(X, \omega)$ and 
all qlc centers of $[X, \omega]$ mapped to $P$ by $\pi$. 
By \cite[Proposition 6.3.1]{fujino-foundations} 
and 
\cite[Theorem 3.35]{kollar}, 
we may assume that the union of all strata of $(Y, B_Y)$ mapped to 
$X'$ by $f$, which is denoted by $Y'$, is 
a union of some irreducible components of $Y$. As usual, 
we put $Y''=Y-Y'$, $K_{Y''}+B_{Y''}=(K_Y+B_Y)|_{Y''}$, 
and $f''=f|_{Y''}$. By 
\cite[Proposition 6.3.1]{fujino-foundations} 
and 
\cite[Theorem 3.35]{kollar} again, 
we may further assume that 
\begin{equation*} 
\left(Y'', (f'')^*\pi^*\sum _{i=1}^{n+1}B_i+\Supp B_{Y''}\right)
\end{equation*}  
is a globally embedded simple normal crossing pair. 
By \cite[Lemma 6.3.13]{fujino-foundations}, 
we can take $0< c<1$ such that 
\begin{itemize}
\item[(a)] we have 
\begin{equation*} 
f''\left(\left(B_{Y''}+c(f'')^*\pi^*\sum _{i=1}^{n+1}B_i\right)^{>1}\right)
\cap \pi^{-1}(P)=\emptyset
\end{equation*}   
or 
\begin{equation*} 
\dim \left(f''\left(\left(B_{Y''}+c(f'')^*\pi^*\sum _{i=1}^{n+1}
B_i\right)^{>1}\right)
\cap \pi^{-1}(P)\right)=0, 
\end{equation*} 
and 
\item[(b)] the following inequality 
\begin{equation*} 
\dim \left(f''\left(\left(B_{Y''}+c(f'')^*\pi^*\sum _{i=1}^{n+1}
B_i\right)^{=1}\right)
\cap \pi^{-1}(P)\right)\geq 1
\end{equation*}  
holds. 
\end{itemize}
By Lemma \ref{d-lem4.23}, we obtain that  
\begin{equation*} 
\left(X, \omega+\pi^*B, f''\colon  (Y'', B_{Y''}+(f'')^*\pi^*B)\to X\right), 
\end{equation*}  
where $B=c\sum _{i=1}^{n+1}B_i$, 
is a quasi-log scheme. 
By construction, we see that (i) holds true and 
$\Nqklt(X, \omega+\pi^*B)$ coincides with 
$\Nqklt(X, \omega)$ outside $\pi^{-1}(P)$ since 
$B_1, \ldots, B_{n+1}$ are general 
very ample Cartier divisors on $S$. 
Hence we have (ii). 
Therefore, (iii) and (iv) follow from (a) and (b), respectively. 

From now on, we further assume that 
$-\omega$ is $\pi$-ample. 
As usual, we put 
\begin{equation*} 
W'=W\cup \Nqlc(X, \omega+\pi^*B). 
\end{equation*} 
By adjunction (see Theorem \ref{d-thm4.6} (i) 
and \cite[Theorem 6.3.5 (i)]{fujino-foundations}), 
$[W', (\omega+\pi^*B)|_{W'}]$ is a quasi-log scheme. 
By Lemma \ref{d-lem4.20}, 
$[W, (\omega+\pi^*B)|_W]$ naturally becomes a quasi-log 
scheme. 
We note that $(\pi^*B)|_W\sim _{\mathbb R}0$ since 
$\pi(W)=P$. 
Therefore, by replacing $(\omega+\pi^*B)|_W$ with 
$\omega|_W$, we see that 
$[W, \omega|_W]$ is a quasi-log scheme. 
By Theorem \ref{a-thm1.9}, 
$[W^\nu, \nu^*\omega]$ becomes 
a quasi-log scheme. 
Note that $-\nu^*\omega$ is ample since 
$\pi\circ \nu(W^\nu)=P$. 
\begin{claim-n} 
We have 
\begin{equation*}
\begin{split}
&\Nqklt(X, \omega)\cap \pi^{-1}(P)\\ &\subset 
\Nqklt(W', (\omega+\pi^*B)|_{W'})\cap \pi^{-1}(P) 
\\&=\Nqklt(W, (\omega+\pi^*B)|_W)
\\&=\Nqklt 
(W, \omega|_W)
\end{split}
\end{equation*}
set theoretically. 
\end{claim-n}

\begin{proof}[Proof of Claim]
We divide the proof into several steps. 
\setcounter{step}{0}
\begin{step}\label{n-claim-step14.5.1} 
By (iii) and Lemma \ref{m-lem13.2}, 
$\Nqlc(X, \omega+\pi^*B)\cap \pi^{-1}(P)$ is empty or a 
point. 
By \cite[Theorem 6.8.3 (i)]{fujino-foundations}, 
every qlc center of $[X, \omega+\pi^*B]$ in $\pi^{-1}(P)$ 
contains $\Nqlc(X, \omega+\pi^*B)\cap \pi^{-1}(P)$ when 
$\Nqlc(X, \omega+\pi^*B)\cap \pi^{-1}(P)\ne \emptyset$. 
When $\Nqlc(X, \omega+\pi^*B)\cap \pi^{-1}(P)=\emptyset$, 
the set of all qlc centers intersecting $\pi^{-1}(P)$ 
has a unique minimal element with respect to the inclusion 
by \cite[Theorem 6.8.3 (ii)]{fujino-foundations}. 
\end{step}

\begin{step}\label{n-claim-step14.5.2} 
In this step, we will check that 
\begin{equation*} 
\Nqklt(X, \omega)\cap \pi^{-1}(P) \subset 
\Nqklt(W', (\omega+\pi^*B)|_{W'})\cap 
\pi^{-1}(P)
\end{equation*} 
holds set theoretically.

If $\Nqklt(X, \omega)\cap \pi^{-1}(P)=\emptyset$, then it is 
obvious. Hence we may assume that 
$\Nqklt(X, \omega)\cap \pi^{-1}(P)\ne \emptyset$. By assumption, 
$\Nqklt(X, \omega)\cap \pi^{-1}(P)$ is zero-dimensional. 
We take $Q\in \Nqklt(X, \omega)\cap \pi^{-1}(P)$. If $Q$ 
is a qlc center of $[X, \omega]$ or $Q\in \Nqlc (X, \omega)$, 
then $Q\in \Nqlc(X, \omega+\pi^*B)$ by the construction of the quasi-log 
scheme structure of $[X, \omega+\pi^*B]$. 
Then we have 
\begin{equation*} 
Q\in \Nqlc(W', (\omega+\pi^*B)|_{W'})\subset 
\Nqklt(W', (\omega+\pi^*B)|_{W'}). 
\end{equation*}  
Therefore, we have 
\begin{equation*} 
Q\in \Nqklt(W', (\omega+\pi^*B)|_{W'})\cap \pi^{-1}(P). 
\end{equation*}  
From now on, we assume that 
$Q$ is not a qlc center of $[X, \omega]$ and that $Q\not\in \Nqlc(X, \omega)$. 
Then, there exists a positive-dimensional qlc center $V$ of $[X, \omega]$ such 
that $Q\in V\cap \pi^{-1}(P)$. 
Since $\Nqklt(X, \omega)=\Nqklt(X, \omega+\pi^*B)$ holds 
outside $\pi^{-1}(P)$ (see (ii)), 
$V$ is also a qlc center of $[X, \omega+\pi^*B]$. 
If $\Nqlc(X, \omega+\pi^*B)\cap \pi^{-1}(P)\ne 
\emptyset$, then $\Nqlc(X, \omega+\pi^*B)\cap \pi^{-1}(P)$ 
is a point by (iii) and Lemma \ref{m-lem13.2}. 
In this case, by \cite[Theorem 6.8.3 (i)]{fujino-foundations}, 
we have $Q\in V\cap \pi^{-1}(P)\cap \Nqlc(X, \omega+\pi^*B)$. 
Hence $Q\in \Nqlc(W', (\omega+\pi^*B)|_{W'})\cap 
\pi^{-1}(P)$. 
This implies that $Q\in \Nqklt(W', (\omega+\pi^*B)|_{W'})\cap 
\pi^{-1}(P)$. 
Thus we further assume that 
$\Nqlc(X, \omega+\pi^*B)\cap \pi^{-1}(P)=\emptyset$. 
By shrinking $X$ around $\pi^{-1}(P)$, we may assume that 
$[X, \omega+\pi^*B]$ is quasi-log canonical. 
Then $Q\in V\cap W\cap \pi^{-1}(P)$ by Step \ref{n-claim-step14.5.1} 
(see also \cite[Theorem 6.8.3 (ii)]{fujino-foundations}). 
Hence $Q\in \Nqklt(W', (\omega+\pi^*B)|_{W'})\cap \pi^{-1}(P)$ 
by \cite[Theorem 6.3.11 (i)]{fujino-foundations}. 
More precisely, $Q$ is a qlc center of 
$[W', (\omega+\pi^*B)|_{W'}]$. 

In any case, we obtain 
\begin{equation*} 
\Nqklt(X, \omega)\cap \pi^{-1}(P) \subset 
\Nqklt(W', (\omega+\pi^*B)|_{W'})\cap 
\pi^{-1}(P)
\end{equation*} 
set theoretically.
\end{step}

\begin{step}\label{n-claim-step14.5.3}
By Step \ref{n-claim-step14.5.1} and 
Lemma \ref{d-lem4.20}, 
\begin{equation*} 
\Nqklt(W', (\omega+\pi^*B)|_{W'})\cap \pi^{-1}(P)=
\Nqklt(W, (\omega+\pi^*B)|_W)
\end{equation*}  
holds set theoretically. By the definition of 
the quasi-log scheme structure of $[W, \omega|_W]$, 
\begin{equation*} 
\Nqklt(W, (\omega+\pi^*B)|_W)=\Nqklt(W, \omega|_W) 
\end{equation*} 
obviously holds. 
\end{step} 
We finish the proof of Lemma \ref{n-lem14.5}. 
\end{proof}
Hence by Claim 
\begin{equation*} 
\nu^{-1}\left(\Nqklt(X, \omega)\cap \pi^{-1}(P)\right)\subset 
\Nqklt(W^\nu, \nu^*\omega)
\end{equation*}  
holds set theoretically since $\nu^{-1}\left(\Nqklt(W, \omega|_W)\right)
=\Nqklt(W^\nu, \nu^*\omega)$ by 
Theorem \ref{a-thm1.9}.
\end{proof}

Let us prove Theorem \ref{n-thm14.2}, which is 
stronger than Theorem \ref{a-thm1.16}. 

\begin{proof}[Proof of Theorem \ref{n-thm14.2}] 
We first use the reduction as in Steps 
\ref{j-step1.6.2}, \ref{j-step1.6.3}, and \ref{j-step1.6.4} in 
the proof of Theorem \ref{a-thm1.6}. 
Let us explain it more precisely for the reader's convenience. 
\setcounter{step}{0}
\begin{step}\label{n-step14.2.1} 
We take an irreducible component $W$ of $X$ such that 
$C^\dag\subset W$. 
We put $X'=W\cup \Nqlc(X, \omega)$. 
Then, by adjunction (see Theorem \ref{d-thm4.6} (i) and 
\cite[Theorem 6.3.5 (i)]{fujino-foundations}), 
$[X', \omega'=\omega|_{X'}]$ is a quasi-log 
scheme. 
By replacing $[X, \omega]$ with $[X', \omega']$, 
we may assume that $X\setminus X_{-\infty}$ 
is irreducible. By Lemma \ref{d-lem4.20}, 
we may replace $X$ with $\overline {X\setminus X_{-\infty}}$ and 
assume that $X$ is a variety. 
Then, by taking the normalization, we may further assume that 
$X$ is a normal variety (see Theorem \ref{a-thm1.9}). 
\end{step}
\begin{step}\label{n-step14.2.2} 
By taking the Stein factorization, we may further assume that 
$\pi_*\mathcal O_X\simeq \mathcal O_S$. 
We put $\Sigma=\Nqklt(X, \omega)$. 
It is sufficient to find a non-constant morphism 
\begin{equation*} 
f\colon \mathbb A^1\longrightarrow (X\setminus \Sigma)\cap 
\pi^{-1}(P)
\end{equation*}  
such that 
the curve $C$, the closure of $f(\mathbb A^1)$ in $X$, 
is a (possibly singular) rational curve satisfying 
$C\cap \Sigma\ne \emptyset$ with 
\begin{equation*} 
0<-\omega\cdot C\leq 1. 
\end{equation*}  
Without loss of generality, we may assume that 
$X$ and $S$ are quasi-projective by shrinking $S$ suitably.
\end{step}
\begin{step}\label{n-step14.2.3} 
By assumption, 
$\dim \pi^{-1}(P)>0$ and 
$\pi^{-1}(P)\cap \Sigma\ne \emptyset$. 
When $\dim S>0$, by 
Lemma \ref{n-lem14.5}, we take an effective 
$\mathbb R$-Cartier divisor $B$ on $S$ such that 
$[X, \omega+\pi^*B]$ is a quasi-log scheme satisfying 
the properties (i), (ii), (iii), and (iv) in Lemma \ref{n-lem14.5}. 
Then we take a positive-dimensional 
qlc center $W$ of $[X, \omega+\pi^*B]$ in 
$\pi^{-1}(P)$ such that 
there is no positive-dimensional qlc center $W^\dag
\subsetneq W$ of $[X, \omega+\pi^*B]$. 
By Lemma \ref{n-lem14.5}, $[W^\nu, \nu^*\omega]$ naturally 
becomes a quasi-log Fano scheme, where 
$\nu\colon W^\nu\to W$ is the normalization. 
When $\dim S=0$, it is sufficient to put $B=0$ and $W=X$. 
By construction, 
$\Nqklt(W^\nu, \nu^*\omega)$ is finite. 
On the other hand, 
$\Nqklt(W^\nu, \nu^*\omega)$ is connected 
(see Lemma \ref{m-lem13.2} and \cite[Theorem 6.8.3]{fujino-foundations}). 
By Lemma \ref{n-lem14.5}, we obtain 
\begin{equation*} 
\emptyset\ne \nu^{-1}\left(\Sigma \cap \pi^{-1}(P)\right) 
\subset \Nqklt(W^\nu, \nu^*\omega). 
\end{equation*}  
Hence $\Nqklt(W^\nu, \nu^*\omega)$ is a point such that 
$\nu^{-1}\left(\Sigma\cap \pi^{-1}(P)\right)=
\Nqklt(W^\nu, \nu^*\omega)$ holds 
set theoretically. 
By applying Theorems \ref{a-thm1.10} and 
\ref{n-thm14.4} to $[W^\nu, \nu^*\omega]$ as in 
Step \ref{j-step1.6.5} in the proof of Theorem \ref{a-thm1.6}, 
we obtain a non-constant 
morphism 
\begin{equation*} 
h\colon  \mathbb A^1\longrightarrow W^\nu \setminus 
\Nqklt(W^\nu, \nu^*\omega)
\end{equation*} 
such that 
$C'$, the closure of $h(\mathbb A^1)$ in $W^\nu$, is 
a (possibly singular) rational curve 
satisfying $C'\cap \Nqklt(W^\nu, \nu^*\omega)\ne \emptyset$ with 
$0<-\nu^*\omega\cdot C'\leq 1$. 
Then 
\begin{equation*} 
f:=\iota\circ \nu\circ h\colon 
\mathbb A^1\longrightarrow (X\setminus \Sigma)
\cap \pi^{-1}(P), 
\end{equation*} 
where $\iota\colon  W\hookrightarrow X$ is a natural 
inclusion, is a desired morphism. 
\end{step}
We finish the proof of Theorem \ref{n-thm14.2}. 
\end{proof}

For the proof of Theorem \ref{a-thm1.20}, we prepare the 
following theorem. The proof of Theorem \ref{n-thm14.6} 
uses a deep result on the existence problem of minimal models 
in \cite{hashizume}. 

\begin{thm}\label{n-thm14.6} 
Let $(X, \Delta)$ be a dlt pair and let $\pi\colon X\to S$ be 
a projective morphism between normal varieties 
such that 
$-(K_X+\Delta)$ is $\pi$-ample. 
We assume that $\Nklt(X, \Delta)$ is not empty such that 
$
\pi\colon \Nklt(X, \Delta)\to \pi(\Nklt(X, \Delta))
$ 
is finite and that 
there exists a curve $C^\dag$ on $X$ such that 
$\pi(C^\dag)$ is a point with $C^\dag\cap 
\Nklt(X, \Delta)\ne \emptyset$. 
Then there exists a non-constant morphism 
\begin{equation*} 
f\colon  \mathbb A^1\longrightarrow X\setminus \Nklt(X, \Delta)
\end{equation*}  
such that $\pi\circ f(\mathbb A^1)$ is a point 
and that 
the curve $C$, the closure of $f(\mathbb A^1)$ in $X$, 
is a {\em{(}}possibly singular{\em{)}} rational curve 
satisfying  
$C\cap \Nklt(X, \Delta)\ne \emptyset$ with  
\begin{equation*} 
0<-(K_X+\Delta)\cdot C\leq 1. 
\end{equation*} 
\end{thm}

\begin{proof} 
By shrinking $S$ suitably, we may assume that 
$X$ and $S$ are both quasi-projective. 
By Lemma \ref{c-lem3.10}, we can construct a projective 
birational morphism $g\colon  Y\to X$ from a normal 
$\mathbb Q$-factorial variety satisfying (i), (ii), and (iv) in Lemma 
\ref{c-lem3.10}. Since $K_Y+\Delta_Y=g^*(K_X+\Delta)$, 
$(K_Y+\Delta_Y)|_{\Nklt(Y, \Delta_Y)}$ is nef 
over $S$ by Lemma \ref{c-lem3.10} (iv). 
Let us consider $\pi\circ g\colon  Y\to S$. 
By construction, $(Y, \Delta_Y)$ is a $\mathbb Q$-factorial 
dlt pair. 

If $\dim S=\dim X$, then $K_Y+\Delta_Y$ is pseudo-effective 
over $S$. 
In this case, we can take an effective $\mathbb R$-divisor 
$A$ on $Y$ such that 
$
K_Y+\Delta_Y+A\sim _{\mathbb R, \pi\circ g}0
$ 
and that $(Y, \Delta_Y+A)$ is dlt 
since $-(K_Y+\Delta_Y)=-g^*(K_X+\Delta)$ is $(\pi\circ g)$-semi-ample. 
Hence $(Y, \Delta_Y)$ has a good minimal model 
over $S$ by \cite[Theorem 1.1]{hashizume} and 
any $(K_Y+\Delta_Y)$-minimal model program 
over $S$ with scaling of an ample 
divisor 
terminates (see \cite[Theorem 2.11]{hashizume}). 

If $\dim S<\dim X$, then $K_Y+\Delta_Y$ is 
not pseudo-effective over $S$ since $-(K_X+\Delta)$ is ample 
over $S$. 
In this case, we have a $(K_Y+\Delta_Y)$-minimal model 
program 
which terminates at a Mori fiber space by \cite{bchm}. 

Therefore, we have a finite 
sequence of flips and divisorial contractions 
\begin{equation*} 
\xymatrix{
Y=:Y_0\ar@{-->}[r]^-{\phi_0}
& Y_1\ar@{-->}[r]^-{\phi_1} & \cdots \ar@{-->}[r]^-{\phi_{i-1}}
& Y_i \ar@{-->}[r]^-{\phi_i}&\cdots \ar@{-->}[r]^-{\phi_{k-1}}& Y_k
} 
\end{equation*} 
starting from $(Y_0, \Delta_{Y_0}):=(Y, \Delta_Y)$ 
over $S$ such that $(Y_k, \Delta_{Y_k})$ is a good minimal model 
of $(Y, \Delta_Y)$ over $S$ or 
$p\colon Y_k\to Z$ is a Mori fiber space with 
respect to $K_{Y_k}+\Delta_{Y_k}$ over $S$,  
where 
$\Delta_{Y_{i+1}}={\phi_i}_*\Delta_{Y_i}$ for 
every $i$. By assumption, we can take a curve $C'$ on $Y$ such that 
$-(K_Y+\Delta_Y)\cdot C'>0$ with 
$C'\cap \Nklt(Y, \Delta_Y)\ne\emptyset$. 
If $\phi_i$ is an isomorphism 
in a neighborhood of 
$\Nklt(Y_i, \Delta_{Y_i})$ for $0\leq i<l$, 
then 
\begin{equation}\label{n-eq14.1}
0<-(K_Y+\Delta_Y)\cdot C'\leq -(K_{Y_l}+\Delta_{Y_l})\cdot 
C'_{Y_l}
\end{equation} 
holds by the negativity lemma, where $C'_{Y_l}$ is the strict transform of 
$C'$ on $Y_l$. 
\setcounter{case}{0}
\begin{case}\label{n-case14.6.1}
We assume that $\phi_i$ is an isomorphism in a neighborhood of 
$\Nklt(Y_i, \Delta_{Y_i})$ for every $i$. 
Then, by \eqref{n-eq14.1}, the final model $Y_k$ has a Mori fiber 
space structure $p\colon Y_k\to Z$ over $S$. 
We note that 
$(K_{Y_k}+\Delta_{Y_k})|_{\Nklt(Y_k, \Delta_{Y_k})}$ is nef over 
$S$. 
Hence the argument in Case \ref{n-case14.3.1} 
in the proof of Theorem \ref{n-thm14.3} works without any changes. 
Then we get a non-constant morphism 
\begin{equation*} 
f\colon \mathbb A^1\longrightarrow X\setminus \Nklt(X, \Delta) 
\end{equation*}  
with the desired properties. 
\end{case}
\begin{case}\label{n-case14.6.2}
By Case \ref{n-case14.6.1}, we may assume that there exists $i_0$ such that 
$\phi_i$ is an isomorphism in a neighborhood of 
$\Nklt(Y_i, \Delta_{Y_i})$ for 
$0\leq i<i_0$ and $\phi_{i_0}$ is not an isomorphism 
in a neighborhood of $\Nklt(Y_{i_0}, \Delta_{Y_{i_0}})$. 
The argument in Case \ref{n-case14.3.2} in the proof of 
Theorem \ref{n-thm14.3} works without any changes. 
Then we get 
a non-constant morphism 
\begin{equation*} 
f\colon  \mathbb A^1\longrightarrow X\setminus \Nklt(X, \Delta) 
\end{equation*}  
with the desired properties. 
\end{case}
We finish the proof of Theorem \ref{n-thm14.6}. 
\end{proof}

We close this section with the proof of Theorem \ref{a-thm1.20}. 
Since adjunction works well for dlt pairs, 
Theorem \ref{a-thm1.20} 
directly follows from Theorem \ref{n-thm14.6}. 

\begin{proof}[Proof of Theorem \ref{a-thm1.20}] 
We put $W=\overline {U_j}$. Then $W$ is an lc 
stratum of $(X, \Delta)$. 
By adjunction, it is well known that we have 
\begin{equation*} 
(K_X+\Delta)|_W=K_W+\Delta_W
\end{equation*}  
such that $(W, \Delta_W)$ is dlt and that the lc centers of 
$(W, \Delta_W)$ are exactly the lc centers of $(X, \Delta)$ that 
are strictly included in $W$ (see, 
for example, \cite[Proposition 3.9.2]{fujino-what}). 
By replacing $\pi\colon X\to S$ with the Stein factorization of 
$\varphi_{R_j}\colon  W\to \varphi_{R_j}(W)$, 
we may assume that $\pi\colon \Nklt(X, \Delta)\to 
\pi(\Nklt(X, \Delta))$ is finite and 
that there exists a curve $C^\dag$ on $X$ such that 
$\pi(C^\dag)$ is a point with $C^\dag\cap \Nklt(X, \Delta)
\ne \emptyset$. 
By Theorem \ref{n-thm14.6}, we obtain a desired non-constant morphism 
\begin{equation*} 
f\colon  \mathbb A^1\to X\setminus \Nklt(X, \Delta)
\end{equation*}  
with the desired properties. 
\end{proof}

As we have already mentioned, we will completely prove 
Conjecture \ref{a-conj1.15} in a joint paper 
with Kenta Hashizume (see \cite{fujino-hashizume1}), 
where we use some deep results on the minimal 
model program for log canonical pairs. 
We strongly recommend the interested reader to 
see \cite{fujino-hashizume1}. 

\end{document}